\documentclass[twoside]{irmaems}
\usepackage{amssymb} 
\usepackage{amsmath} 
\usepackage{latexsym}
\usepackage{graphicx}
\usepackage{amscd}
\usepackage[all]{xy}

\renewcommand\theequation{\arabic{section}.\arabic{equation}}

\theoremstyle{definition} 

 \newtheorem{definition}{Definition}[section]
 \newtheorem{remark}[definition]{Remark}
 \newtheorem{example}[definition]{Example}
 \newtheorem{problem}[definition]{Problem}
 \newtheorem{problems}[definition]{Problems}

\newtheorem*{notation}{Notation}  
\newtheorem*{conventions}{Convention}

\theoremstyle{plain}      

 \newtheorem{proposition}[definition]{Proposition}
 \newtheorem{theorem}[definition]{Theorem}
 \newtheorem{corollary}[definition]{Corollary}
 \newtheorem{lemma}[definition]{Lemma}

\newcommand{\Hom}{\operatorname{Hom}}

\newcommand{\End}{\operatorname{End}}

\newcommand{\Z}{\ensuremath{\mathbb{Z}}}
\newcommand{\R}{\ensuremath{\mathbb{R}}}
\newcommand{\Tmatrix}[1]{\mathop{\left( {#1} \right)}\nolimits}
\newcommand{\Aut}{\mathrm{Aut}\,}
\newcommand{\Out}{\mathrm{Out}\,}
\newcommand{\Ker}{\mathop{\mathrm{Ker}}\nolimits}
\newcommand{\hsymb}[1]{\mbox{\strut\rlap{\smash{\Huge$#1$}}\quad}}

\newcommand{\Acy}{F^{\mathrm{acy}}}

\newcommand{\Sgb}{\ensuremath{\Sigma_{g,1}}}
\newcommand{\Mgb}{\ensuremath{\mathcal{M}_{g,1}}}
\newcommand{\Cgb}{\ensuremath{\mathcal{C}_{g,1}}}
\newcommand{\Hgb}{\ensuremath{\mathcal{H}_{g,1}}}
\newcommand{\Igb}{\ensuremath{\mathcal{I}_{g,1}}}

\newcommand{\Mg}[2]{\mathcal{M}_{#1,#2}}
\newcommand{\Cg}[2]{\mathcal{C}_{#1,#2}}
\newcommand{\Hg}[2]{\mathcal{H}_{#1,#2}}
\newcommand{\Sg}[2]{\Sigma_{#1,#2}}

\newcommand{\Nil}[1]{#1^{\mathrm{nil}}}
\newcommand{\AC}[1]{#1^{\mathrm{acy}}}
\newcommand{\Id}{\mathop{\mathrm{id}}\nolimits}

\newcommand{\dis}{\displaystyle}
\newcommand{\deriv}[2]{\frac{\partial #1}{\partial #2}}

\begin{document}

\title{A survey of Magnus representations for mapping class groups 
and homology cobordisms of surfaces}

\author{Takuya Sakasai\thanks{
Work partially supported by 
Grant-in-Aid for Scientific Research, 
(No.~21740044), 
Ministry of Education, Science, 
Sports and Technology, Japan.}}

\address{Department of Mathematics, Tokyo Institute of Technology,\\
2-12-1 Oh-okayama, Meguro-ku, Tokyo, 152-8551, Japan.\\
email:\,\tt{sakasai@math.titech.ac.jp}
}

\maketitle


\begin{abstract} 
This is a survey of Magnus representations 
with particular emphasis on their applications to 
mapping class groups and monoids (groups) of 
homology cobordisms of surfaces. 
In the first half, we begin by recalling 
the basics of the Fox calculus and 
overview Magnus representations for automorphism groups of 
free groups and mapping class groups of surfaces with 
related topics. 
In the latter half, we discuss in detail 
how the theory in the first half 
extends to homology cobordisms of surfaces and present 
a number of applications from recent researches. 
\end{abstract}



\tableofcontents  

\section{Introduction}\label{sec:intro}
Let $\Sg{g}{n}$ be a compact connected oriented smooth surface of genus $g$ 
with $n$ boundary components ($n$ may be $0$). 
We take a base point $p$ in $\partial \Sg{g}{n}$ when $n \ge 1$ and 
arbitrarily when $n=0$. 
The {\it mapping class group} 
$\mathcal{M}_{g,n}$ of $\Sigma_{g,n}$ 
is defined as the group of all isotopy classes of orientation-preserving 
diffeomorphisms of $\Sigma_{g,n}$ which fix the boundary pointwise. 
The area of research covered by the mapping class group 
in contemporary mathematics is very broad: 
algebraic geometry, differential geometry, 
hyperbolic geometry, complex analysis, topology, 
combinatorial group theory, mathematical physics etc. 
The group $\Mg{g}{n}$ serves as 
the modular group of the Teichm\"uller space, 
which is the main theme of this handbook, 
and this yields a strong 
connection among the above subjects. 

Now let us consider $\Mg{g}{n}$ 
from a topological point of view, 
which is our approach for studying $\Mg{g}{n}$ in this 
chapter. 
While $\Mg{g}{n}$ does not act on the surface $\Sg{g}{n}$ 
itself, it works as transformation groups of many 
discretized objects related 
to $\Sg{g}{n}$: 
the set of isotopy classes of curves, 
the fundamental group $\pi_1 (\Sg{g}{n},p)$ (when $n \ge 1$), 
the homology groups etc. 
In many cases, these discretized (simplified in a sense) data 
lose no topological information on the objects involved, 
for the classical theory of surface topology says that 
homotopical information of a surface governs 
its topology. 
For example, we use $\Mg{g}{n}$ more often 
than the diffeomorphism group of $\Sg{g}{n}$ 
in the construction of 
three-dimensional manifolds by 
Heegaard splittings or Dehn surgery, 
and also that of surface bundles over a manifold. 
When $n = 1$, 
a theorem of Dehn and Nielsen says that 
an element in $\Mg{g}{1}$ 
is completely characterized by its action on $\pi_1 (\Sg{g}{1},p)$, 
which is a free group of rank $2g$. 
This suggests a method for studying $\Mg{g}{n}$ through 
the theory of 
the {\it automorphism group of a free group}. 

{\it Magnus representations} are matrix representations for 
free groups and their automorphism groups. 
The definition is usually given in terms of 
the {\it Fox derivative} and it looks like a Jacobian matrix of 
a differentiable map. 
With their relationship to homology and cohomology of groups, 
Magnus representations have been used as a fundamental tool 
for studying various groups 
by researchers in combinatorial group theory for many years. 
En route, applications to $\Mg{g}{n}$ have also been given. 

Compared with the history of mapping class groups and 
automorphism groups of free groups,  
the study of monoids and groups of 
{\it homology cobordisms} over a surface is 
a quite new theme of research. 
This research was independently 
initiated by Goussarov \cite{gou} and Habiro \cite{habiro} 
in their investigations of three-dimensional manifolds 
via so-called {\it clover} or {\it clasper} surgery, 
which are known to be essentially the same. 
These surgery techniques are said to be a ``topological 
commutator calculus'', which invokes a connection to $\Mg{g}{n}$ 
as automorphisms of $\pi_1 (\Sg{g}{n},p)$. 
In fact, Garoufalidis-Levine \cite{gl} established a 
connection between the above surgery techniques and 
classical algebraic topology related to $\Mg{g}{n}$ in terms of 
Massey products on the first cohomology of three-dimensional manifolds. 

In this chapter, we use the word ``{\it homology cylinders}'' for 
homology cobordisms over a surface with fixed markings of their boundaries. 
These markings are necessary to define a product operation 
on the set of 
homology cylinders as a generalization of the mapping class group. 
We can consider, for example, 
homology 3-spheres and pure string links 
with $n$ strings to be homology cylinders 
over $\Sigma_{0,1}$ and $\Sigma_{0,n+1}$. 
For a given homology cylinder over $\Sigma_{g,n}$, 
we can construct another one by changing its markings by 
using $\Mg{g}{n}$. Therefore homology cylinders 
enable us to treat important objects 
in three-dimensional topology simultaneously. 
This motivates the study of homology cylinders 
with particular stress on its algebraic structure. 

The purpose of this chapter is to survey research on 
structures of mapping class groups and monoids (groups) 
of homology cylinders through 
their Magnus representations as a common tool for study. 
Note that homology cylinders are also discussed in 
the chapter of Habiro and Massuyeau \cite{habiromassuyeau}. 
However, our approach here is 
distinct from theirs and more group-theoretical. 
The author hopes that their chapter and the present one 
could complement each other and offer the readers an introduction to 
this fruitful subject. 

Here we briefly mention the content of this chapter. 
Basically, it is divided into two parts: 
The first part is intended for serving 
as a survey of Magnus representations 
for automorphism groups of free groups and 
mapping class groups of surfaces. 
In Section \ref{sec:fox}, we first recall 
the Fox calculus as a tool for many computations 
in this chapter. We put stress on its relationship to 
homology and cohomology of groups. 
Section \ref{sec:Magnus1} is devoted to give a machinery 
of Magnus representations with relations to 
automorphisms of the derived quotients of 
a free group. Finally, in Section \ref{sec:MagnusMCG}, 
we overview 
applications of Magnus representations to mapping class 
groups of surfaces following works of Morita and Suzuki. 
The second part of this chapter begins in Section \ref{sec:HC}, where 
the monoid and related groups of homology cylinders over a surface 
are introduced. 
Sections \ref{sec:universal} and \ref{sec:extension} 
are devoted to discussing methods for extending 
Magnus representations for mapping class groups to 
homology cylinders. 
In Section \ref{sec:application}, 
we present several topics on homology cylinders 
where Magnus representations play important roles.  

\begin{conventions}
All maps act on elements from the {\it left}. 
We often use the same notation to write a map and 
induced maps on quotients of its source or target. 
Homology and cohomology groups are assumed to be with 
coefficients in the ring of integers $\mathbb{Z}$ 
and all manifolds are assumed to be smooth 
unless otherwise indicated. 
\end{conventions}

The author is deeply grateful to Athanase Papadopoulos for 
giving him a chance to write this chapter and providing 
many valuable suggestions. 
The author also would like to thank Kazuo Habiro, 
Nariya Kawazumi, Teruaki Kitano, Gw\'ena\"el Massuyeau, 
Takayuki Morifuji, Takao Satoh and Masaaki Suzuki 
for their careful reading and helpful comments on the manuscript, 
and Hiroshi Goda for 
permitting the author to use the pictures 
in Example \ref{ex:concordance}.

\section{Fox calculus}\label{sec:fox}
We begin by recalling the Fox calculus, which 
was defined by Fox in the 1950s and which has been 
known as an important tool in the study of 
free groups and their automorphisms. 
Magnus representations are defined 
as an application of this machinery. 
Historically, 
Magnus representations were defined 
{\it without} the Fox calculus. 
However, we use the Fox calculus since it 
is now widely accepted to be standard and offers 
a clear connection to low-dimensional topology. 
Good references for this topic have been the original paper of 
Fox \cite{fox1} and Birman's book \cite{bi}. 
Our discussion is almost parallel to the latter 
with particular stress on the relationship of 
the Fox calculus to homology and cohomology of groups. 
The relation becomes a key to generalizing the machinery 
so that it can be applied to objects in more broad range than 
free groups and their automorphisms in the latter half of this chapter. 

Since we shall work in {\it non-commutative} rings 
almost everywhere in this chapter, 
we first need to fix our notation in detail. 

\begin{notation}
For a group $G$ and two of its subgroups $G_1$ and $G_2$, 
we denote by $[G_1,G_2]$ the commutator subgroup of $G_1$ and $G_2$. 
We set $[x,y]= x y x^{-1} y^{-1}$ for $x$, $y \in G$. 
The integral (or rational) group ring of $G$ is denoted 
by $\Z [G]$ (or $\mathbb{Q} [G]$). 
For a matrix $A$ with entries in a ring $R$ 
and a ring homomorphism $\varphi:R \to R'$, 
we denote by ${}^{\varphi} A$ the 
matrix obtained from $A$ by applying $\varphi$ to each entry. 
$A^T$ denotes the transpose of $A$. 
When $R$ is $\Z [G]$ 
or its right field of fractions 
(if it exists), we denote by $\overline{A}$ 
the matrix obtained from $A$ by applying the involution induced 
from $(x \mapsto x^{-1},\ x \in G)$ to each entry. 
For a module $M$, we write $M^n$ 
the module of column vectors with $n$ entries in $M$. 
\end{notation}

\subsection{Fox derivatives}\label{subsec:fox}

Let $G$ be a group and let $M$ be a left $\Z [G]$-module. 
A {\it crossed homomorphism} \index{crossed homomorphism} 
(or {\it derivation}) 
from $G$ to $M$ is a map $f: G \to M$ satisfying
\[f(xy)=f(x)+x f(y)\]
for all $x,y \in G$. 
In other words, it 
is a homomorphism $(f,\mathrm{id}_G): G \to M \rtimes G$, where 
$M \rtimes G$ is the semi-direct product with 
the group structure given by 
\[(m_1,g_1) \cdot (m_2, g_2) = (m_1 + g_1 m_2 ,g_1 g_2)\]
for $m_1,m_2 \in M$ and $g_1, g_2 \in G$. 
When $G$ is $F_n$, a free group of rank $n$, 
the latter description shows that 
a crossed homomorphism $f:F_n \to M$ is 
determined by its values on any generating set of $F_n$. 
In general, 
if $G$ has a (not necessarily finite) presentation 
$\langle x_1, x_2, \ldots \mid r_1,r_2,\ldots \rangle$, 
crossed homomorphisms $f: G \to M$ correspond to 
crossed homomorphisms 
$f: \langle x_1, x_2, \ldots \rangle \to M$ satisfying 
$f(r_i)=0$ for all $i=1,2,\ldots$, where 
$\langle x_1, x_2, \ldots \rangle$ is the free group generated 
by $\{x_1, x_2, \ldots\}$. 

Let us define Fox derivatives for $F_n$. 
We take a basis $\{ \gamma_1, \gamma_2, \ldots, \gamma_n\}$ of 
$F_n$. The group $F_n$ acts on $\Z [F_n]$ by left multiplication, 
so that $\Z [F_n]$ is a left $\Z [F_n]$-module. 

\begin{definition}\label{def:Fox}
The {\it Fox derivative} \index{Fox derivative} 
(or {\it free derivative}) 
with respect to $\gamma_j$ in the basis 
$\{ \gamma_1, \gamma_2, \ldots, \gamma_n\}$ 
is the crossed homomorphism 
\[\deriv{}{\gamma_j} : F_n \longrightarrow \Z [F_n]\]
defined by $\dis\deriv{\gamma_i}{\gamma_j} = \delta_{ij}$ 
(Kronecker's delta). 
We use the same notation for its extension 
\[\deriv{}{\gamma_j} : \Z [F_n] \longrightarrow \Z [F_n]\]
to $\Z [F_n]$ as an additive map. 
\end{definition}

Fundamental properties of Fox derivatives are as follows. 
Let $\mathfrak{t}:\Z [F_n] \to \Z$ be the trivializer 
(or augmentation homomorphism) defined by 
$\mathfrak{t} (\sum_{v \in F_n} a_v v)= \sum_{v \in F_n} a_v$. 
\begin{proposition}\label{prop:foxproperty}
$(1)$ \ The equality \ 
$\dis\deriv{\gamma_i^{-1}}{\gamma_j}=-\delta_{ij} \gamma_i^{-1}$ \ holds. 

\noindent
$(2)$ \ For $g, h \in \Z [F_n]$, we have 
\[\dis\deriv{(gh)}{\gamma_j}=\deriv{g}{\gamma_j} \mathfrak{t}(h) 
+g \deriv{h}{\gamma_j}.\]

\noindent
$(3)$ \ $($Chain rule\,$)$ Let $\varphi:F_n \to F_n$ be an 
endomorphism of $F_n$. 
For any $w \in F_n$, we have
\[\deriv{\varphi(w)}{\gamma_j} = \sum_{k=1}^n 
\left(\varphi\left(\deriv{w}{\gamma_k}\right)\right) 
\left(\deriv{\varphi(\gamma_k)}{\gamma_j}\right).\]

\noindent
$(4)$ \ $($``Fundamental formula'' of Fox calculus\,$)$ 
For $g \in \Z [F_n]$, we have 
\[g-\mathfrak{t}(g) = \sum_{j=1}^n \deriv{g}{\gamma_j} (\gamma_j-1).\]

\noindent
$(5)$ \ Let $\rho: F_n \to \Gamma$ be a homomorphism. Then 
$v \in F_n$ satisfies 
\[\rho \left(\deriv{v}{\gamma_j}\right) = 0 \qquad 
\mbox{for $j=1,2,\ldots,n$}\]
if and only if $v \in [\Ker \rho, \Ker \rho]$. 
\end{proposition}
\noindent
(1) and (2) are easily proved. As for (3), (4) and (5), 
we prove them in the next subsections 
by relating them to homology and cohomology of groups.

\subsection{The Magnus representation 
for a free group}\label{subsec:magnus_free}

Now we define the Magnus representation for $F_n$ 
by using Fox derivatives. 
Let $S=\{s_1,s_2, \ldots, s_n\}$ 
be a set of formal parameters. We denote 
by $(\Z [F_n])[S]$ 
the polynomial ring over $\Z [F_n]$ 
with variables $S$, 
where the elements of $S$ are supposed to 
commute with one another and with the elements of $\Z [F_n]$. 
\begin{definition}
For $w \in F_n$, we put a matrix
\[(w):= \begin{pmatrix} w & 
\dis\sum_{j=1}^n\left(\deriv{w}{\gamma_j}\right)s_j \\
0 & 1
\end{pmatrix}.\]
Then the map $w \mapsto (w)$ gives a homomorphism 
$F_n \to \mathrm{GL}(2,(\Z F_n) [S])$ 
called the {\it Magnus representation} 
\index{Magnus representation!for $F_n$} 
for $F_n$.
\end{definition}

This representation was first given by Magnus \cite{magnus} 
without using Fox derivatives (see Remark \ref{rem:fox}). 
It is clearly injective. 
On the other hand, it follows from 
Proposition \ref{prop:foxproperty} (5) that 
for a homomorphism $\rho:F_n \to \Gamma$, 
the kernel of the homomorphism $w \mapsto {}^{\rho} (w)$ is 
$[\Ker \rho, \Ker \rho]$. 
In particular, if we take the abelianization map 
$\mathfrak{a}: F_n \to H_1:=H_1 (F_n) \cong \Z^n$ as $\rho$, 
we obtain an injection of the {\it metabelian quotient} 
\index{metabelian quotient} 
$F_n/[[F_n,F_n],[F_n,F_n]]$ into 
$\mathrm{GL}(2,(\Z [H_1]) [S])$, where the definition of $(\Z [H_1])[S]$ 
is given by replacing $F_n$ with $H_1$ 
in the definition of $(\Z [F_n])[S]$. 

\begin{remark}\label{rem:fox}
The Magnus representation of $F_n$ can be 
described by using crossed homomorphisms. 
Indeed, we can unify the Fox derivatives 
$\dis\deriv{}{\gamma_j}$ $(j=1,2,\ldots,n)$ 
into a homomorphism 
\[\left(\left(\deriv{}{\gamma_1},\deriv{}{\gamma_2},\ldots,
\deriv{}{\gamma_n}\right)^T, \mathrm{id}_{F_n}\right): 
F_n \longrightarrow (\Z [F_n])^n \rtimes F_n.\]
This map is equivalent to the Magnus representation. 
On the other hand, if we consider a homomorphism 
$F_n \to (\Z [F_n])^n \rtimes F_n$ given by 
\[\gamma_j \longmapsto 
\bordermatrix{
&  j\ \ \,  & \cr
& (0, \ldots, 0,\, 1 \,, 0, \ldots, 0)^T, & \gamma_j}, \]
we can {\it define} the Fox derivatives as its first projection. 
This corresponds to Magnus' original description. 
\end{remark}

\subsection{Homology and cohomology of groups}\label{subsec:grouphomology}

In this subsection, 
we interpret the definition and fundamental properties of 
Fox derivatives and the Magnus representation for $F_n$ in terms of 
homology and cohomology. 
This interpretation makes it easier to give their 
topological applications. 
As space is limited, however, we 
refer to Brown's book \cite{br} for the general theory. 
Instead, here we give explicit chain and cochain complexes 
to calculate the homology and cohomology of a given group. 

Let $G$ be a group and $M$ be a left $\Z [G]$-module. 
We denote by $C_i$ the free $\Z [G]$-module generated by 
the symbols $[g_1|g_2|\cdots|g_i]$ corresponding to 
$i$-tuples of elements $g_1, g_2, \ldots, g_i$ of $G$ 
($C_0 \cong \Z [G]$ generated by $[\cdot]$). 
We define the chain complex $C_\ast (G;M)$ and 
cochain complex $C^\ast (G;M)$ by
\begin{align*}
C_i (G;M) &= C_i \otimes_{\Z [G]} M, \\
C^i (G;M) &= \Hom_{\Z [G]} (C_i,M).
\end{align*}
\noindent 
Here the tensor product 
is taken for the {\it right} $\Z [G]$-module $C_i$ and 
the left $\Z [G]$-module $M$, and 
$\Hom_{\Z [G]} (C_i,M)$ consists of 
$\Z [G]$-``equivariant'' homomorphisms $f$ in the sense that 
$f$ satisfies $f(c g)= g^{-1} f(c)$ 
for $c \in C_i$ and $g \in G$. 
(These conventions are slightly different 
from those in \cite{br}.) 
The boundary operator 
$\partial_i: C_i (G;M) \to C_{i-1}(G;M)$ is given by
\begin{align*}
& \ \ \partial_i([g_1|g_2|\cdots|g_i]\otimes m) \\
=&\ 
[g_2|g_3|\cdots|g_i] \otimes g_1^{-1} m -
[g_1 g_2| g_3 |\cdots|g_i] \otimes m \\
&+[g_1 |g_2 g_3|g_4|\cdots|g_i] - \cdots 
+(-1)^{i-1}[g_1|\cdots| g_{i-2} |g_{i-1}g_i] \otimes m \\
&+(-1)^i[g_1|\cdots| g_{i-2} |g_{i-1}] \otimes m 
\end{align*}
\noindent
and the coboundary operator 
$\delta_i: C^i (G;M) \to C^{i+1}(G;M)$ is given by
\begin{align*}
& \ \ (\delta_i f) ([g_1|g_2|\cdots | g_{i+1}])\\
=&\ g_1f([g_2|g_3|\cdots|g_{i+1}])
-f([g_1 g_2| g_3 |\cdots|g_{i+1}])\\
&+f([g_1 |g_2 g_3|g_4|\cdots|g_{i+1}]) - \cdots 
+(-1)^if([g_1|\cdots| g_{i-1} |g_ig_{i+1}]) \\
&+(-1)^{i+1}f([g_1|\cdots| g_{i-1} |g_i])
\end{align*}
for $f \in C^i (G;M)$ regarded as a function 
from the set of symbols $[g_1|g_2|\cdots|g_i]$ 
to $M$. We denote the corresponding homology and cohomology groups 
by $H_\ast (G;M)$ and $H^\ast (G;M)$. 
We obtain the following explicit description of homology and 
cohomology in degree $0$ by observing the complexes. 
\begin{align*}
H_0 (G;M)& 
= M/\langle m-gm \mid m \in M, g \in G\rangle,\\
H^0 (G;M)& 
= \{ m \in M \mid \mbox{$gm=m$ for any $g \in G$}\}.
\end{align*}
Next we take a close look at $H^1 (G;M)$. The condition for 
$f \in C^1(G;M)$ to be a cocycle is that 
\[0=(\delta_1f)([g_1|g_2])= g_1f([g_2])-f([g_1 g_2])+f([g_1])\]
holds for any $g_1, g_2 \in G$. If we naturally identify a cochain in 
$C^1 (G;M)$ with a map from $G$ to $M$, this cocycle condition is 
nothing more than the definition of crossed homomorphisms from $G$ to $M$ 
mentioned in Section \ref{subsec:fox}. 
Therefore the module 
of $1$-cocycles is 
written as the module $\mathrm{Cross}(G;M)$ of crossed homomorphisms 
from $G$ to $M$. A $1$-coboundary is 
obtained from each element 
$m \in M \cong C^0(G;M)$ 
corresponding to the function $[\cdot] \mapsto m$, 
and we have 
\[(\delta_0 m)([g])=g m -m\]
for $m \in M$ and $g \in G$. We call such a crossed homomorphism 
(after the above identification) a {\it principal crossed homomorphism} 
and denote the module of principal crossed homomorphisms by 
$\mathrm{Prin}(G;M)$. Consequently we have 
\[H^1 (G;M) \cong \mathrm{Cross}(G;M)/\mathrm{Prin}(G;M).\]
\begin{example}
Let $M=\Z$ with the trivial $G$-action. 
We have 
\begin{align*}
H_0 (G;\Z) &\cong H^0 (G;\Z) \cong \Z, \\
H_1 (G;\Z) &\cong G/[G,G], \ \mbox{the abelianization of $G$},\\
H^1 (G;\Z) &\cong \Hom (G,\Z) = \Hom(H_1(G;\Z),\Z).
\end{align*}
\end{example}
\begin{remark}
For any $\Z [G]$-bimodule $M$ 
(for example, $M=\Z [G]$), 
$C_\ast (G;M)$ and $C^\ast (G;M)$ have natural 
{\it right} actions of $G$, with our convention. Moreover, 
the operators $\partial_i$ and $\delta_i$ are 
equivariant with respect to this right action, 
so that $H_\ast (G;M)$ and $H^\ast (G;M)$ become 
right $\Z [G]$-modules. 
\end{remark}

Let us return to our concern. The Fox derivative $\dis\deriv{}{\gamma_j}$ 
can be seen as a $1$-cocycle in 
$\mathrm{Cross} (F_n;\Z [F_n]) \subset C^1 (F_n;\Z [F_n])$ 
sending 
$\gamma_j$ to $1$ and $\gamma_i$ $(i\neq j)$ to $0$. 
Since a crossed homomorphism $F_n \to \Z [F_n]$ is determined by 
its values on any basis of $F_n$, we see that 
$\left\{ \dis\deriv{}{\gamma_1}, \dis\deriv{}{\gamma_2}, \ldots, 
\dis\deriv{}{\gamma_n}\right\}$ forms a basis of 
$\mathrm{Cross} (F_n;\Z [F_n])$ as a right $\Z [F_n]$-module. 

\begin{proof}[Proof of Proposition $\ref{prop:foxproperty} (3)$] 
Let $\varphi^\ast \Z [F_n]$ be the left $\Z [F_n]$-module 
whose underlying abelian group is $\Z [F_n]$ 
but on which $F_n$ acts through $\varphi$. 
We easily see that 
$\mathrm{Cross} (F_n;\varphi^\ast \Z [F_n])$ is a free 
right $\Z [F_n]$-module of rank $n$ generated by 
\[\left\{ \varphi\left(\dis\deriv{}{\gamma_1}\right), \ 
\varphi\left(\dis\deriv{}{\gamma_2}\right), \ \ldots, \ 
\varphi\left(\dis\deriv{}{\gamma_n}\right)\right\},\] 
where $F_n$ acts from the right by 
the usual, not through $\varphi$, multiplication. 
Now the map $w \mapsto \dis\deriv{\varphi (w)}{\gamma_j}$ is 
in $\mathrm{Cross} (F_n;\varphi^\ast \Z [F_n])$, so that 
we can put $\dis\deriv{\varphi ( \cdot )}{\gamma_j} = 
\dis\sum_{k=1}^n 
\left(\varphi\left(\dis\deriv{}{\gamma_k}\right)\right) \cdot g_k$. 
If we substitute $\gamma_k$ in this equality, we have 
$g_k=\dis\deriv{\varphi (\gamma_k)}{\gamma_j}$ 
and our claim follows. 
\end{proof}
\begin{proof}[Proof of Proposition $\ref{prop:foxproperty} (4)$] 
It suffices to show our claim when $g=v \in F_n$, a monomial. 
We can easily check that the equality 
\[\delta_0 1 = \sum_{j=1}^n\left(\deriv{}{\gamma_j}\right) 
\cdot (\gamma_j-1) \quad 
\in C^1 (F_n;\Z [F_n])\]
holds for $1 \in C^0 (F_n;\Z [F_n])\cong \Z [F_n]$. 
Applying this $1$-cochain to $v$, 
we obtain the desired equality. 
\end{proof}
\begin{remark}
For groups $H \subset G$ and a left $\Z [G]$-module $M$, 
the relative homology $H_\ast (G,H;M)$ (resp.\ 
cohomology $H^\ast (G,H;M)$) can be defined 
by considering $C_\ast (G;M)/C_\ast (H;M)$ (resp.\ 
$\Ker (C^\ast (G;M) \to C^\ast (H;M))$) as usual. 
\end{remark}

\subsection{Fox derivatives in low-dimensional topology}\label{subsec:lowdim}
Fox derivatives and low-dimensional topology are connected by 
using the topological definition of (co)homology of groups. 

\begin{conventions}
For a connected CW-complex $X$, we denote by $\widetilde{X}$ 
its universal covering. We take a base point $p$ of $X$ and 
a lift $\widetilde{p}$ of $p$ 
as a base point of $\widetilde{X}$. 
The group $G:=\pi_1 (X,p)$ acts on $\widetilde{X}$ from 
the right through its deck transformation group, namely, 
the lift of $\gamma \in G$ starting from $\widetilde{p}$ 
reaches $\widetilde{p} \gamma^{-1}$. 
When $X$ is a finite complex, 
we regard the cellular chain complex 
$C_{\ast} (\widetilde{X})$ of $\widetilde{X}$, on which 
$G$ acts from the right, 
as a collection of 
free right $\Z [G]$-modules consisting of column vectors together 
with boundary operators given by left multiplication of matrices. 
For a left $\Z [G]$-module $M$, 
the twisted chain complex $C_{\ast} (X;M)$ is given 
by the tensor product of 
the right $\Z [G]$-module $C_{\ast} (\widetilde{X})$ and 
the left $\Z [G]$-module $M$. This complex gives 
the {\it twisted homology group} 
\index{twisted (co)homology groups} $H_\ast (X;M)$. 
The twisted 
cochain complex $C^\ast (X;M)$ and 
the {\it twisted cohomology group} $H^\ast (X;M)$ 
are defined similarly. 
\end{conventions}

In topology, the homology $H_\ast (G;M)$ and 
cohomology $H^\ast (G;M)$ of a group $G$ 
with coefficients in a left $\Z [G]$-module $M$ 
are defined as twisted (co)homology groups
\begin{align*}
H_\ast (G;M) &= H_\ast (K(G,1);M),\\
H^\ast (G;M) &= H^\ast (K(G,1);M),
\end{align*}
\noindent
where $K(G,1)$ is the {\it Eilenberg-MacLane space} 
\index{Eilenberg-MacLane space} of $G$. 
The space $K(G,1)$ is characterized 
uniquely up to homotopy equivalence 
as a connected CW-complex satisfying
\[\pi_1 (K(G,1))=G, \qquad 
\pi_i (K(G,1))=0 \quad (i \ge 2).\]
\begin{remark}
There are several methods for checking that 
this definition coincides 
with that in the previous section. 
We refer to Brown's book \cite{br} again. 
One method is to see that 
the complex in the previous section is derived from 
the {\it fat realization} of a 
(semi-)simplicial structure of $K(G,1)$.  
\end{remark}
Note that for any CW-complex $X$ with $\pi_1 X=G$, we have 
\begin{align*}
&H_0 (X;M) \cong H_0 (G;M), \qquad H_1 (X;M) \cong H_1 (G;M),\\
&H^0 (X;M) \cong H^0 (G;M), \qquad H^1 (X;M) \cong H^1 (G;M)
\end{align*}
\noindent 
since $K(G,1)$ can be obtained from $X$ by attaching $3$-cells, 
$4$-cells,$\ldots$, so that all higher homotopy groups are eliminated. 
We also see that there exists an epimorphism
\begin{equation}\label{eq:secondhomology}
\SelectTips{cm}{}
\xymatrix{
H_2 (X) \ar@{>>}[r] & H_2 (G).}
\end{equation}
\noindent
The kernel is exactly the image of 
the Hurewicz homomorphism $\pi_2 X \to H_2 (X)$. 

For a group $G$ with a presentation 
$\langle x_1, x_2, \ldots \mid r_1,r_2,\ldots \rangle$, 
we construct a $2$-complex $X$ consisting of 
one $0$-cell, say $p$, one $1$-cell for each generator 
and one $2$-cell for each relation with 
an attaching map according to the word. 
Then $\pi_1 X=G$. Using this construction, 
we now look for a practical method for calculating 
$H_1 (G;M)$ in case 
$G$ has a finite presentation 
$\langle x_1, x_2, \ldots, x_k \mid r_1,r_2,\ldots,r_l \rangle$. 
Consider the chain complex
\[C_2 (X;M) 
\stackrel{\partial_2}{\longrightarrow} 
C_1 (X;M) \stackrel{\partial_1}{\longrightarrow} 
C_0 (X;M)\longrightarrow 0.\]
This complex can be rewritten as 
\[M^l \xrightarrow{D_2 \cdot} 
M^k \xrightarrow{D_1 \cdot}
M \longrightarrow 0\]
with matrices $D_1$, $D_2$. 
Observing the lifts of the loops $x_1, \ldots, x_k$ and 
$r_1, \ldots r_l$ starting from the base point $\widetilde{p}$ 
of $\widetilde{X}$, we see that 
\begin{align*}
D_1 &= \begin{pmatrix}
x_1^{-1}-1 & x_2^{-1}-1 & \cdots & x_k^{-1}-1
\end{pmatrix},\\
D_2 &=
\left(\overline{\left(\frac{\partial r_j}{\partial x_i} 
\right)}\right)_{\begin{subarray}{c}
{}1 \le i \le k\\
1 \le j \le l
\end{subarray}}
\end{align*}
over $\Z [G]$. Here and hereafter the words 
``over $\Z [G]$'' means that we are considering 
all entries to be in $\Z [G]$ as 
images of the natural homomorphism 
$\Z [\langle x_1, \ldots, x_k \rangle] \to \Z [G]$. 
The relation $D_1 D_2=O$ follows 
from Proposition \ref{prop:foxproperty} (4). 
From this expression of $D_2$, we actually see that 
Fox derivatives contribute to 
low-dimensional topology in 
the calculation of $H_1 (G;M)=H_1 (X;M)$. 

\begin{example}\label{ex:1st}
Let $\rho: G \to \Gamma$ be an epimorphism. 
Take a CW-complex $X$ with 
$\pi_1 X=G$. 
Then $C_\ast (X;\Z [\Gamma])$ just corresponds to 
the cellular complex of the $\Gamma$-covering 
$X_\Gamma$ of $X$ with respect to $\rho$. 
Hence we have 
\[H_1 (G;\Z [\Gamma]) \cong H_1 (X;\Z [\Gamma]) \cong 
H_1 (X_\Gamma) \cong H_1 (\Ker \rho).\]
\end{example}
\begin{proof}[Proof of Proposition $\ref{prop:foxproperty} (5)$] 
First, we may suppose that 
$\rho:F_n \to \Gamma$ is an epimorphism. We now 
consider $H_1 (F_n;\Z [\Gamma])$. 
The space $K(F_n,1)$ is given by a bouquet $X$ of $n$ circles 
corresponding to the generating system 
$\{\gamma_1, \gamma_2, \ldots, \gamma_n\}$ of $F_n$. 
By an observation similar to the one for the matrix $D_2$ above, 
we see that the abelianization map 
\[\Ker \rho \cong \pi_1 X_\Gamma \longrightarrow 
H_1 (X_\Gamma) = 
\mbox{\{$1$-cycles on $X_\Gamma$\}} 
\subset C_1 (X_\Gamma) \cong (\Z [\Gamma])^n\]
is given by 
\[v \longmapsto 
\sideset{^\rho\!}{^T}{\Tmatrix{\begin{array}{cccc}
\overline{\dis\deriv{v}{\gamma_1}} & 
\overline{\dis\deriv{v}{\gamma_2}} & \cdots & 
\overline{\dis\deriv{v}{\gamma_n}} 
\end{array}}}.\]
The kernel of this map is 
$[\Ker \rho, \Ker \rho]$ and our claim immediately 
follows. 
\end{proof}
\begin{example}\label{ex:aug}
Let $X$ be the bouquet in the above proof of 
Proposition \ref{prop:foxproperty} (5) with base point $p$. 
Consider the homology exact sequence 
\begin{align*}
H_1 (X;\Z [F_n]) &\longrightarrow H_1 (X,\{p\};\Z [F_n]) 
\longrightarrow H_0 (\{p\};\Z [F_n]) \\
&\longrightarrow H_0 (X;\Z [F_n]) 
\longrightarrow H_0 (X,\{p\};\Z [F_n]), 
\end{align*}
\noindent 
which can be regarded as that for the pair $(F_n,\{1\})$ 
of groups. 
Clearly $H_1 (X;\Z [F_n])=H_1(\widetilde{X})=0$ and 
$H_0 (X;\Z [F_n])=H_0(\widetilde{X})=\Z$ 
(see Example \ref{ex:1st}). 
From the cell structures, we immediately see that 
$H_1 (X,\{p\};\Z [F_n]) =C_1 (X,\{p\};\Z [F_n]) \cong (\Z [F_n])^n$, 
$H_0 (X,\{p\};\Z [F_n]) \cong 0$ 
and $H_0 (\{p\};\Z [F_n]) \cong \Z [F_n]$. 
Hence the above exact sequence is rewritten as 
\[0 \longrightarrow (\Z [F_n])^n \stackrel{\chi}{\longrightarrow} 
\Z [F_n] \longrightarrow \Z \longrightarrow 0.\]
The third map is given by the trivializer $\mathfrak{t}$ with 
the kernel $I(F_n)$ called the {\it augmentation ideal} of $\Z [F_n]$. 
Consequently, the map $\chi$ induces an isomorphism 
\begin{equation}\label{eq:aug}
\chi: (\Z [F_n])^n \stackrel{\cong}{\longrightarrow} 
I(F_n)
\end{equation}
\noindent
as right $\Z [F_n]$-modules. 
We can easily check that 
\[\chi ((a_1,a_2,\ldots,a_n)^T) = \sum_{i=1}^n (\gamma_i^{-1}-1) a_i\]
for $(a_1,a_2,\ldots,a_n)^T \in (\Z [F_n])^n$. 
That is, $\{\gamma_1^{-1}-1,\gamma_2^{-1}-1,\ldots,
\gamma_n^{-1}-1\}$ forms a right free basis of $I(F_n)$. Note that 
the map $F_n \to (\Z F_n)^n$ sending $v \in F_n$ to 
$\overline{\chi^{-1} (v^{-1}-1)}$ {\it recovers} 
the Fox derivatives (cf. Proposition \ref{prop:foxproperty} (4)):
\[\overline{\chi^{-1} (v^{-1}-1)}= 
\begin{pmatrix}
\dis\deriv{v}{\gamma_1} & \dis\deriv{v}{\gamma_2} & \cdots & 
\dis\deriv{v}{\gamma_n} 
\end{pmatrix}^T.\]
\end{example}

We close this section by the following 
application of the Fox calculus to 
knot theory: 
\begin{example}[The Alexander polynomial of a knot]\label{ex:alexanderpolyn}
Let $K$ be a {\it knot} in the 3-sphere $S^3$. That is, 
$K$ is a smoothly embedded circle in $S^3$. 
The fundamental group of the knot exterior $E(K):=S^3 - N(K)$ of 
$K$, where $N(K)$ is an open tubular neighborhood of $K$, is called 
the {\it knot group} $G(K)$ of $K$. 
We have $H_1 (E(K)) \cong H_1 (G(K)) \cong \Z$ generated by 
the meridian $t$ of $K$ with a fixed orientation. 
The {\it Alexander module} 
$\mathcal{A}^\Z (K)$ of $K$ is defined as 
\[\mathcal{A}^\Z (K) := H_1 (E(K);\Z [\langle t \rangle]) 
=H_1 (G(K);\Z [\langle t \rangle]).\]
Then the {\it Alexander polynomial} $\Delta_K (t)$ of $K$ 
is defined as the determinant of a square (say $k \times k$) matrix $D$ 
representing $\mathcal{A}^\Z (K)$, namely $D$ fits into 
an exact sequence 
\[\Z [\langle t \rangle]^k \stackrel{D \cdot}{\longrightarrow} 
\Z [\langle t \rangle]^k 
\longrightarrow \mathcal{A}^\Z (K) \longrightarrow 0.\]
It can be checked that $\det D$ up to 
multiplication by $\pm t^m$ ($m \in \Z$) 
does not depend on the choice of $D$. 
To obtain such a matrix $D$, take a Wirtinger presentation 
of $G(K)$, which gives a presentation of the form 
$\langle x_1, x_2, \ldots, x_{k+1} 
\mid r_1, r_2, \ldots, x_k \rangle$. 
Using the arguments in this subsection, we can see that 
the square matrix 
$\left(\overline{\left(\frac{\partial r_j}{\partial x_i} 
\right)}\right)_{1 \le i, j \le k}$ 
represents $\mathcal{A}^\Z (K)$. 
\end{example}

\section{Magnus representations for automorphism groups 
of free groups}
\label{sec:Magnus1}

In this section, we overview generalities of 
Magnus representations for 
the automorphism group $\Aut F_n$ 
of a free group 
$F_n=\langle \gamma_1, \gamma_2, \ldots, \gamma_n \rangle$. 
Applications to the mapping class 
group of a surface are discussed in the next section. 

\begin{definition}\label{def:magnus_free}
The {\it $($universal\,$)$ Magnus representation} 
\index{Magnus representation!for $\mathrm{Aut}\,F_n$} for 
$\Aut F_n$ is the map
\[r:\Aut F_n \to M(n,\Z [F_n])\]
assigning to $\varphi \in \Aut F_n$ the matrix
\[r(\varphi):=\left(\overline{\left(\frac{\partial \varphi(\gamma_j)}
{\partial \gamma_i} \right)}\right)_{i,j},\]
which we call the {\it Magnus matrix} \index{Magnus matrix} 
for $\varphi$. 
\end{definition}
\noindent
While we call the map $r$ the Magnus ``representation'', it is 
actually a crossed homomorphism in the following sense: 
\begin{proposition}\label{prop:MagnusAut_crossed}
For $\varphi$, $\psi \in \Aut F_n$, the equality
\[r (\varphi \psi) = r(\varphi) \cdot {}^\varphi r(\psi)\]
holds. In particular, the image of $r$ is included in the set 
$\mathrm{GL}(n,\Z [F_n])$ of invertible matrices.
\end{proposition}
\begin{proof}
Although we need to be careful about the noncommutativity 
of $\Z [F_n]$, the proof is an easy application of Proposition 
\ref{prop:foxproperty} (3) together with the fact that 
$r(\mathrm{id}_{F_n})=I_n$. 
\end{proof}
\begin{remark}
The second assertion of Proposition 
\ref{prop:MagnusAut_crossed} is part of 
Birman's inverse function theorem \cite{birman_inverse} stating 
that an endomorphism $\psi:F_n \to F_n$ is an automorphism if and 
only if $r(\psi)$, which makes sense, 
belongs to $\mathrm{GL}(n,\Z [F_n])$. 
\end{remark}
The map $r$ is injective since
$\varphi(\gamma_i)$ is recovered 
from $r(\varphi)$ 
by applying Proposition 
\ref{prop:foxproperty} (4) to the $i$-th column
for each $\varphi \in \Aut F_n$, 
namely there is no lack of information. To obtain a 
genuine representation, a homomorphism, we need to reduce 
information of the map $r$ 
as follows. 
Let $\rho:F_n \to \Gamma$ be an epimorphism whose kernel is 
characteristic, namely $\Ker \rho$ is invariant under 
all automorphisms of $F_n$. Then $\rho$ induces a natural homomorphism 
$\Aut F_n \to \Aut \Gamma$. We consider the matrix $r_\rho (\varphi)$ 
obtained from the Magnus matrix $r (\varphi)$ 
by applying the map $\rho:\Z [F_n] \to \Z [\Gamma]$ to each entry. 
\begin{proposition}\label{prop:MagnusAut1}
Let $\rho :F_n \to \Gamma$ be an epimorphism whose kernel is 
characteristic. 
Then the restriction of the map
\[r_\rho: \Aut F_n \longrightarrow \mathrm{GL}(n,\Z [\Gamma])\] 
to $\Ker (\Aut F_n \to \Aut \Gamma)$ is a homomorphism. 
Moreover, the kernel of $r_\rho$ coincides with 
that of the natural homomorphism 
\[\Aut F_n \longrightarrow \Aut (F_n/[\Ker \rho,\Ker \rho]).\] 
\end{proposition}
\begin{proof}
The first half of our assertion follows from 
Proposition \ref{prop:MagnusAut_crossed}. 
To show the second, we first note that 
the map $\Aut F_n \to \Aut \Gamma$ is decomposed as 
\[\Aut F_n \longrightarrow \Aut (F_n/[\Ker \rho,\Ker \rho]) 
\longrightarrow \Aut (F_n/\Ker \rho) = \Aut \Gamma,\]
so that 
$\Ker (\Aut F_n \to \Aut (F_n/[\Ker \rho,\Ker \rho])) \subset 
\Ker (\Aut F_n \to \Aut \Gamma)$. 

Suppose $\varphi \in \Ker (\Aut F_n \to \Aut (F_n/[\Ker \rho,\Ker \rho]))$, 
then there exists $v_i \in [\Ker \rho,\Ker \rho]$ such that 
$\varphi (\gamma_i)=\gamma_i v_i$ for each $i$. Using Proposition 
\ref{prop:foxproperty} (2) and (5), we have
\[\rho \left(\deriv{\varphi (\gamma_i)}{\gamma_j}\right)=
\rho \left(\deriv{(\gamma_i v_i)}{\gamma_j} \right)=
\rho \left(
\deriv{\gamma_i}{\gamma_j}+\gamma_i \deriv{v_i}{\gamma_j}\right)
=\delta_{ij}.\]
Therefore $r_\rho (\varphi) = I_n$ and 
$\Ker (\Aut F_n \to \Aut (F_n/[\Ker \rho,\Ker \rho])) \subset \Ker r_\rho$. 
On the other hand, if $\varphi \in \Ker r_\rho$, 
we have 
\begin{align*}
\rho\left(\deriv{(\gamma_i^{-1}\varphi(\gamma_i))}{\gamma_j}\right)
&=\rho\left(\deriv{\gamma_i^{-1}}{\gamma_j}
+\gamma_i^{-1}\deriv{\varphi (\gamma_i)}{\gamma_j}\right)\\
&=\rho(-\delta_{ij}\gamma_i^{-1}) + 
\rho\left(\gamma_i^{-1}\deriv{\varphi (\gamma_i)}{\gamma_j}\right)\\
&=-\delta_{ij}\rho(\gamma_i^{-1})+\rho(\gamma_i^{-1}) \delta_{ij} \\
&=0.
\end{align*}
By Proposition \ref{prop:foxproperty} (5), 
we see that $\gamma_i^{-1} \varphi(\gamma_i) \in [\Ker \rho, \Ker \rho]$. 
This means $\varphi \in \Ker (\Aut F_n \to 
\Aut (F_n/[\Ker \rho,\Ker \rho]))$ 
and hence $\Ker r_\rho \subset 
\Ker (\Aut F_n \to \Aut (F_n/[\Ker \rho,\Ker \rho]))$ follows. 
This completes the proof. 
\end{proof}
\begin{remark}
If we use the description of Fox derivatives in Example \ref{ex:aug}, 
the Magnus representation $r_\rho$ associated with $\rho$ as above 
can be regarded as a 
transformation of $I(F_n) \otimes_{\Z [F_n]} \Z [\Gamma]$ by 
the diagonal action of $\Aut (F_n)$. 
\end{remark}
\begin{example}\label{ex:homologyRep}
The trivializer $\mathfrak{t}:\Z [F_n] \to \Z$ 
is induced from the trivial homomorphism $F_n \to \{1\}$ and 
$\Aut F_n$ acts trivially on $\{1\}$. Hence 
we have a homomorphism 
\[r_\mathfrak{t}:\Aut F_n \longrightarrow \mathrm{GL}(n,\Z).\]
Note that $r_\mathfrak{t}$ coincides with the action of $\Aut F_n$ on 
$H_1:=H_1 (F_n) \cong \Z^n$. 
Nielsen studied the map $r_\mathfrak{t}$ and showed that 
it is surjective \cite{ni}. 
The kernel $IA_n:=\Ker r_\mathfrak{t}$ is called 
the {\it IA-automorphism} \index{IA-automorphism group} group. 
A finite generating set of 
$IA_n$ was first given by Magnus \cite{magnus2}. 
We can see a summary of 
these works in Morita's chapter 
\cite[Section 4]{morita_survey} in the first volume of this handbook. 
\end{example}
\begin{example}\label{ex:Magnus_IA}
The abelianization homomorphism 
$\mathfrak{a}: F_n \to H_1$ is 
most frequently used as a reduction homomorphism $\rho$. 
In this case, the restriction of $r_\mathfrak{a}$ to $IA_n$ yields 
a homomorphism 
\[r_\mathfrak{a}:IA_n \longrightarrow \mathrm{GL}(n,\Z [H_1]).\]
This representation was first introduced by Bachmuth \cite{bachmuth}, who 
defined the representation as a transformation of 
the $(1,2)$-entry of the Magnus representation $F_n \to 
\mathrm{GL}(2,(\Z [H_1])[S])$ 
in Section \ref{subsec:magnus_free} and studied the automorphism group 
of the metabelian quotient of $F_n$. 

In relation to topology, 
we consider the {\it braid group} $B_n$ 
and the {\it pure braid group} $P_n \subset B_n$ of $n$ strings. 
For the definitions, we refer to Paris' chapter \cite{paris} in 
the second volume of this handbook 
as well as to Birman's book \cite{bi}. 
We now focus on Artin's theorem \cite{artin, artin2} saying that 
there exists a natural embedding of $B_n$ into $\Aut F_n$, which 
embeds $P_n$ into $IA_n$. By postcomposing 
$r_\mathfrak{a}$ with 
this embedding, we obtain a homomorphism
\[r_\mathfrak{a}:P_n \longrightarrow \mathrm{GL}(n,\Z [H_1])\]
called the {\it Gassner representation} \index{Gassner representation} 
\cite{gassner}. 
To obtain a representation for $B_n$, we further 
reduce $\Z [H_1]$ to 
$\Z [\langle t \rangle]$ by the homomorphism 
$\mathfrak{b}: H_1 \to \langle t \rangle$ 
sending each $\gamma_j$ to $t$. Then 
we obtain a homomorphism 
\[r_{\mathfrak{b} \circ \mathfrak{a}}: B_n \longrightarrow 
\mathrm{GL}(n,\Z [\langle t \rangle])\]
called the {\it Burau representation} \index{Burau representation} 
\cite{burau}. Note 
that in their papers, Burau and Gassner 
did not use the Fox calculus, but 
gave explicit formulas to construct their representations. 
\end{example}

Continuing $r_\mathfrak{t}$ and $r_\mathfrak{a}$, we now consider 
the following system consisting of a filtration of $\Aut F_n$ and 
a sequence of Magnus representations. 
Let 
\[F_n^{(0)}:= F_n \supset F_n^{(1)} \supset F_n^{(2)} \supset 
F_n^{(3)} \supset \cdots\] 
be the {\it derived series} \index{derived series} of $F_n$ defined by 
$F_n^{(k+1)}=[F_n^{(k)},F_n^{(k)}]$ for $k \ge 0$. 
$F_n^{(k)}$ is a characteristic subgroup of $F_n$ 
and we denote by $\mathfrak{p}_k:F_n \to F_n/F_n^{(k)}$ 
the natural projection. 
Note that $\mathfrak{p}_0=\mathfrak{t}$ and 
$\mathfrak{p}_1=\mathfrak{a}$. 
Define a filtration 
\[I^0 A_n := \Aut F_n \supset I^1 A_n \supset I^2 A_n 
\supset I^3 A_n \supset \cdots\]
of $\Aut F_n$ by $I^k A_n := \Ker (\Aut F_n \to \Aut (F_n/F_n^{(k)}))$ 
for $k \ge 1$. Note that $I^1 A_n=IA_n$. 
By Proposition \ref{prop:MagnusAut1}, the map
\[r_{\mathfrak{p}_k}: I^k A_n \longrightarrow 
\mathrm{GL}(n,\Z [F_n/F_n^{(k)}])\]
is a homomorphism with 
\begin{align*}
\Ker r_{\mathfrak{p}_k}&=
\Ker (\Aut F_n \to \Aut(F_n/[F_n^{(k)},F_n^{(k)}]))\\
&=\Ker (\Aut F_n \to \Aut(F_n/F_n^{(k+1)})) \\
&= I^{k+1}A_n.
\end{align*}
\noindent
That is, each of the Magnus representations $r_{\mathfrak{p}_k}$ plays 
the role of an obstruction for an automorphism in $I^k A_n$ to be 
in the next filter. 
It seems difficult to determine 
the image of $r_{\mathfrak{p}_k}$ 
in general. However, if we put 
\[G_n^{(k)}:=
\left\{\begin{array}{c|c}
\!
{\small (a_{ij})_{i,j} \in \mathrm{GL}(n,\Z  [F_n/F_n^{(k)}])} & 
{\small 
\begin{array}{l}
\sum_{i=1}^n (\gamma_i^{-1}-1)a_{ij}= \gamma_j^{-1}-1\\
\mbox{in $\Z  [F_n/F_n^{(k)}]$ for all $j$.}
\end{array}}
\end{array}\!\!\right\},\]
then we can check that $G_n^{(k)}$ is a subgroup of 
$\mathrm{GL}(n,\Z  [F_n/F_n^{(k)}])$ and 
it includes the image of $r_{\mathfrak{p}_k}$ 
since 
we can apply Proposition \ref{prop:foxproperty} (4) under the 
condition that 
$\varphi (\gamma_j) = \gamma_j \in F_n/F_n^{(k)}$ for all $j$. 
\begin{proposition}
There exists a canonical injective homomorphism 
\[\Phi: G_n^{(k)} \longrightarrow \Aut(F_n/F_n^{(k+1)})\]
and we have an exact sequence 
\[1 \longrightarrow G_n^{(k)} 
\stackrel{\Phi}{\longrightarrow} \Aut(F_n/F_n^{(k+1)}) 
\longrightarrow \Aut(F_n/F_n^{(k)}).\]
\end{proposition}
\begin{proof}
We first define a homomorphism $\Phi: G_n^{(k)} 
\to \Aut(F_n/F_n^{(k+1)})$. 
Let $A =(a_{ij})_{i,j} \in G_n^{(k)}$. 
Then $\sum_{i=1}^n (\gamma_i^{-1}-1)a_{ij}= \gamma_j^{-1}-1$ holds 
for $j=1,2,\ldots,n$. 
By \cite[Theorem 3.7]{bi}, there exists $v_j \in F_n$ satisfying 
$\mathfrak{p}_k (v_j)= \gamma_j \in F_n/F_n^{(k)}$ and 
$\mathfrak{p}_k \left( 
\overline{\deriv{v_j}{\gamma_i}}\right) = a_{ij} \in 
\Z [F_n/F_n^{(k)}]$. 
Moreover such a $v_k$ is unique 
up to $F_n^{(k+1)}$. 
Define an endomorphism 
$\varphi_A:F_n \to F_n$ by 
$\varphi_A (\gamma_j)=v_j$. By construction, 
it induces the identity map on $F_n/F_n^{(k)}$. 
Then Proposition \ref{prop:foxproperty} (3) 
shows that $A \mapsto \varphi_A$ defines a homomorphism $\Phi$ 
from $G_n^{(k)}$ to 
the monoid of endomorphisms of $F_n/F_n^{(k+1)}$. 
Since $G_n^{(k)}$ is a group, the image of $\Phi$ 
should be included in $\Aut(F_n/F_n^{(k+1)})$. 
Consequently we obtain a homomorphism 
$\Phi: G_n^{(k)} \to \Aut(F_n/F_n^{(k+1)})$. 
The composition $G_n^{(k)} \xrightarrow{\Phi} \Aut(F_n/F_n^{(k+1)}) \to 
\Aut(F_n/F_n^{(k)})$ is trivial by construction. 
On the other hand, the homomorphism 
$r_{\mathfrak{p}_k}: I^k A_n \longrightarrow G_n^{(k)}$ 
induces a homomorphism  
$\Ker (\Aut(F_n/F_n^{(k+1)}) \to \Aut(F_n/F_n^{(k)})) \to G_n^{(k)}$, 
which gives the inverse of $\Phi$. 
This shows the exactness of the sequence. 
\end{proof}
\begin{remark}
The above system can be regarded as 
a refinement of the Andreadakis filtration 
of $\Aut F_n$, which is defined by 
using the {lower central series} of $F_n$ 
(see Section \ref{subsec:JohnsonMorita} and 
Morita \cite[Section 7]{morita_survey}). 
\end{remark}

For a characteristic subgroup $G \subset F_n$, 
an automorphism of $F_n/G$ is said to be {\it tame} \index{tame automorphism} 
if it is induced from one of $F_n$. The study of tame automorphisms has 
been of great interest among researchers 
in combinatorial group theory. 
We refer to the surveys by Gupta \cite{gupta1} 
and Gupta-Shpilrain \cite{gupta_shpilrain} for details. 
As for $\Aut(F_n/F_n^{(k)})$, 
the following results are known.
Bachmuth \cite{bachmuth} 
and Bachmuth-Mochizuki \cite{bm1,bm2} showed that 
all the automorphisms of $F_n/F_n^{(2)}$ are tame if $n =2$ or $n \ge 4$ 
while there exists a non-tame automorphism when $n=3$. 
The latter fact was first shown by Chein \cite{chein}. Moreover 
it was shown by Shpilrain \cite{shpilrain} that 
non-tame automorphisms of $F_n/F_n^{(k)}$ 
do exist and the map 
$\Aut(F_n/F_n^{(k+1)}) \to \Aut(F_n/F_n^{(k)})$ is not 
surjective for $n \ge 4$ and $k\ge 3$. 

Recently, Satoh \cite{satoh_magnus} showed that 
$I^2 A_n$ is not finitely generated and moreover 
$H_1 (I^2 A_n)$ has infinite rank for $n \ge 2$ 
(see also Church-Farb \cite{cf} mentioned in the next section). 

We here pose the following problems: 
\begin{problems} 
(1) \ Determine the image of $r_{\mathfrak{p}_k}$. 

\noindent
(2) \ Find a structure on the filtration $\{I^k A_n\}$. 
For example, the associated graded module, 
namely the direct sum of the successive quotients of 
the Andreadakis filtration has 
a natural graded Lie algebra structure.
\end{problems}

\section{Magnus representations for mapping class groups}\label{sec:MagnusMCG}

Now we start our discussion on applications of Magnus representations to 
mapping class groups of surfaces. 

\subsection{Definition}\label{subsec:MCG}
Let $\Sigma_{g,n}$ be a compact oriented surface of genus 
$g$ with $n$ boundary components. 
We take a base point $p$ in $\partial \Sg{g}{n}$ when $n \ge 1$ and 
arbitrarily when $n=0$. 
The fundamental group 
$\pi_1 \Sigma_{g,n}$ of $\Sigma_{g,n}$ with respect to the 
base point $p$ is a free group of 
rank $2g+n-1$ except the cases $n=0$, 
when $\pi_1 \Sigma_{g,0}$ is given as a quotient of 
a free group of rank $2g$ by one relation. 

The {\it mapping class group} \index{mapping class group} 
of $\Sg{g}{n}$ is the group of 
all isotopy classes of orientation-preserving diffeomorphisms of 
$\Sg{g}{n}$, where all diffeomorphisms and 
isotopies are assumed to fix $\partial \Sg{g}{n}$ pointwise when $n >0$. 
For a general theory of mapping class groups, 
we refer to Birman \cite{bi}, 
Ivanov \cite{ivanov}, Farb-Margalit \cite{Farb_Margalit} as well as 
Morita's chapter \cite{morita_survey}. 

The case where $n=1$ is now of particular interest 
and we assume it hereafter. 
In our context, the following theorem, which is often called 
the \index{Dehn-Nielsen theorem}{\it Dehn-Nielsen} theorem, 
is crucial for applying the techniques 
mentioned in the previous sections to $\Mgb$. 
We put $\pi:=\pi_1 \Sigma_{g,1}$ for simplicity. 
\begin{theorem}[The Dehn-Nielsen theorem]\label{thm:DN_MCG}
The natural action of $\Mgb$ on $\pi$ induces an isomorphism
\[\Mgb \cong \{\varphi \in \Aut \pi \mid \varphi (\zeta) =\zeta\},\]
where $\zeta \in \pi$ corresponds the boundary loop $\partial \Sgb$. 
\end{theorem}
\noindent
The corresponding injection $\sigma: \Mgb \hookrightarrow \Aut \pi$ is 
called the {Dehn-Nielsen embedding}. 
\begin{remark}\label{rem:extension}
When $n=0$, $\Mg{g}{0}$ acts on $\pi_1 \Sg{g}{0}$ 
{\it up to conjugation} and the corresponding Dehn-Nielsen theorem is 
stated in terms of the {\it outer} 
automorphism group $\Out (\pi_1 \Sg{g}{0})$ of $\pi_1 \Sg{g}{0}$. 
When $n \ge 2$, the homomorphism $\Mg{g}{n} \to \Aut (\pi_1 \Sg{g}{n})$ 
is {\it not} injective since the Dehn twist along a loop parallel to 
one of the boundaries not containing $p$ 
acts trivially on $\pi_1 \Sg{g}{n}$. Therefore 
we need a special care in treating such boundaries. 
Filling these $n-1$ boundaries by $n-1$ copies of $\Sigma_{1,1}$, 
we have an embedding $\Sg{g}{n} \hookrightarrow \Sg{g+n-1}{1}$. 
A diffeomorphism of $\Sg{g}{n}$ naturally extends to 
one of $\Sg{g+n-1}{1}$ such that it restricts to the identity map 
on $\Sg{g+n-1}{1}-\Sg{g}{n}$. This induces a homomorphism 
$\mathcal{M}_{g,n} \to \mathcal{M}_{g+n-1,1}$, which is known to be 
injective. Hence we have $\Mg{g}{n} \subset \Mg{g+n-1}{1} 
\subset \Aut (\pi_1 \Sg{g+n-1}{1})$. 
\end{remark}

We now take $2g$ oriented loops 
$\gamma_1, \gamma_2, \ldots, \gamma_{2g}$ 
as in Figure \ref{fig:generator}.
They form a basis of $\pi$ and we often identify $\pi$ with 
the free group 
$F_{2g}=\langle \gamma_1, \gamma_2, \ldots, \gamma_{2g} \rangle$ 
of rank $2g$. We have 
$\zeta=[\gamma_1, \gamma_{g+1}][\gamma_2, \gamma_{g+2}] \cdots 
[\gamma_g, \gamma_{2g}]$.

\begin{figure}[htbp]
\begin{center}
\includegraphics{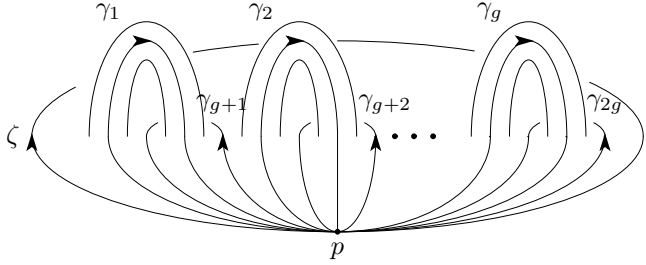}
\end{center}
\caption{Our basis of $\pi_1 \Sigma_{g,1}$}
\label{fig:generator}
\end{figure}

We put $H:=H_1 (\Sgb)=H_1 (\Sgb, \partial \Sgb)$. 
The group $H$ can be identified with $\Z^{2g}$ by choosing 
$\{\gamma_1, \gamma_2, \ldots, \gamma_{2g}\}$ as a 
basis of $H$, where 
we write $\gamma_j$ again for $\gamma_j$ as 
an element of $H=\pi/[\pi,\pi]$. 
Poincar\'e duality endows 
$H$ with a non-degenerate anti-symmetric bilinear form 
\[\mu: H \otimes H \longrightarrow \Z\]
called the \index{intersection pairing}{\it intersection pairing}. 
The above basis of $H$ 
is a symplectic basis with respect to $\mu$, 
namely we have 
\[\mu (\gamma_i, \gamma_j)=\mu (\gamma_{g+i},\gamma_{g+j})=0, 
\qquad 
\mu(\gamma_{i}, \gamma_{g+j}) = -\mu(\gamma_{g+j},\gamma_{i}) = 
\delta_{ij}\]
\noindent
for $i,j=1,2,\ldots,g$. 
The action of $\Mgb$ on $H \cong \Z^{2g}$ defines a homomorphism 
$\sigma : \Mgb \to \mathrm{GL}(2g,\Z)$. 
Since this action preserves the intersection pairing, 
the image of $\sigma$ is included in the {\it symplectic group} 
\index{symplectic group} (or the {\it Siegel modular group}) 
\[\mathrm{Sp} (2g,\Z) = \{ X \in \mathrm{GL}(2g,\Z) \mid X^T J X =J\},\]
where $J=\begin{pmatrix}
O & I_g \\ -I_g & O \end{pmatrix}$. Note that 
$\mathrm{Sp} (2g,\Z) \subset \mathrm{SL}(2g,\Z)$. It is classically known that 
$\sigma:\Mgb \to \mathrm{Sp}(2g,\Z)$ is surjective. 
Consequently we have an exact sequence
\[1 \longrightarrow \mathcal{I}_{g,1} \longrightarrow 
\mathcal{M}_{g,1} \stackrel{\sigma}{\longrightarrow} 
\mathrm{Sp}(2g,\mathbb{Z}) 
\longrightarrow 1,\]
where $\Igb:=\Ker \sigma$ is called \index{Torelli group} 
the {\it Torelli group}. 

Using the Dehn-Nielsen theorem, we define 
the (universal) Magnus representation 
\[r : \Mgb \longrightarrow \mathrm{GL}(2g,\Z [\pi])\]
for $\Mgb$ \index{Magnus representation!for $\mathcal{M}_{g,1}$} 
by assigning to 
$\varphi \in \Mgb$ the matrix \index{Magnus matrix} 
$r(\varphi):=\left(\overline{\left(\dis\frac{\partial \varphi(\gamma_j)}
{\partial \gamma_i} \right)}\right)_{i,j}$. 
The map $r$ is an injective 
crossed homomorphism by 
Proposition \ref{prop:MagnusAut_crossed} 
and the paragraph subsequent to it. 
By definition, 
$\sigma=r_\mathfrak{t}$ holds 
and we have $\Igb=\Mgb \cap IA_{2g}$. 
Also we have a crossed homomorphism 
\[r_\mathfrak{a}: \Mgb \longrightarrow \mathrm{GL}(2g,\Z [H]),\]
which restricts to a homomorphism 
\[r_\mathfrak{a}|_{\Igb} : \Igb \longrightarrow 
\mathrm{GL}(2g,\Z [H]),\]
called the \index{Magnus representation!for $\mathcal{I}_{g,1}$} 
{\it Magnus representation for the Torelli group}. 
Applications of these Magnus representations to $\Mgb$ were 
first given by Morita in \cite{mo9, mo}. 
After that, the study of the Magnus representation 
for the Torelli group has been intensively pursued by 
Suzuki \cite{su,suzuki_irred,suz,suz2}. 
\begin{remark}
In \cite{abp}, Andersen-Bene-Penner constructed 
{\it groupoid} lifts of the Dehn-Nielsen embedding 
$\sigma: \Mgb \hookrightarrow \Aut \pi$ and 
the Magnus representation $r$ to 
the {\it Ptolemy groupoid} $\mathfrak{Pt}(\Sgb)$ of $\Sgb$. 
This groupoid may be regarded as a discrete model of paths 
in the (decorated) Teichm\"uller space \cite{penner} and 
$\Mgb$ can be embedded as the oriented paths starting from a fixed 
vertex $v$ and reaching vertices in the same $\Mgb$-orbit as $v$. 
\end{remark}

\subsection{Symplecticity and its topological interpretation}
\label{subsec:magnus_MCG_symp}

Now we overview known properties of 
the Magnus representations $r$ and 
$r_\mathfrak{a}$. The first one is called the (twisted) symplecticity. 

\begin{theorem}[Morita \cite{mo}, Suzuki \cite{suz}, Perron \cite{perron}]
\label{thm:symplecticMCG}
For any $\varphi \in \mathcal{M}_{g,1}$, the Magnus matrix $r (\varphi)$ 
satisfies the equality 
\[\overline{r (\varphi)^T} \ \widetilde{J} \ r (\varphi) = 
{}^{\varphi} \widetilde{J},\]
where $\widetilde{J}=
\left( \begin{array}{cc} J_1 & J_2 \\ J_3 & J_4 \end{array}\right) 
\in \mathrm{GL} (2g, \Z [\pi])$ is defined by 

{\small 
\begin{align*}
J_1 &=
\left( \begin{array}{ccccc}
1-\gamma_1 & & & & \\
(1-\gamma_2)(1-\gamma_1^{-1})& 1-\gamma_2 &&&\hsymb{0}\\
(1-\gamma_3)(1-\gamma_1^{-1})& (1-\gamma_3)(1-\gamma_2^{-1})& 
1-\gamma_3&&\\
\vdots & \vdots && \ddots &\\
(1-\gamma_g)(1-\gamma_1^{-1})& (1-\gamma_g)(1-\gamma_2^{-1})& \cdots & 
& 1-\gamma_g
\end{array}\right),\\
J_2 &=
\left( \begin{array}{ccccc}
\gamma_1 \gamma_{g+1}^{-1}& & & & \\
(1-\gamma_2)(1-\gamma_{g+1}^{-1})& \gamma_2 \gamma_{g+2}^{-1} &&
&\hsymb{0}\\
(1-\gamma_3)(1-\gamma_{g+1}^{-1})& (1-\gamma_3)(1-\gamma_{g+2}^{-1})& 
\gamma_3 \gamma_{g+3}^{-1} &&\\
\vdots & \vdots && \ddots &\\
(1-\gamma_g)(1-\gamma_{g+1}^{-1})& (1-\gamma_g)(1-\gamma_{g+2}^{-1})& 
\cdots & & \gamma_g \gamma_{2g}^{-1}
\end{array}\right),\\
J_3 &=
\left( \begin{array}{cccc}
1-\gamma_1^{-1} - \gamma_{g+1} & & & \\
(1-\gamma_{g+2})(1-\gamma_1^{-1})& 1-\gamma_2^{-1} - \gamma_{g+2} 
&&\hsymb{0}\\
(1-\gamma_{g+3})(1-\gamma_1^{-1})& (1-\gamma_{g+3})(1-\gamma_2^{-1})& 
\ddots & \\ 
\vdots & \vdots & & \\ 
(1-\gamma_{2g})(1-\gamma_1^{-1})& (1-\gamma_{2g})(1-\gamma_2^{-1})& 
\cdots & 1-\gamma_g^{-1} - \gamma_{2g}
\end{array}\right),\\
J_4 &=
\left( \begin{array}{ccccc}
1-\gamma_{g+1}^{-1} & & & &\\
(1-\gamma_{g+2})(1-\gamma_{g+1}^{-1})& 1- \gamma_{g+2}^{-1} 
&&&\hsymb{0} \\
(1-\gamma_{g+3})(1-\gamma_{g+1}^{-1})& (1-\gamma_{g+3})(1-\gamma_{g+2}^{-1})& 
1- \gamma_{g+3}^{-1}&&\\
\vdots & \vdots && \ddots &\\
(1-\gamma_{2g})(1-\gamma_{g+1}^{-1})& (1-\gamma_{2g})(1-\gamma_{g+2}^{-1})& 
\cdots & & 1-\gamma_{2g}^{-1}
\end{array}\right).
\end{align*}
}
\end{theorem}
\noindent
Note that the matrix $\widetilde{J}$ 
first appeared in 
Papakyriakopoulos' paper \cite{papa} and it 
is mapped to the matrix $J$ by the trivializer 
$\mathfrak{t}: \Z [\pi] \to \Z$. 
Morita used a finite generating system of 
$\mathcal{M}_{g,1}$ to show that 
the equality holds for each element of the system. 
On the other hand, Suzuki gave a 
topological description of the Magnus representation 
and showed that the equality holds for any element of $\Mgb$. 
Perron's proof is similar to Suzuki's. 

Suzuki's description is as follows. 
First, consider the twisted homology 
$H_1 (\Sgb, \{p\};\Z [\pi])$, which 
coincides with the usual homology 
$H_1 (\widetilde{\Sgb}, f^{-1}(p))$ 
of the universal covering $f: \widetilde{\Sgb} \to \Sgb$. 
This module is isomorphic to $(\Z [\pi])^{2g}$ and 
the set of lifts $\widetilde{\gamma_1}, \widetilde{\gamma_2}, 
\ldots, \widetilde{\gamma_{2g}}$ 
(see Convention in Section \ref{subsec:lowdim}) 
of $\gamma_1, \gamma_2, \ldots, \gamma_{2g}$ forms 
a basis as a right $\Z [\pi]$-module. 
The action of $\varphi \in \Mgb$ on $\Sgb$ is uniquely lifted on 
$\widetilde{\Sgb}$ so that $\widetilde{p}$ is fixed. 
It induces a right $\Z [\pi]$-equivariant 
isomorphism of $H_1 (\Sgb, \{p\};\Z [\pi])$. 
Suzuki showed that the matrix representation of 
this equivariant isomorphism under the above basis 
coincides with $r (\varphi)$. 

Next Suzuki considered an intersection pairing 
\[\langle \ \cdot,\cdot \ \rangle: 
H_1 (\Sgb, \{p\};\Z [\pi]) \times H_1 (\Sgb, \{p\};\Z [\pi]) 
\longrightarrow \Z [\pi]\]
on $H_1 (\Sgb, \{p\};\Z [\pi])$ called the 
{\it higher intersection number} 
in \cite{suz} by using Papakyriakopoulos' idea of 
{\it biderivations} \cite{papa} in $\Z [\pi]$. 
Note that Turaev \cite{turaev} also 
gave a construction similar to biderivations. 
Let $c_1, c_2$ be paths on $\widetilde{\Sgb}$ 
connecting two points of $f^{-1}({p})$. 
We take another base point $q \in \partial \Sgb$ and 
decompose $\partial \Sgb$ into two segments $A$ and $B$ as in 
Figure \ref{fig:decomposition}. 

\begin{figure}[htbp]
\begin{center}
\includegraphics{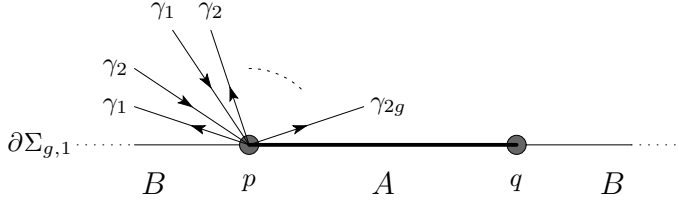}
\end{center}
\caption{Decomposition of $\partial \Sigma_{g,1}$}
\label{fig:decomposition}
\end{figure}

We slide $c_2$ along the lifts of $A$ so that the resulting path 
connects two points of $f^{-1}({q})$. Then we set 
\[\langle c_1, c_2 \rangle = \sum_{\gamma \in \pi} 
\widetilde{\mu}(c_1 \gamma, c_2) \gamma,\]
where $c_1 \gamma$ is the path obtained from $c_1$ by the right action 
of $\gamma$ and $\widetilde{\mu}(c_1 \gamma, c_2)$ 
is the usual intersection number 
of $c_1 \gamma$ and $c_2$ on $\widetilde{\Sgb}$. 
We can naturally extend this pairing of paths to 
the desired pairing of $H_1 (\Sgb, \{p\};\Z [\pi])$ so that 
\[\langle u f ,v \rangle = \overline{f} \langle u ,v \rangle, \qquad 
\langle u ,vf \rangle=\langle u  ,v \rangle f\]
holds for any $f \in \Z [\pi]$ and $u,v \in H_1 (\Sgb, \{p\};\Z [\pi])$. 
This pairing is clearly preserved by 
the action of $\mathcal{M}_{g,1}$. Then 
the twisted symplecticity is obtained by 
writing this invariance under our basis of 
$H_1 (\Sgb, \{p\};\Z [\pi])$. 
\begin{remark}\label{rem:decomp}
Sliding the path $c_2$ in the above procedure has 
the following homological meaning, which 
was pointed out in Turaev \cite{turaev}. 
In the source of the pairing $\langle \ \cdot,\cdot \ \rangle$, 
we identify the left $H_1 (\Sgb, \{p\};\Z [\pi])$ with 
$H_1 (\Sgb, A;\Z [\pi])$ by using the inclusion 
$(\Sgb,\{p\}) \hookrightarrow (\Sgb,A)$ and similarly the right with 
$H_1 (\Sgb, B;\Z [\pi])$ by 
\[(\Sgb,\{p\}) \hookrightarrow 
(\Sgb,A) \hookleftarrow (\Sgb, \{q\}) 
\hookrightarrow (\Sgb,B),\] 
which corresponds to the slide. 
Then we can take a homological intersection between the pairs 
$(\Sgb,A)$ and $(\Sgb,B)$ arising from Poincar\'e-Lefschetz duality. 
\end{remark}

\subsection{Non-faithfulness and decompositions}
\label{subsec:magnus_MCG_nonfaithful}
Here we focus on 
the Magnus representation 
$r_\mathfrak{a}:\mathcal{I}_{g,1} \to \mathrm{GL}(2g,\Z [H])$ 
for the Torelli group. 
We first mention the following fact first found by Suzuki: 
\begin{theorem}[Suzuki \cite{su}]\label{thm:Suzuki_notfaithful}
The Magnus representation for the Torelli group $\mathcal{I}_{g,1}$ 
is not faithful, namely $\ker r_\mathfrak{a} \neq \{1\}$, 
for $g \ge 2$.
\end{theorem}
\noindent
Suzuki's proof exhibits an example, which looks not so complicated 
but needs a long computation. After that 
he gave an improvement \cite{suz2} based on the 
topological interpretation of $r_\mathfrak{a}$. 
See also Perron \cite{perron}. 
Along this line, the following remarkable result was 
recently shown by Church-Farb: 
\begin{theorem}[Church-Farb \cite{cf}]\label{thm:ChurchFarb}
$\Ker r_\mathfrak{a}$ is not finitely generated. Moreover, 
$H_1 (\Ker r_\mathfrak{a})$ has infinite rank for $g \ge 2$.
\end{theorem}
\noindent
Note that their argument can be also applied to 
$I^2 A$. 

On the other hand, 
whether the Gassner representation, 
the corresponding representation for braids, 
is faithful or not is unknown 
for $n \ge 4$. When $n=3$, it was shown to be faithful by 
Magnus-Peluso \cite{magnus-peluso} 
(see also \cite[Theorem 3.15]{bi}). 

A decisive difference between the Magnus representation 
for $\Igb$ and the Gassner representation 
for $P_n$ appears 
in their {\it irreducible} decompositions. 
Here, the word ``irreducible'' means that 
there exist no invariant {\it direct summands} of 
$(\Z [H])^{2g}$ (or $(\Z [H_1])^n$), which is a slight 
abuse of terminology. 
It is easily checked that the Gassner representation has a 
1-dimensional trivial subrepresentation 
(see \cite[Lemma 3.11.1]{bi}). Moreover, 
Abdulrahim \cite{abd} showed 
by using a technique of complex specializations 
that the Gassner representation is the direct sum of the trivial 
representation and an $(n-1)$-dimensional irreducible 
representation. 

As for the Magnus representation for $\Igb$, 
Suzuki gave the following decomposition of $r_{\mathfrak{a}}$ after 
extending the target: 
\begin{theorem}[Suzuki \cite{suzuki_irred}]
Let 
\[R=\Z [\gamma_1^{\pm1}, \ldots, \gamma_{2g}^{\pm1}, 
1/(1-\gamma_{g+1}), \ldots, 1/(1-\gamma_{2g})].\] 
Then the Magnus representation 
\[r_\mathfrak{a} : \Igb \longrightarrow \mathrm{GL}(2g,R)\] 
for the Torelli group with an extension of its target has a 
$1$-dimensional subrepresentation which is not a direct summand. 
Moreover the quotient $(2g-1)$-dimensional representation has 
a $(2g-2)$-dimensional subrepresentation which is not 
a direct summand and whose quotient is a 
$1$-dimensional trivial representation. 
\end{theorem}
\noindent
In the proof, Suzuki gave a matrix $P \in \mathrm{GL}(2g,R)$ such that 
\begin{equation}\label{eq:irred}
P^{-1} r_\mathfrak{a}(\varphi) P =
\left(\begin{array}{c|ccc|c}
1 & & \ast & & \ast \\
\hline
0 &&&&\\
\vdots & & r_\mathfrak{a}'(\varphi) & & \ast \\
0 &&&&\\
\hline
0 & 0 & \cdots & 0 & 1
\end{array}\right)
\end{equation}
\noindent
holds for any $\varphi \in \Igb$ 
with an irreducible representation 
$r_\mathfrak{a}':\Igb \to \mathrm{GL}(2g-2,R)$. 

\begin{remark}
Here we comment on the topological meaning of the above 
decomposition. For simplicity, 
we use the quotient field 
$\mathcal{K}_H:= \Z [H](\Z [H] -\{0\})^{-1}$ of $\Z [H]$ and 
consider $r_{\mathfrak{a}}: \Igb \to \mathrm{GL}(2g,\mathcal{K}_H)$. 
The homology exact sequence shows that 
\[0 \longrightarrow H_1 (\Sgb;\mathcal{K}_H) 
\longrightarrow H_1 (\Sgb,\{p\};\mathcal{K}_H) 
\longrightarrow H_0 (\{p\};\mathcal{K}_H) \longrightarrow 0\]
is exact and it can be written as 
\[0 \longrightarrow \mathcal{K}_H^{2g-1}
\longrightarrow \mathcal{K}_H^{2g} 
\longrightarrow \mathcal{K}_H \longrightarrow 0.\]
The map $\mathcal{K}_H^{2g} \to \mathcal{K}_H$ 
coincides with 
$\partial_1: C_1 (\Sgb,\{p\};\mathcal{K}_H) \to 
C_0 (\Sgb,\{p\};\mathcal{K}_H)$. 
Now the representation $r_{\mathfrak{a}}$ 
works as a transformation of $H_1 (\Sgb,\{p\};\mathcal{K}_H) \cong 
\mathcal{K}_H^{2g}$ and we can check that 
it preserves $H_1 (\Sgb;\mathcal{K}_H) \cong 
\mathcal{K}_H^{2g-1}$, namely we have a subrepresentation as 
a transformation of $H_1 (\Sgb;\mathcal{K}_H)$. Taking 
a basis of $\mathcal{K}_H^{2g}$ from 
ones of $\mathcal{K}_H^{2g-1}$ and $\mathcal{K}_H$, we obtain 
the first decomposition corresponding 
to the upper left $(2g-1)$-matrix of (\ref{eq:irred}). One more step 
is obtained by finding the vector 
\begin{align*}
&\begin{pmatrix}
\overline{\dis\deriv{\zeta}{\gamma_1}} & 
\overline{\dis\deriv{\zeta}{\gamma_2}} 
& \cdots & 
\overline{\dis\deriv{\zeta}{\gamma_{2g}}} \end{pmatrix}^T \\
=&\begin{pmatrix}
1-\gamma_{g+1}^{-1} & \cdots & 1-\gamma_{2g}^{-1} & 
\gamma_{1}^{-1}-1 & \cdots & \gamma_{g}^{-1}-1 \end{pmatrix}^T
\end{align*}
\noindent
to be an invariant vector belonging to $\Ker \partial_1 = 
H_1 (\Sgb;\mathcal{K}_H) \cong 
\mathcal{K}_H^{2g-1}$. A similar observation can be applied to 
the Gassner representation for $P_n$. 
In this case, however, the invariant vector 
corresponding to the trivial subrepresentation does not 
belong to the subspace $\Ker \partial_1$, 
so that we cannot obtain an $(n-2)$-dimensional 
subrepresentation from this.  
\end{remark}

The following observation might be useful 
for further comparison of the two representations. 
Let $L$ be a pure braid with $g$ strings. 
Consider a closed tubular neighborhood of 
the union of the loops $\gamma_{g+1}, \gamma_{g+2}, \ldots, 
\gamma_{2g}$ in $\Sgb$ (see Figure \ref{fig:generator}) 
to be the image of 
an embedding $\iota:\Sg{0}{g+1} \hookrightarrow \Sgb$ of 
a $g$ holed disk $\Sg{0}{g+1}$ as in Figure \ref{fig:ht}. 

\begin{figure}[htbp]
\begin{center}
\includegraphics[width=.95\textwidth]{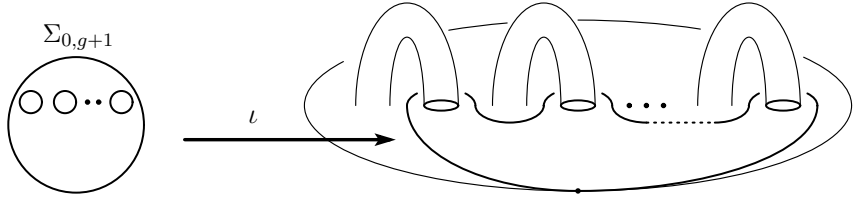}
\end{center}
\caption{The embedding $\iota:\Sigma_{0,g+1} \hookrightarrow 
\Sigma_{g,1}$}
\label{fig:ht}
\end{figure}

\noindent
Since $P_g$ can be regarded as a subgroup of $\Mg{0}{g+1}$, 
we have an injective homomorphism $I: P_g \hookrightarrow \Mgb$ 
by a method similar to that mentioned 
in Remark \ref{rem:extension}. 
The construction of the map $I$ is due to Oda \cite{oda} and 
Levine \cite{le3} (see also Gervais-Habegger \cite{GH}). 
As in the following way, 
we can compare the restriction of 
the universal Magnus representation $r$ for $\Mgb$ 
to $P_g$ with that for $\Aut F_g= \Aut (\pi_1\Sg{0}{g+1})$ 
denoted here by $r_G:P_g \to \mathrm{GL}(g, \Z [\pi_1\Sg{0}{g+1}])$. 
Note that we are now identifying $\pi_1 \Sg{0}{g+1}$ 
with the subgroup of $\pi$ generated by 
$\gamma_{g+1}, \ldots, \gamma_{2g}$. 
By construction, we obtain the following: 
\begin{proposition}\label{prop:conn}
For any pure braid $L \in P_g$, 
$r (I(L)) = 
\left(\begin{array}{cc}
I_{g} & 0_g \\ \ast & r_{G} (L)
\end{array}\right)$. 
\end{proposition}
\noindent
Here we must remark that 
the embedding $P_g \hookrightarrow \Mgb$ has an 
ambiguity due to framings, which count how many times 
one applies Dehn twists along each of the loops parallel to 
the inner boundary of $\Sigma_{0,g+1}$. 
However we can check that the lower right part of $r (I(L))$ 
is independent of the framings. 

While the entire image $I(P_g)$ is not included in $\Igb$, 
we can easily check that $I([P_g,P_g]) \subset \Igb$. 
Suppose $L \in P_g$ 
is in the kernel of the Gassner representation. 
Then $L \in [P_g,P_g]$
(see \cite[Theorem 3.14]{bi}), so that $I(L) \in \Igb$. 
The symplecticity of 
$r$ shows that the lower left part of $r_\mathfrak{a} (I(L))$ is $O$. 
Consequently, we have observed that for $L \in P_g$, 
$L$ is in the kernel of the Gassner representation if and only if 
$I(L)$ is in the kernel of 
the Magnus representation $r_\mathfrak{a}$ for $\Igb$.

\subsection{Determinant of the Magnus representation}\label{subsec:detmag1}

Now we focus on the Magnus representation 
$r_\mathfrak{a} : \Mgb \to \mathrm{GL}(2g,\Z [H])$ as a crossed homomorphism, 
whose importance was first pointed out by Morita. 
We put $k:= \det \circ r_\mathfrak{a} : \mathcal{M}_{g,1} \to 
(\Z [H])^\times = \pm H$, where $\pm H$ is regarded as 
the multiplicative group of monomials in $\Z [H]$. 
The image of $k$ is included in 
$H$ since $\mathfrak{t} (k (\varphi))=\det (\sigma (\varphi))=1$. 
We here turn the group $H$ as a multiplicative group into 
the additive one as usual. 
\begin{theorem}[Morita \cite{morita_jac1, mo9}]\label{thm:H1MgH}
$H^1 (\Mgb ; H) \cong \Z$ for $g \ge 2$ and it is generated 
by $k$.
\end{theorem}
\noindent
This cohomology class, which has many natural representatives 
as crossed homomorphisms arising from various contexts, 
is referred as to the {\it Earle class} in 
Kawazumi's chapter \cite{kawazumi} of the second volume of this handbook. 

Consider the composition 
\[H^1 (\mathcal{M}_{g,1};H) \otimes H^1 (\mathcal{M}_{g,1};H) 
\stackrel{\cup}{\longrightarrow} 
H^2 (\mathcal{M}_{g,1};H \otimes H) 
\stackrel{\mu}{\longrightarrow} 
H^2 (\mathcal{M}_{g,1})\]
and apply it to 
$k \otimes k$, where $\cup$ denotes the cup 
product and $\mu$ denotes the 
map which contracts the coefficient $H \otimes H$ to $\Z$ by 
the intersection pairing on $H$. 
At the cocycle level, $\mu (k \cup k)$ is given by 
\[\mu (k \cup k)([\varphi | \psi])=
\mu (k(\varphi), \varphi (k(\psi)))\]
for $\varphi$, $\psi \in \mathcal{M}_{g,1}$. 
Then the following was shown by Morita: 
\begin{theorem}[Morita \cite{morita_jac2}]\label{thm:det_to_e1}
For $g \ge 2$, we have 
\[\mu (k \cup k) = -e_1 \in H^2 (\mathcal{M}_{g,1}),\]
where $e_1$ is the first {\it Miller-Morita-Mumford} class. 
\end{theorem}
\noindent
With Meyer's results \cite{me}, 
Harer \cite{harer} (see also Korkmaz-Stipsicz \cite{ks}) showed that 
$H^2 (\Mgb) \cong \Z$ for $g \ge 3$ and it is known that 
$e_1$ is twelve times the generator up to sign. 
We refer to Morita's paper \cite{mo1} and 
Kawazumi's chapter \cite{kawazumi} for the definition and 
generalities on 
the Miller-Morita-Mumford classes. 
\begin{remark}
Satoh \cite{satoh_twisted_H} 
proved that $H^1 (\Aut F_n;H_1 (F_n)) \cong \Z$ for $n \ge 3$ 
and we can check 
that the generator is represented by the crossed homomorphism
\[|\det r_\mathfrak{a}|: \Aut F_n \longrightarrow H_1 (F_n)\] 
sending $\varphi \in \Aut F_n$ to $h \in H_1 (F_n)$ with 
$\det (r_\mathfrak{a} (\varphi))=\pm h \in \pm H_1 (F_n)$. 
In particular, the pullback map 
$H^1 (\Aut F_{2g};H) \to H^1 (\Mgb;H)$ is an isomorphism 
under an identification $H_1 (F_{2g}) \cong H$. 
However, there exist 
no corresponding statement to Theorem \ref{thm:det_to_e1} since 
we do not have a natural intersection pairing on 
$H_1 (F_n)$. In fact, Gersten \cite{gersten} 
showed that $H^2(\Aut F_n) \cong \Z/2\Z$ for $n \ge 5$. 
See Kawazumi \cite[Theorem 7.2]{kawazumi_Magnus} for more details.
\end{remark}

\subsection{The Johnson filtration and Magnus representations}
\label{subsec:JohnsonMorita}

A filtration of $\Mgb$ is obtained by taking intersections with 
the filtration $\{I^k A_{2g}\}_{k=0}^\infty$ of 
$\Aut \pi=\Aut F_{2g}$. 
However, as far as the author knows, nothing is known about 
$I^k A_{2g} \cap \Mgb$ for $k \ge 3$. 
This reflects the difficulty in treating the derived series 
of $F_{2g}$. In the study of $\Mgb$, instead,  
the filtration arising from the 
{\it lower central series} \index{lower central series} of $\pi$ 
is frequently used. Recall that the lower central series 
\[\Gamma^1 (G) :=G \supset \Gamma^2 (G) \supset \Gamma^3 (G) 
\supset \cdots\]
of a group $G$ is defined by 
$\Gamma^{k+1} (G)=[G,\Gamma^{k} (G)]$ for $k \ge 1$. 
The group $\Gamma^k (G)$ is a characteristic subgroup of $G$. 
We denote the $k$-th {\it nilpotent quotient} \index{nilpotent quotient} 
$G/(\Gamma^k (G))$ of $G$ by $N_k(G)$. 
Here $N_2 (G)=H_1 (G)$. 

Let $\mathfrak{q}_k: \pi \to N_k (\pi)$ be the natural projection. 
Consider 
the composition
\[\sigma_k: \Mgb \stackrel{\sigma}{\hookrightarrow} \Aut \pi 
\to \Aut (N_k (\pi))\]
of the Dehn-Nielsen embedding and the map 
induced from $\mathfrak{q}_k$. 
This defines a filtration
\[\Mgb[1]:=\Mgb \supset \Mgb[2] \supset \Mgb[3] \supset \Mgb[4] 
\supset \cdots\]
called the {\it Johnson filtration} 
\index{Johnson filtration!for $\mathcal{M}_{g,1}$} of $\Mgb$ 
by setting $\Mgb[k]:=\Ker \sigma_k$. 
Note that $\Mgb[2]= \Igb$. 
The corresponding filtration for $\Aut F_n$ 
was studied earlier by Andreadakis \cite{andrea} and we here call it 
the {\it Andreadakis filtration} \index{Andreadakis filtration}. 
Since $\pi$ is known to be 
residually nilpotent, namely 
$\dis\bigcap_{k \ge 1}^\infty \Gamma^k (\pi) =\{1\}$, 
we have $\dis\bigcap_{k \ge 1}^\infty \Mgb[k]=\{1\}$. 
The Andreadakis filtration of $\Aut F_n$ has a 
similar property. 

In the above cited paper, Andreadakis constructed an exact sequence 
\[1 \longrightarrow \Hom (H,\Gamma^k (\pi)/\Gamma^{k+1}(\pi)) 
\longrightarrow \Aut (N_{k+1} (\pi))\longrightarrow 
\Aut (N_k (\pi)) \longrightarrow 1,\]
from which we obtain a homomorphism 
\[\tau_k:=\sigma_{k+2}|_{\Mgb[k+1]}: \Mgb[k+1] \longrightarrow 
\Hom (H,\Gamma^{k+1} (\pi)/\Gamma^{k+2}(\pi))\]
with $\Ker \tau_k=\Mgb[k+2]$, 
called the $k$-th {\it Johnson homomorphism} \index{Johnson homomorphism} 
for $k \ge 1$. 
That is, 
the successive quotients of the Johnson filtration are 
described by the Johnson homomorphisms. 
We refer to Johnson's survey \cite{jo4} for his original 
description and 
to the chapters of Morita \cite{morita_survey} and 
Habiro-Massuyeau \cite{habiromassuyeau} for the details of 
these homomorphisms. 
The theory of the Johnson homomorphisms has been studied 
intensively by many researchers and is now highly 
developed (see Morita \cite{morita_beyond} 
for example). 

\begin{remark}\label{rem:trace}
It is known that there exist non-tame automorphisms of 
$N_k (\pi)$. In fact, Bryant-Gupta \cite{bg} showed 
that if $n \ge k-2$, 
$\Aut (N_k (F_n))$ is generated by the tame automorphisms 
and {\it one} non-tame automorphism written explicitly. 
It follows from Andreadakis' exact sequence 
that we may use $\mathrm{Coker}\, \tau_k$ 
to detect the non-tameness. 
Morita \cite{mo} studied $\mathrm{Coker}\, \tau_k$ 
by using his {\it trace  maps} and showed 
they are non-trivial for general $k$. 
\end{remark}

Corresponding to the Johnson filtration, we have 
a crossed homomorphism 
\begin{align*}
r_{\mathfrak{q}_k}&: 
\Mgb \longrightarrow \mathrm{GL}(2g,\Z [N_k(\pi)]),\\
\intertext{whose restriction to $\Mgb[k]$ is a homomorphism}
r_{\mathfrak{q}_k}&: \Mgb[k] \longrightarrow \mathrm{GL}(2g,\Z [N_k(\pi)])
\end{align*}
\noindent
for each $k \ge 2$. 
Note that $r_{\mathfrak{q}_2}=r_\mathfrak{a}$, 
the Magnus representation for $\Igb$. 
\begin{problem}
Determine whether $r_{\mathfrak{q}_k}$ 
is faithful or not for $k \ge 3$. Also, 
determine the image of 
$r_{\mathfrak{q}_k}$ for $k \ge 2$.
\end{problem}
As for the relationship between the Johnson filtration and 
Magnus representations, Morita \cite{mo9} gave a method 
for computing $\tau_{k-1} (\varphi)$ from 
$r_{\mathfrak{q}_k} (\varphi)$ 
for $\varphi \in \Mgb$. For example, we can easily 
calculate $\tau_1 (\varphi)$ from 
$\det (r_{\mathfrak{q}_2} (\varphi))$. 
Note also that Morita's trace maps mentioned above are 
highly related to $\det r_{\mathfrak{q}_2}$. 

Here we pose the converse as a problem. 
\begin{problem}
Describe explicitly how we can reproduce 
$r_{\mathfrak{q}_k}$ from the ``totality'' of 
the Johnson homomorphisms. 
\end{problem}
\noindent
Suzuki \cite{suz} showed that 
$\Mgb[k] \not\subset \Ker r_{\mathfrak{q}_2}$ 
for every $k \ge 2$ by using 
the topological description of $r_{\mathfrak{q}_2}$. 

Another approach to the Johnson homomorphisms using 
the {\it Magnus expansion} is studied by Kawazumi \cite{kawazumi_Magnus} 
(see also the chapters by Kawazumi \cite{kawazumi} and 
Habiro-Massuyeau \cite{habiromassuyeau} in this handbook). 
It would be interesting to compare his construction 
with Magnus representations.

\subsection{Applications to three-dimensional topology}\label{subsec:3-dim}
We close the first part of this survey by briefly mentioning some 
relationships between the Magnus representation $r_{\mathfrak{q}_2}$ 
and three-dimensional topology. 
It also serves as an original model for the results discussed in the 
second part. 

There exist several methods for making a three-dimensional manifold 
from an element of $\Mgb$ such as Heegaard splittings, mapping tori and 
open book decompositions. We here recall the last two. 

For a diffeomorphism $\varphi$ of $\Sgb$ fixing $\partial \Sgb$ pointwise, 
the {\it mapping torus} \index{mapping torus} 
$T_\varphi^\partial$ of $\varphi$ is defined as 
\[T_\varphi^\partial := \Sgb \times [0,1]/((x,1) = (\varphi(x),0)) \quad 
x \in \Sgb.\]
The manifold $T_\varphi^\partial$ is a $\Sgb$-bundle over $S^1$. 
We fill the boundary of $T_\varphi^\partial$ by 
a solid torus $S^1\times D^2$, so that each disk $\{x\} \times D^2$ caps 
a fiber $\Sgb \times \{t\}$. Then we obtain a closed 
3-manifold $T_\varphi$ 
also called the {\it mapping torus} of $\varphi$. 
If we change the attaching of $S^1 \times D^2$ so that 
each disk $\{x\} \times D^2$ caps 
$\{q\} \times S^1 \subset (\partial \Sgb) \times S^1 = 
\partial T_\varphi^\partial$, then we have 
another closed 3-manifold $C_\varphi$ 
called  
the {\it closure} of $\varphi$. 
We also say that $C_\varphi$ has an \index{open book decomposition} 
{\it open book decomposition}. The core $S^1 \times \{(0,0)\}$ 
of the glued solid torus in $C_\varphi$ is 
called the {\it binding} and $\varphi$ is called the {\it monodromy}. 
Note that the above constructions of $T_\varphi^\partial$, $T_\varphi$ and 
$C_\varphi$ depend only on the isotopy class of $\varphi$, so that 
they are well-defined for each element of $\Mgb$. More precisely, 
they depend on the conjugacy class in $\Mgb$. 
From the presentation 
$\pi=\langle \gamma_1, \gamma_2, \ldots \gamma_{2g} \rangle$ of $\pi$, 
we can easily obtain 
\begin{align*}
\pi_1 T_\varphi^\partial &= 
\langle \gamma_1, \gamma_2, \ldots \gamma_{2g}, \lambda \mid 
\gamma_i \lambda \varphi(\gamma_i)^{-1} \lambda^{-1} 
(1\le i \le 2g) \rangle,\\
\pi_1 T_\varphi &= 
\langle \gamma_1, \gamma_2, \ldots \gamma_{2g}, \lambda \mid 
\textstyle\prod_{j=1}^g [\gamma_j,\gamma_{g+j}], 
\gamma_i \lambda \varphi(\gamma_i)^{-1} \lambda^{-1}, 
(1\le i \le 2g) \rangle,\\
\pi_1 C_\varphi &= \langle \gamma_1, \gamma_2, \ldots \gamma_{2g} 
\mid \gamma_i \varphi(\gamma_i)^{-1} (1\le i \le 2g) \rangle, 
\end{align*}
\noindent
where $\lambda$ corresponds to the loop $\{p\} \times S^1$ 
in $T_\varphi^\partial$ and $T_\varphi$. 

The {\it $($multi-variable$)$ Alexander polynomial} $\Delta_G$ 
is an invariant of finitely presentable groups. 
It can be regarded as an invariant of 
compact manifolds by considering their fundamental groups. 
For a finitely presentable group $G$, 
the polynomial $\Delta_G$ is computed from 
the Alexander module 
\[\mathcal{A}^\Z (G):= H_1 (G;\Z [H_1 (G)])\]
by a purely algebraic procedure. We here omit the details and refer 
to Turaev's book \cite{tu2} 
for the definition and its relationship to torsions. 
For a knot group $G(K)$, the polynomial $\Delta_{G(K)}$ with 
$\lambda$ replaced by $t$ 
coincides with 
the Alexander polynomial $\Delta_K (t)$ of $K$ mentioned 
in Example \ref{ex:alexanderpolyn}. 

When $\varphi \in \Igb$, $H=H_1 (\Sgb)$ is naturally embedded in 
$H_1 (T_\varphi^\partial)$, $H_1 (T_\varphi)$ and $H_1 (C_M)$. 
Then we can easily check that the Magnus representation 
$r_{\mathfrak{q}_2} (\varphi)$ 
can be used to describe the multi-variable Alexander polynomials of 
$T_\varphi^\partial$, $T_\varphi$ and $C_\varphi$. For example, 
we have 
\[\Delta_{\pi_1 T_\varphi} \doteq 
\frac{\det (\lambda I_{2g}-\overline{r_{\mathfrak{q}_2} 
(\varphi)})}{(1-\lambda)^2}
\in \Z [H_1 (T_\varphi)] = \Z [H \times \langle \lambda \rangle],\]
where $\doteq$ means that the equality holds up to multiplication by 
monomials. A generalization of this formula was given by 
Kitano-Morifuji-Takasawa \cite{kmt} in their study of 
$L^2$-torsion invariants of mapping tori. 

Another application is given when $C_\varphi=S^3$. In this case, 
we focus on the binding, which gives a knot $K$ in $S^3$, 
of the open book decomposition. Such a $K$ is 
called a \index{fibered knot}{\it fibered knot}. We can check that the 
Alexander polynomial $\Delta_K (t)$ is given by 
\begin{equation}\label{eq:alexander_fibered}
\Delta_K (t) \doteq \det (I_{2g}-t \cdot \sigma_2(\varphi)) = 
\det (I_{2g}-t \cdot r_{\mathfrak{t}} (\varphi)).
\end{equation}
\noindent
Since $H$ collapses to the trivial group in $H_1 (E(K)) \cong \Z$, 
we cannot readily have a formula which generalizes 
(\ref{eq:alexander_fibered}) by 
using $r_{\mathfrak{q}_2}$. 
In Section \ref{subsec:higher-order}, 
we discuss the details about this in a more general situation.

\section{Homology cylinders}\label{sec:HC}
Now we start the second half of our survey. In this section, we introduce 
homology cylinders over a surface 
and give a number of examples. We also describe 
how Johnson homomorphisms are extended to the monoid and group of 
homology cylinders. 

\subsection{Definition and examples}\label{subsec:defHC}
The definition of homology cylinders goes back to  
Goussarov \cite{gou}, Habiro \cite{habiro}, 
Garoufalidis-Levine \cite{gl} and Levine \cite{levine} in their study 
of finite type invariants of 3-manifolds. 
Strictly speaking, the definition below is closer to that in 
\cite{gl} and \cite{levine}. 
Note that homology cylinders are called ``homology cobordisms'' 
in the chapter of Habiro-Massuyeau \cite{habiromassuyeau}, where 
the terminology ``homology cylinders'' is used for a more restricted 
class of 3-manifolds. 

\begin{definition}\label{def:HC}
A {\it homology cylinder\/} \index{homology cylinder} 
{\it over} $\Sigma_{g,n}$ 
consists of a compact oriented 3-manifold $M$ 
with two embeddings $i_{+}, i_{-}: \Sigma_{g,n} \hookrightarrow \partial M$, 
called the {\it markings}, such that:
\begin{enumerate}
\renewcommand{\labelenumi}{(\roman{enumi})}
\item
$i_{+}$ is orientation-preserving and $i_{-}$ is orientation-reversing; 
\item 
$\partial M=i_{+}(\Sigma_{g,n})\cup i_{-}(\Sigma_{g,n})$ and 
$i_{+}(\Sigma_{g,1})\cap i_{-}(\Sigma_{g,1})
=i_{+}(\partial\Sigma_{g,n})=i_{-}(\partial\Sigma_{g,n})$;
\item
$i_{+}|_{\partial \Sigma_{g,n}}=i_{-}|_{\partial \Sigma_{g,n}}$; 
\item
$i_{+},i_{-} : H_{*}(\Sigma_{g,n})\to H_{*}(M)$ 
are isomorphisms, namely 
$M$ is a {\it homology product} over $\Sigma_{g,n}$. 
\end{enumerate}
We denote a homology cylinder by $(M,i_{+},i_{-})$ or simply $M$. 
\end{definition}
Two homology cylinders $(M,i_+,i_-)$ and $(N,j_+,j_-)$ over $\Sg{g}{n}$ 
are said to be {\it isomorphic} if there exists 
an orientation-preserving diffeomorphism $f:M \xrightarrow{\cong} N$ 
satisfying $j_+ = f \circ i_+$ and $j_- = f \circ i_-$. 
We denote by $\Cg{g}{n}$ the set of all isomorphism classes 
of homology cylinders over $\Sg{g}{n}$. 
We define a product operation on $\Cg{g}{n}$ by 
\[(M,i_+,i_-) \cdot (N,j_+,j_-)
:=(M \cup_{i_- \circ (j_+)^{-1}} N, i_+,j_-)\]
for $(M,i_+,i_-)$, $(N,j_+,j_-) \in \Cg{g}{n}$, 
which endows $\Cg{g}{n}$ with a monoid 
\index{monoid of homology cylinders} structure. 
Here the unit is 
$(\Sigma_{g,n} \times [0,1], \mathrm{id} \times 1, \mathrm{id} \times 0)$, 
where collars of $i_+ (\Sigma_{g,n})=(\mathrm{id} \times 1) (\Sigma_{g,n})$ 
and $i_- (\Sigma_{g,n})=(\mathrm{id} \times 0)(\Sg{g}{n})$ 
are stretched half-way along $(\partial \Sigma_{g,n}) \times [0,1]$ 
so that $i_+ (\partial \Sigma_{g,n})=i_- (\partial \Sigma_{g,n})$. 

\begin{example}\label{ex:mgtocg}
For each diffeomorphism $\varphi$ of 
$\Sigma_{g,n}$ which fixes $\partial \Sigma_{g,n}$ pointwise, 
we can construct a homology cylinder by setting 
\[(\Sigma_{g,n} \times [0,1], \mathrm{id} \times 1, 
\varphi \times 0)\]
with the same treatment of the boundary as above. 
It is easily checked that the isomorphism class of 
$(\Sigma_{g,n} \times [0,1], \mathrm{id} \times 1, \varphi \times 0)$ 
depends only on the (boundary fixing) 
isotopy class of $\varphi$ and that 
this construction gives a monoid homomorphism 
from the mapping class group $\mathcal{M}_{g,n}$ 
to $\mathcal{C}_{g,n}$. In fact, it is 
an injective homomorphism 
(see Garoufalidis-Levine \cite[Section 2.4]{gl}, Levine 
\cite[Section 2.1]{levine}, Habiro-Massuyeau's chapter 
\cite[Section2.2]{habiromassuyeau} 
and \cite[Proposition 2.3]{gs08}). 
\end{example}
\noindent
By this example, 
we may regard $\Cg{g}{n}$ as an {\it enlargement} of $\Mg{g}{n}$, 
where the usage of the word ``enlargement'' comes from the title of 
Levine's paper \cite{levine}. In fact, we will see that 
the Johnson homomorphisms and Magnus representations for 
$\Mg{g}{n}$ are naturally extended. 

In \cite{gl}, Garoufalidis-Levine further 
introduced {\it homology cobordisms} of homology cylinders, 
which give an equivalence relation among homology cylinders. 
\begin{definition}
Two homology cylinders $(M,i_+,i_-)$ and $(N,i_+,i_-)$ over 
$\Sigma_{g,n}$ are said to be {\it homology cobordant} 
if there exists a compact oriented smooth 4-manifold $W$ such that: 
\begin{enumerate}
\item $\partial W = M \cup (-N) /(i_+ (x)= j_+(x) , \,
i_- (x)=j_-(x)) \quad x \in \Sigma_{g,n}$; 
\item The inclusions $M \hookrightarrow W$, $N \hookrightarrow W$ 
induce isomorphisms on the integral homology.  
\end{enumerate}
\end{definition}
\noindent
We denote by $\mathcal{H}_{g,n}$ 
the quotient set of $\mathcal{C}_{g,n}$ with respect 
to the equivalence relation of homology cobordism. 
The monoid structure of $\mathcal{C}_{g,n}$ induces 
a group structure of $\mathcal{H}_{g,n}$. 
It is known that $\Mg{g}{n}$ can be embedded in $\Hg{g}{n}$ 
(see Cha-Friedl-Kim \cite[Section 2.4]{cfk}). 
We call $\Hg{g}{n}$ the {\it homology cobordism group} 
\index{homology cobordism group of homology cylinders} of 
homology cylinders. 

\begin{example}
Homology cylinders were originally introduced in the theory of 
clasper (clover) surgery and finite type invariants of 3-manifolds 
due to Goussarov \cite{gou} and 
Habiro \cite{habiro} independently. 
Since 
clasper surgeries do not change the homology of 3-manifolds, 
the theory fits well to the setting of homology cylinders. 
It is known that every homology cylinder is obtained from the 
trivial one by doing some clasper surgery and then changing 
the markings by the mapping class group 
(see Massuyeau-Meilhan \cite{mm}). While clasper surgery 
brings a quite rich structure to $\Cg{g}{n}$, 
here we do not take it up in detail. 
See the chapter of Habiro-Massuyeau \cite{habiromassuyeau} and 
references in it. 

Another approach from the theory of finite type invariants 
to homology cylinders was 
obtained by Andersen-Bene-Meilhan-Penner \cite{abmp}.
\end{example}

The following constructions give us direct methods for obtaining 
homology cylinders whose underlying 3-manifolds are not 
product manifolds. 
\begin{example}\label{ex:sphere}
For each homology 3-sphere $X$, the connected sum 
$((\Sigma_{g,n} \times [0,1]) \# X, \Id \times 1, \Id \times 0)$ 
gives a homology cylinder. It can be checked that 
this correspondence is an injective monoid homomorphism from 
the monoid $\theta_\Z^3$ of all (integral) homology 3-spheres 
whose product is given by connected sum to $\Cg{g}{n}$. 
In fact, it induces isomorphisms 
$\theta_\Z^3 \cong \mathcal{C}_{0,1} \cong \Cg{0}{0}$. 
Moreover, this construction is compatible with homology cobordisms, so 
that we have a homomorphism from the homology cobordism group $\Theta_\Z^3$ to 
$\Cg{g}{n}$, which is also shown to be injective 
(see Cha-Friedl-Kim \cite[Proof of Theorem 1.1]{cfk}). 
It is a challenging problem to extract new information on $\Theta_\Z^3$ 
from the theory of homology cylinders. 
At present, no result has been obtained. 
\end{example}

\begin{example}[Levine \cite{le3}]\label{sgtocg}
A string link is a generalization of a braid defined by Habegger-Lin 
\cite{habe_lin}. 
While we omit here the definition, the difference between 
the two notions is 
clear from Figure \ref{fig:braid-string}. 

\begin{figure}[htbp]
\begin{center}
\includegraphics{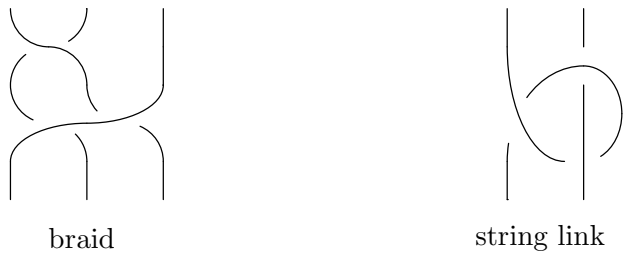}
\end{center}
\caption{Braid and string link}
\label{fig:braid-string}
\end{figure}

From a pure string link $L \subset D^2 \times [0,1]$ with 
$g$ strings, 
we can construct a homology cylinder as follows. 
Recall the embedding $\iota: \Sg{0}{g+1} \hookrightarrow \Sgb$ of 
a $g$ holed disk $\Sg{0}{g+1} \subset D^2$ 
in Section \ref{subsec:magnus_MCG_symp} 
and Figure \ref{fig:ht}. 
Let $C$ be the complement of an open tubular neighborhood of $L$ in 
$D^2 \times [0,1]$. For any choice of a framing of $L$, 
a homeomorphism $h:\partial C \xrightarrow{\cong} 
\partial (\iota(\Sg{0}{g+1}) \times [0,1])$ is fixed. 
(Note that a string link with a framing itself can be regarded as 
a homology cylinder over $\Sg{0}{g+1}$.) 
Then the manifold $M_L$ obtained from $\Sgb \times [0,1]$ 
by removing $\iota(\Sg{0}{g+1}) \times [0,1]$ and regluing 
$C$ by $h$ becomes a homology cylinder with the same marking as 
the trivial homology cylinder. 
This construction can be seen as a generalization of 
the embedding $P_g \hookrightarrow \Mgb$ in 
Section \ref{subsec:magnus_MCG_symp} and it 
gives an injective monoid homomorphism 
from the monoid of pure string links to $\Cgb$. 
Moreover it induces an injective group homomorphism from the 
concordance group of pure string links with $g$ strings to 
$\Hgb$. 

Habegger \cite{habe} gave another construction of homology cylinders 
from pure string links. 
\end{example}

\begin{example}[\cite{gs08,gs10}]\label{ex:HFknot}
Let $K$ be a knot in $S^3$ with a Seifert surface $R$ of genus $g$. 
By cutting open the knot exterior $E(K)$ along $R$, we obtain a 
manifold $M_R$. The boundary $\partial M_R$ is the union of two copies of 
$R$ glued along their boundary circles, which are just the knot $K$. 
The pair $(M_R,K)$ is called 
the {\it complementary sutured manifold} of $R$. 
We can check that the following properties are equivalent to each other: 
\begin{itemize}
\item[(a)] The Alexander polynomial $\Delta_K (t)$ is monic and 
its degree is equal to twice the genus of $g=g(K)$ of $K$; 

\item[(b)] The Seifert matrix $S$ of any minimal genus Seifert surface $R$ 
of $K$ is invertible over $\mathbb{Z}$; 

\item[(c)] The complementary sutured manifold $(M_R,K)$ 
for any minimal genus Seifert surface $R$ 
is a homology product over $R$. 
\end{itemize}
We call a knot having the above properties 
a \index{homologically fibered knot}{\it homologically fibered knot}, 
where the name 
comes from the fact that fibered knots satisfy them. 
Thus if we fix an identification 
of $\Sgb$ with $R$ for a homologically fibered knot, 
we obtain a homology cylinder over $\Sgb$. 
Note that aside from the name, 
the equivalence of the above conditions (a), (b), (c) 
was mentioned in Crowell-Trotter \cite{ct}. 
There exists 
a similar discussion for links. 
\end{example}

We close this subsection by two observations on connections between 
homology cylinders and 
the theory of 3-manifolds. 

First, the constructions of closed 3-manifolds mentioned in 
Section \ref{subsec:3-dim} have their analogue for homology cylinders. 
For example, the {\it closure} $C_M$ of a homology cylinder 
$(M,i_+,i_-) \in \Cgb$ is defined as 
\[C_M := M/(i_+ (x)= i_-(x)) \quad x \in \Sgb.\]
By a topological consideration, we see that 
this construction is the same as gluing 
$\Sgb \times [0,1]$ to $M$ along their boundaries and 
also as the description of Habiro-Massuyeau 
\cite[Definition 2.7]{habiromassuyeau}. 
The closure construction is 
compatible with the homology cobordism relation, 
denoted by H-cob, namely 
we have the following commutative diagram:
\begin{center}
{\small
\vspace{8pt}
\hspace{0pt}
\SelectTips{cm}{}
\xymatrix{
\displaystyle\bigsqcup_{g \ge 0}\, \Cgb 
\ar@{>>}[rr]^{\hspace{-25pt} \mbox{\scriptsize{closing}}} \ar@{>>}[d]& &
\mbox{\{closed 3-manifolds\}} \ar@{>>}[d]\\
\displaystyle\bigsqcup_{g \ge 0}\, \Hgb 
\ar@{>>}[rr]^{\hspace{-40pt} 
\mbox{\scriptsize{closing}}}& & 
\{\mbox{closed 3-manifolds}\}/
\mbox{(H-cob.)}
}}
\end{center}
\noindent
Therefore, roughly speaking, $\Hgb$ might be regarded 
as a group structure on 
the set of homology cobordism classes of closed 3-manifolds. 
We have a similar discussion for clasper surgery equivalence. 

Second, irreducibility of 3-manifolds often plays an important role in 
the theory of 3-manifolds (see Hempel's book \cite{hempel} 
for generalities). Correspondingly, we define: 
\begin{definition}\label{def:irred}
A homology cylinder $(M,i_+,i_-) \in \Cgb$ 
is said to be {\it irreducible} if 
the underlying 3-manifold $M$ is irreducible. We denote by 
$\Cgb^{\mathrm{irr}}$ the subset of $\Cgb$ 
consisting of all irreducible homology cylinders.
\end{definition}
\noindent
By a standard argument using irreducibility, we can show that 
$\Cgb^{\mathrm{irr}}$ is a submonoid of $\Cgb$. 
In particular, $\Cg{0}{0}^{\mathrm{irr}} \cong 
\Cg{0}{1}^{\mathrm{irr}} \cong \{1\}$. 
Irreducible homology cylinders are all Haken manifolds 
since $|H_1 (M)|=\infty$ for 
any $M \in \Cg{g}{n}$ 
unless $(g,n)=(0,0), (0,1)$, 

For every $(M,i_+,i_-) \in \Cgb$, 
the underlying 3-manifold $M$ has a prime decomposition of the form
\[M \cong M_0 \sharp X_1 \sharp X_2 \sharp \cdots \sharp X_n,\]
where $M_0$ is the unique prime factor having $\partial M$ and 
$X_1, X_2, \ldots, X_n$ are homology 3-spheres. 
Note that $(M_0,i_+,i_-) \in \Cgb^{\mathrm{irr}}$. 
Using Myers' theorem \cite[Theorem 3.2]{myers}, we have the following 
description on the homology cobordism group of 
irreducible homology cylinders: 
\begin{proposition}
Every homology cylinder in $\Cgb$ with $g \ge 1$ 
is homology cobordant to 
an irreducible one. That is, 
\[\Cgb^{\mathrm{irr}}/(\mbox{$\mathrm{H}$-$\mathrm{cob}$.}) = \Hgb.\]
\end{proposition}

\subsection{Stallings' theorem and the 
Johnson filtration}
\label{subsec:JMhomHC}
From now on, we limit our discussion to the case where $n=1$ as in 
the first part of this chapter. In this subsection, 
we briefly recall how to extend the (reduced versions of) 
Dehn-Nielsen embedding and 
Johnson homomorphisms to homology 
cylinders. 
\begin{conventions}
We use the point $p \in \partial \Sgb$ 
as the common base point of $\Sgb$, $i_+(\Sgb)$, $i_-(\Sgb)$, 
a homology cylinder $M$, etc. 
\end{conventions}

For a given $(M,i_+,i_-) \in \Cgb$, 
two homomorphisms $i_+,i_-:\pi_1 \Sgb \to \pi_1 M$ are not 
generally isomorphisms. 
However, the following holds: 
\begin{theorem}[Stallings \cite{st}] \label{thm:st}
Let $A$ and $B$ be groups and $f:A \rightarrow B$ 
be a $2$-connected homomorphism. Then the induced map 
$f:N_k (A) \longrightarrow N_k (B)$ 
is an isomorphism for each $k \ge 2$.
\end{theorem}
\noindent
Here, a homomorphism $f:A \to B$ is 
said to be {\it $2$-connected} \index{$2$-connected homomorphism} 
if $f$ induces an isomorphism on 
the first homology, and an epimorphism on the second homology. 
In this chapter, the words ``Stallings' theorem'' \index{Stallings' theorem} 
always means Theorem \ref{thm:st}. 
Using the epimorphism (\ref{eq:secondhomology}), 
we can see that two homomorphisms 
$i_+,i_-:\pi = \pi_1 \Sgb \rightarrow \pi_1 M$ 
are both $2$-connected for any 
$(M,i_+,i_-) \in \Cgb$. 
Therefore, they induce isomorphisms on the 
nilpotent quotients of $\pi$ and $\pi_1 M$. 
For each $k \ge 2$, 
we can define a map 
$\sigma_k : \Cgb \rightarrow \Aut (N_k (\pi))$ by 
\[\sigma_k (M,i_+,i_-):= (i_+)^{-1} \circ i_-,\]
which gives a monoid homomorphism. 
It can be checked 
that $\sigma_k (M,i_+,i_-)$ depends only on 
the homology cobordism class of $(M,i_+,i_-)$, 
so that we have a group homomorphism 
$\sigma_k : \Hgb \rightarrow \Aut (N_k (\pi))$. 
The restriction of $\sigma_k$ 
to the subgroup $\Mgb \subset \Cgb$ 
coincides with the homomorphism $\sigma_k$ 
mentioned in Section \ref{subsec:JohnsonMorita}. 
The homomorphisms $\sigma_k$ $(k=2,3,\ldots)$ define 
filtrations 
\begin{align*}
\Cgb[1]&:=\Cgb \supset \Cgb[2] \supset \Cgb[3] \supset \Cgb[4] 
\supset \cdots\\
\Hgb[1]&:=\Hgb \supset \Hgb[2] \supset \Hgb[3] \supset \Hgb[4] 
\supset \cdots
\end{align*}
called the {\it Johnson filtration} 
\index{Johnson filtration!for homology cylinders} 
of $\Cgb$ and $\Hgb$ by 
setting $\Cgb[k]:=\Ker \sigma_k$ and $\Hgb[k]:=\Ker \sigma_k$. 

By definition, the image of 
the homomorphism $\sigma_k$ is included in 
\[\mathrm{Aut}_0 (N_k (\pi)):=
\left\{ \varphi \in \Aut (N_k (\pi)) \biggm| 
\begin{array}{c}
\mbox{There exists a lift 
$\widetilde{\varphi} \in \End \pi$ of $\varphi$}\\
\mbox{satisfying $\widetilde{\varphi}(\zeta) \equiv \zeta 
\bmod \Gamma^{k+1}(\pi)$.} 
\end{array}\right\}. \]
On the other hand, Garoufalidis-Levine 
and Habegger independently showed the following:
\begin{theorem}[Garoufalidis-Levine \cite{gl}, 
Habegger \cite{habe}]
\label{thm:gl} For $k \ge 2$, 
the image of $\sigma_k$ coincides with 
$\mathrm{Aut}_0 (N_k (\pi))$. 
\end{theorem}
\noindent
As seen in Section \ref{subsec:JohnsonMorita}, 
the $k$-th Johnson homomorphism is obtained by restricting 
$\sigma_{k+2}$ to $\Cgb[k+1]$ and $\Hgb[k+1]$. 
Garoufalidis-Levine \cite[Proposition 2.5]{gl} showed that 
Andreadakis' exact sequence in Section 
\ref{subsec:JohnsonMorita} (with $k$ shifted) 
restricts to the exact sequence 
\[1 \longrightarrow \mathfrak{h}_{g,1} (k) 
\longrightarrow \mathrm{Aut}_0 (N_{k+2} (\pi)) 
\longrightarrow \mathrm{Aut}_0 (N_{k+1} (\pi))
\longrightarrow 1.\]
Here $\mathfrak{h}_{g,1} (k) \subset 
\Hom (H,\Gamma^{k+1} (\pi)/\Gamma^{k+2} (\pi))$ is 
defined as the kernel of the composition 
\begin{align*}
\Hom (H,\Gamma^{k+1} (\pi)/\Gamma^{k+2} (\pi)) 
&\cong H^\ast \otimes (\Gamma^{k+1} (\pi)/\Gamma^{k+2} (\pi))\\ 
&\cong H \otimes (\Gamma^{k+1} (\pi)/\Gamma^{k+2} (\pi))\\
&= (\Gamma^1 (\pi)/\Gamma^2 (\pi)) 
\otimes (\Gamma^{k+1} (\pi)/\Gamma^{k+2} (\pi)) \\
&\to \Gamma^{k+2} (\pi)/\Gamma^{k+3} (\pi),
\end{align*}
\noindent
where we used the Poincar\'e duality $H^\ast =H$ in the second row 
and the last map is obtained by taking commutators. 
\begin{corollary}
The $k$-th Johnson homomorphisms 
$\tau_k: \Cgb[k+1] \to \mathfrak{h}_{g,1} (k)$ and 
$\tau_k: \Hgb[k+1] \to \mathfrak{h}_{g,1} (k)$ are 
surjective for any $k \ge 1$. 
\end{corollary}

Recall that in the case of 
the mapping class group, all the 
$\sigma_k:\Mgb \to \Aut (N_k (\pi))$ for $k \ge 2$ 
are induced from a single homomorphism 
$\sigma:\Mgb \to \Aut \pi$. 
Then we pose the following question: 
Does there exist a homomorphism $\Hgb \to \Aut \, G$ for 
some group $G$ which induces $\sigma_k:\Hgb \to \Aut (N_k(\pi))$ 
for all $k \ge 2$? 
One of the answers is to use the map 
$\sigma^{\mathrm{nil}}:\Hgb \to \Aut (\Nil{\pi})$ obtained 
by combining the homomorphisms $\sigma_k$ for all $k \ge 2$, 
where $\Nil{\pi}:=
\varprojlim_k N_k(\pi)$ is 
the nilpotent completion of $\pi$. 
In fact, it was shown by Bousfield \cite{bo} that 
$N_k(G) \cong N_k (\Nil{G})$ holds for any finitely generated 
group $G$ and hence we have a natural homomorphism 
$\Aut (\Nil{G}) \to \Aut (N_k(G))$. 
However, $\Nil{G}$ is in general enormous and difficult to treat. 
In Section \ref{sec:universal}, 
we introduce  the {\it acyclic closure} 
(or {\it HE-closure}) 
$\AC{G}$ of a group $G$ as a reasonable extension and apply it 
to our situation.

\section{Magnus representations for homology cylinders I}
\label{sec:universal}
In this section, we extend the (universal) Magnus representation $r$ to 
$\Cgb$ and $\Hgb$. The construction is based on 
Le Dimet's argument \cite{ld} for the extension of the 
Gassner representation for string links. 
However what we present here is its generalized version: 
We construct our extended representations as crossed homomorphisms 
and use more general (not necessarily commutative) rings. 

In the construction, there are two key ingredients: 
the acyclic closure of a group $G$ 
and the Cohn localization $\Lambda_G$ of $\Z [G]$. 
The former is used to give 
a generalization of the Dehn-Nielsen theorem with no reduction 
and then we construct the (universal) 
Magnus representation $r$ for $\Cgb$ and $\Hgb$ 
with the aid of the latter. 

\subsection{Observation on fundamental groups of homology cylinders}
\label{subsec:observation}
The definition of the acyclic closure of a group is given 
purely in terms of group theory, whose relationship to 
topology seems to be unclear at first glance. 
Here we digress and give 
an observation to see the background. 

For a homology cylinder $(M,i_+,i_-) \in \Cgb$, 
if we could have a natural assignment of an automorphism of $\pi$, 
there would be no problem. 
However, it seems in general difficult (maybe impossible) to do so because 
$\pi_1 M$ can be ``bigger'' than $\pi_1 \Sgb$:  

\medskip
\hspace{100pt}
\SelectTips{cm}{}
\xymatrix{
\pi_1 \Sgb \ar@/^4mm/[r]^{i_+} \ar@/_4mm/[r]_{i_-} & 
\pi_1 M \ar@{.>}[l]|\times
}

\medskip
\noindent
The observation we now start is intended to give an ``estimation'' 
of how big $\pi_1 M$ can be. 

The usual handle decomposition theory and Morse theory say that 
$M$, a homology cobordism over $\Sgb$, 
is obtained from $\Sgb \times [0,1]$ by attaching 
a number of $1$-handles $h_1^1, h_2^1, \ldots, h_m^1$ and 
the same number of $2$-handles $h_1^2, h_2^2, \ldots, h_m^2$ to 
$\Sgb \times \{1\}$. 
Then $\pi_1 M$ can be written as 
\[\pi_1 M \cong \frac{\pi \ast \langle x_1,x_2,\ldots,x_m \rangle}
{\langle r_1, r_2,\ldots, r_m\rangle},\]
where $x_i$ corresponds to attaching $h_i^1$ and 
$r_j$ to $h_j^2$, and 
$\pi \ast \langle x_1,x_2,\ldots,x_m \rangle$ denotes the free product 
of $\pi$ and $\langle x_1,x_2,\ldots,x_m \rangle$. 
Put $F_m=\langle x_1,x_2,\ldots,x_m \rangle$. 
The condition $H_\ast (M,i_-(\Sgb))=0$ implies that 
the image of $\{r_1, r_2, \ldots, r_m\}$ under the map 
\[\pi \ast F_m \xrightarrow{\mathrm{proj.}} F_m 
\longrightarrow H_1 (F_m)\]
forms a basis of $H_1 (F_m) \cong \Z^m$. 
Repeating Tietze transformations, 
we can rewrite the above presentation into one of the form
\[\pi_1 M \cong \frac{\pi \ast F_m}
{\langle x_1 v_1^{-1}, x_2 v_2^{-1}, \ldots, x_m v_m^{-1}\rangle}\]
with $v_j \in \Ker (\pi \ast F_m \xrightarrow{\mathrm{proj.}} F_m 
\longrightarrow H_1 (F_m))$ for $j=1,2,\ldots,m$. 

On the other hand, given a group of the form 
\begin{equation}\label{eq:pi1}
G=\frac{\pi \ast F_m}
{\langle x_1 w_1^{-1}, x_2 w_2^{-1}, \ldots, x_m w_m^{-1}\rangle}
\end{equation}
\noindent
with $w_j \in \Ker (\pi \ast F_m \xrightarrow{\mathrm{proj.}} F_m 
\to H_1 (F_m))$ for $j=1,2,\ldots,m$, 
we can construct a cobordism $W$ over $\Sgb \times [0,1]$ 
with $\pi_1 W \cong G$ by 
attaching ($4$-dimensional) 1-handles 
$h_1^1, h_2^1, \ldots, h_m^1$ and 
$2$-handles $h_1^2, h_2^2, \ldots, h_m^2$ to 
$(\Sgb \times [0,1]) \times \{1\} \subset 
(\Sgb \times [0,1]) \times [0,1]$ according to 
the words $x_i w_i^{-1}$. 
We denote by $M$ the opposite side of 
$(\Sgb \times [0,1]) \times \{0\}$ in $\partial W$, namely 
$\partial W = (\Sgb \times [0,1]) \cup (-M)$. 
By construction, the manifold 
$M$ with the same markings as 
$(\Sgb \times [0,1], \Id \times 1, \Id \times 0)$ 
defines a homology cylinder in $\Cgb$ and $W$ is 
a homology cobordism between $M$ and $\Sgb \times [0,1]$. 
By the duality of handle decompositions, the cobordism 
$W$ is also obtained from $M \times [0,1]$ by attaching 
$2$-handles and $3$-handles. Therefore we have a surjective 
homomorphism $\pi_1 M \twoheadrightarrow 
\pi_1 W \cong G$. That is, roughly speaking, 
$\pi_1 M$ is ``bigger'' than $G$. (In higher-dimensional cases 
discussed in Section \ref{subsec:higher}, 
we have an isomorphism $\pi_1 M \cong G$.) Consequently, we have: 
\begin{proposition}
The fundamental group $\pi_1 M$ of $(M,i_+,i_-) \in \Cgb$ 
can be written in 
the form $(\ref{eq:pi1})$ for some $m$ 
with $w_j$'s in $\Ker (\pi \ast F_m \xrightarrow{\mathrm{proj.}} F_m 
\to H_1 (F_m))$. 
Conversely, for any group $G$ having 
such a form, there exists a homology cylinder 
$(M,i_+,i_-) \in \Cgb$ such that 
$\pi_1 M$ surjects onto $G$. 
\end{proposition}

\subsection{The acyclic closure of a group}
\label{subsec:acyclic}

The concept of the acyclic closure \index{acyclic closure} 
(or HE-closure in \cite{le2}) 
of a group was defined as a variation of the algebraic closure of a group 
by Levine \cite{le1,le2}. Topologically, 
the algebraic (acyclic) closure of a group $G$ can be 
obtained as the fundamental group of (a variation of) the 
{\it Vogel localization} of any CW-complex $X$ with $\pi_1 X \cong G$ 
(see Le Dimet's book \cite{ld0}). 
We summarize here the 
definition and fundamental properties. 
We also refer to Hillman's book \cite{hi} and Cha's paper \cite{cha}. 
The proofs of the propositions in this subsection 
are almost the same as those for the algebraic closures in \cite{le1}. 

\begin{definition}
Let $G$ be a group, and let $F_m=\langle x_1, x_2, \ldots, x_m 
\rangle$ be a free group of rank $m$. \\
(i) $w=w(x_1,x_2, \ldots, x_m) \in G \ast F_m$, 
a word in $x_1, x_2, \ldots, x_m$ and elements of $G$, 
is said to be  {\it acyclic} if 
\[w \in \Ker \left( G \ast F_m \xrightarrow{\mathrm{proj}} F_m 
\longrightarrow H_1 (F_m) \right).\]
(ii) Consider the following ``equation'' with variables 
$x_1,x_2,\ldots,x_m$:
\[
\left\{\begin{array}{ccl}
x_1 & = & w_1(x_1, x_2, \ldots, x_m)\\
x_2 & = & w_2(x_1, x_2, \ldots, x_m)\\
& \vdots & \\
x_m & = & w_m(x_1, x_2, \ldots, x_m)
\end{array}\right. .
\]
When all words $w_1,w_2,\ldots,w_m \in G \ast F_m$ are 
acyclic, we call such an equation an {\it acyclic system} 
over $G$.\\
(iii) A group $G$ is said to be {\it acyclically closed} 
(AC, for short) if 
every acyclic system over $G$ with $m$ variables has 
a unique ``solution'' in $G$ for any $m \ge 0$, where 
a ``solution'' means a homomorphism $\varphi$ that makes 
the following diagram commutative:

\medskip
\hspace{70pt}
\SelectTips{cm}{}
\xymatrix{
G \ar[drr]^{\Id} \ar[d]& & \\
\displaystyle\frac{G \ast F_m}{\langle x_1 w_1^{-1},\ldots, 
x_m w_m^{-1}\rangle} \ar[rr]_-\varphi & & G
}
\end{definition}
\noindent
\begin{example}
Let $G$ be an abelian group. For $g_1,g_2,g_3 \in G$, 
consider the equation
\[\left\{\begin{array}{l}
x_1=g_1 x_1 g_2 x_2 x_1^{-1} x_2^{-1}\\
x_2=x_1 g_3 x_1^{-1}
\end{array}\right. ,
\]
which is an acyclic system. Then we have a unique solution 
$x_1=g_1 g_2 , \, x_2=g_3$.
\end{example}
\noindent
As we can expect from this example, all abelian groups are AC. 
Moreover, all nilpotent groups and the nilpotent completion 
of a group are AC, which can be deduced from the following 
fundamental properties of AC-groups: 
\begin{proposition}[{\cite[Proposition 1]{le1}}] 
${\rm (a)}$ Let $\{ G_{\alpha} \}_\alpha$ be a family of 
AC-subgroups of an AC-group $G$. Then $\bigcap_\alpha G_\alpha$ is 
also an AC-subgroup of $G$.\\
${\rm (b)}$ Let $\{ G_{\alpha} \}_\alpha$ be a family of AC-groups. 
Then $\prod_\alpha G_\alpha$ is also an AC-group.\\
${\rm (c)}$ When $G$ is a central extension of $H$, then 
$G$ is an AC-group if and only if $H$ is an AC-group.\\
${\rm (d)}$ For any direct system $($resp. inverse system$)$ 
of AC-groups, the direct limit $($resp. inverse limit$)$ is 
also an AC-group.
\end{proposition}

Let us define the acyclic closure of a group. 
\begin{proposition}[{\cite[Proposition 3]{le1}}]
For any group $G$, there exists a pair of 
a group $G^{\mathrm{acy}}$ and 
a homomorphism $\iota_G:G \rightarrow G^{\mathrm{acy}}$ 
satisfying the following properties:
\begin{enumerate}
\item $G^{\mathrm{acy}}$ is an AC-group.
\item Let $f:G \rightarrow A$ be a homomorphism 
and suppose that $A$ is an AC-group. 
Then there exists a unique homomorphism 
$f^{\mathrm{acy}}:G^{\mathrm{acy}} \rightarrow A$ 
which satisfies $f^{\mathrm{acy}} \circ \iota_G = f$. 
\end{enumerate}
\noindent
Moreover such a pair is unique up to isomorphism.
\end{proposition}
\begin{definition}
We call $\iota_G$ (or $G^{\mathrm{acy}}$) obtained above 
the {\it acyclic closure} of $G$.
\end{definition}
\noindent
Taking the acyclic closure of a group is functorial, namely, 
for each group homomorphism $f:G_1 \to G_2$, we have 
the induced homomorphism $f^{\mathrm{acy}}:G_1^{\mathrm{acy}} \to 
G_2^{\mathrm{acy}}$ by applying the universal property of 
$G_1^{\mathrm{acy}}$ to the homomorphism 
$\iota_{G_2} \circ f$, and the composition of homomorphisms 
induces that of the corresponding homomorphisms 
on acyclic closures. 

The most important properties of the acyclic closure 
are the following: 
\begin{proposition}[{\cite[Proposition 4]{le1}}]\label{prop:2cn}
For every group $G$, the acyclic closure 
$\iota_G:G \rightarrow G^{\mathrm{acy}}$ is $2$-connected.
\end{proposition}
\begin{proposition}[{\cite[Proposition 5]{le1}}]\label{prop:isom}
Let $G$ be a finitely generated group and $H$ 
a finitely presentable group. For each 
$2$-connected homomorphism 
$f:G \rightarrow H$, the induced homomorphism 
$f^{\mathrm{acy}}:G^{\mathrm{acy}} 
\to H^{\mathrm{acy}}$ on acyclic closures is an isomorphism. 
\end{proposition}
\noindent
From Proposition \ref{prop:2cn} and Stallings' theorem, the 
nilpotent quotients of a group and 
those of its acyclic closure are isomorphic. Note that 
the homomorphism $\iota_G$ is not necessarily injective: 
consider a perfect group $G$ and the 2-connected homomorphism 
$G \to \{1\}$. As for a free group $F_m$, 
its residual nilpotency shows 
that $\iota_{F_m}$ is injective. 

\begin{proposition}[{\cite[Proposition 6]{le1}}]\label{prop:seq}
For any finitely presentable group $G$, 
there exists a sequence of finitely presentable groups 
and homomorphisms
\[G=P_0 \rightarrow P_1 \rightarrow P_2 
\rightarrow \cdots \rightarrow P_k \rightarrow P_{k+1} 
\rightarrow \cdots \] 
satisfying the following properties:
\begin{enumerate}
\item $G^{\mathrm{acy}}= 
\displaystyle\varinjlim_k P_k$, and 
$\iota_G : G \rightarrow G^{\mathrm{acy}}$ 
coincides with the limit map of the 
above sequence. 
\item $G \rightarrow P_k$ is a $2$-connected homomorphism 
for any $k \ge 1$. 
\end{enumerate}
\end{proposition}
\noindent
From this proposition, we see, in particular, that the acyclic 
closure of a finitely presentable group is a countable set.

\subsection{Dehn-Nielsen type theorem}
Now we return to our discussion on 
homology cylinders. 
For each homology cylinder $(M,i_+,i_-) \in \Cgb$, 
we have a commutative diagram
\[\begin{CD}
\pi @>i_->> \pi_1 M @<i_+<< \pi\\ 
@V\iota_{\pi}VV @V\iota_{\pi_1 M}VV @VV\iota_{\pi}V \\
\AC{\pi} @>i_-^{\mathrm{acy}}>\cong> (\pi_1 M)^{\mathrm{acy}} 
@<i_+^{\mathrm{acy}}<\cong< \AC{\pi} 
\end{CD}\]
by Proposition \ref{prop:isom}. 
From this, we obtain a monoid homomorphism defined by
\[\sigma^{\mathrm{acy}}:\Cgb \longrightarrow \Aut (\AC{\pi}) \qquad 
\left( \, (M,i_+,i_-) \mapsto 
(i_+^{\mathrm{acy}})^{-1} \circ i_-^{\mathrm{acy}} \, \right)\]
and it induces a group homomorphism 
$\sigma^{\mathrm{acy}}:\Hgb \rightarrow \Aut (\AC{\pi})$. 

Here we describe a generalization of the Dehn-Nielsen theorem. 
Recall that $\zeta \in \pi \subset \AC{\pi}$ 
is a word corresponding to the boundary loop of $\Sgb$. 
\begin{theorem}[{\cite[Theorem 6.1]{sa2}}]\label{thm:autsp}
The image of $\sigma^{\mathrm{acy}}:\Hgb \to \Aut (\AC{\pi})$ is 
\[\mathrm{Aut}_0 (\AC{\pi}) :=
\{ \varphi \in \Aut (\AC{\pi}) \mid
\varphi(\zeta) = \zeta \in \AC{\pi} \}.\]
\end{theorem}
\noindent
In the proof, we immediately see that the image of 
$\sigma^{\mathrm{acy}}$ 
is included in $\mathrm{Aut}_0 (\AC{\pi})$ since  
$i_+ (\zeta)=i_- (\zeta) \in \pi_1 M$ for every 
$(M,i_+,i_-) \in \Cgb$. 
Conversely, given an element $\varphi \in 
\mathrm{Aut}_0 (\AC{\pi})$, 
we need to construct a homology cylinder $M=(M,i_+,i_-)$ satisfying 
$\sigma^{\mathrm{acy}}(M)=\varphi$. 
The construction is based on that of Theorem \ref{thm:gl} 
\cite[Theorem 3]{gl} due to Garoufalidis-Levine. 
In our context, however, we must pay extra attention because 
there is a difference between our situation and theirs: 
Although the composition 
$\pi \to \AC{\pi} \to N_k(\AC{\pi}) \cong N_k (\pi)$ 
is surjective, $\iota_{\pi}:\pi \to \AC{\pi}$ is not. 

Note that 
$\Hgb[[\infty]]:= \Ker ( \sigma^{\mathrm{acy}})$ 
is non-trivial in contrast with the case of the mapping class group. 
Indeed the homology cobordism group $\Theta_{\Z}^3$ 
of homology 3-spheres is included in it. 
See also Section \ref{subsec:signature} 
for more about $\Hgb[[\infty]]$. 
As for the mysterious group $\Hgb[[\infty]]$, 
we have the following problem:
\begin{problem}
Determine whether $\Hgb[[\infty]]$ 
coincides with the group 
\begin{align*}
\Hgb[\infty]:=&\bigcap_{k \ge 2}\Ker 
( \sigma_k: \Hgb \to \Aut (N_k(\pi)) )\\
=&\Ker ( \sigma^{\mathrm{nil}}: \Hgb \to \Aut (\Nil{\pi}) ),
\end{align*}
\noindent
in which $\Hgb[[\infty]]$ is included. 
This is closely related to the question whether 
the natural map $\AC{\pi} \to \Nil{\pi}$ is injective or not.
\end{problem}

\subsection{Extension of the Magnus representation}

As in the original case, the (universal) Magnus representation for 
homology cylinders is obtained from that for 
$\Aut (\Acy_n)$ through the Dehn-Nielsen type theorem. 
The resulting representation is also a crossed homomorphism. 

For the construction, 
we need another tool called (a special case of) 
the {\it Cohn localization} \index{Cohn localization} 
or the {\it universal localization}. 
We refer to \cite[Section 7]{co} for details. 
\begin{proposition}[Cohn \cite{co}]\label{prop:cohn} 
Let $G$ be a group with the trivializer $\mathfrak{t}:\Z [G] \to \Z$. 
Then there exists a pair of 
a ring $\Lambda_G$ and a ring homomorphism 
$l_G:\Z [G] \to \Lambda_G$ satisfying the following properties:
\begin{enumerate}
\item For every matrix $m$ with coefficients in $\Z [G]$, 
if \,${}^\mathfrak{t}m$ is invertible 
then \,${}^{l_G}m$ is also invertible.
\item The pair $(\Lambda_G, l_G)$ is universal among all pairs 
having the property $(1)$. 
\end{enumerate}
Furthermore the pair $(\Lambda_G, l_G)$ is unique up to isomorphism. 
\end{proposition}
\noindent
Note that any automorphism of a group $G$ induces an 
automorphism of $\Z [G]$ 
and moreover of $\Lambda_{G}$ 
by the universal property of $\Lambda_{G}$. 
\begin{example}
When $G=H_1 (F_n)$, we have 
\[\Lambda_{H_1 (F_n)} \cong \left\{ \, 
\displaystyle\frac{f}{g} \, \biggm| 
f,g \in \Z [H_1 (F_n)], \ 
\mathfrak{t}(g)=\pm 1 \right\}.\]
\end{example}

We write $\gamma_i$ again for the image of $\gamma_i$ by 
$\iota_{F_n}:F_n=\langle \gamma_1,\gamma_2,\ldots,\gamma_n \rangle 
\hookrightarrow \Acy_n$. Now we can check the following 
facts on $\Lambda_{\Acy_n}$: 
\begin{lemma}\label{lem:key}
$(1)$ \ The composition $\Z [F_n] \xrightarrow{\iota_{F_n}} \Z [\Acy_n] 
\xrightarrow{l_{\Acy_n}} \Lambda_{\Acy_n}$ is injective. 

\noindent
$(2)$ \ Let $G$ be a finitely presentable group and let $f:F_n \to G$ 
be a $2$-connected homomorphism. Then 
$H_i (G,f(F_n);\Lambda_G)=0$ holds for $i =0,1,2$. Moreover, we may 
take $\Acy_n$ as $G$. 
\end{lemma}
\begin{proof}[Sketch of Proof]
Consider the composition of the ring homomorphism
$\Z [\Acy_n] \to \Z [\Nil{F_n}]$ with the Magnus expansion, 
which can be extended to $\Z [\Nil{F_n}]$. 
It is known that the Magnus expansion is injective on $\Z [F_n]$. 
We can check that this composition 
satisfies Property $(1)$ of Proposition \ref{prop:cohn}, so that 
the Magnus expansion can be extended to $\Lambda_{\Acy_n}$. 
Hence $(1)$ follows. 

For the proof of the first assertion of $(2)$, see \cite[Lemma 5.11]{sa2}. 
We may put $G=\Acy_n$ by Proposition \ref{prop:seq} and 
commutativity of homology and direct limits. 
\end{proof}
Lemma \ref{lem:key} $(2)$ leads us to show the following, 
which can be regarded as a generalization 
of the isomorphism (\ref{eq:aug}). 
The proof is almost the same as that 
of \cite[Proposition 1.1]{ld}. 
\begin{proposition}\label{extfree} 
The homomorphism
\[\chi: \Lambda_{\Acy_n}^n \longrightarrow 
I(\Acy_n) \otimes_{\Z [\Acy_n]} \Lambda_{\Acy_n}\]
sending $(a_1,\ldots,a_n)^T \in \Lambda_{\Acy_n}^n$ to \ 
$\displaystyle\sum_{i=1}^n (\gamma_i^{-1}-1) \otimes a_i$ 
is an isomorphism of right $\Lambda_{\Acy_n}$-modules, where 
$I(\Acy_n):=\Ker (\mathfrak{t}:\Z [\Acy_n] \to \Z)$.
\end{proposition}

\begin{definition}
For $1 \le i \le n$, we define the {\it extended Fox derivative} 
\index{Fox derivative!extended} 
\[\frac{\partial}{\partial \gamma_i}:\Acy_n 
\longrightarrow \Lambda_{\Acy_n}\]
by the formula
\[\begin{array}{rccc}
\left( \displaystyle\frac{\partial}{\partial \gamma_1}, 
\frac{\partial}{\partial \gamma_2}, \ldots, 
\frac{\partial}{\partial \gamma_n} \right)^T:& \Acy_n & 
\longrightarrow &
\Lambda_{\Acy_n}^n \\
& \mbox{\rotatebox[origin=c]{90}{$\in$}} & & 
\mbox{\rotatebox[origin=c]{90}{$\in$}} \\
& v & \longmapsto & \overline{\chi^{-1} ((v^{-1}-1) \otimes 1)}.
\end{array}\]
\end{definition}
\noindent
The extended Fox derivatives 
coincide with the original ones 
if we restrict them to $F_n$ (cf. Example \ref{ex:aug}). 
They share many properties 
as mentioned in \cite[Proposition 1.3]{ld}. 
In particular, we have the equality
\[(v^{-1} -1) \otimes 1= \sum_{i=1}^n (\gamma_i^{-1} -1) \otimes 
\overline{\left(\frac{\partial v}{\partial \gamma_i}\right)} \ \ 
\in 
I(\Acy_n) \otimes_{\Z [\Acy_n]} \Lambda_{\Acy_n}
\]
for any $v \in \Acy_n$. 
\begin{definition}
The {\it $($universal\,$)$ Magnus representation} 
\index{Magnus representation!for 
$\mathrm{Aut}\,F^{\mathrm{acy}}_n$} 
for $\Aut (\Acy_n)$ is 
the map 
\[r:\Aut (\Acy_n) \to M(n,\Lambda_{\Acy_n})\]
assigning to $\varphi \in \Aut (\Acy_n)$ the matrix 
\[r(\varphi):=\left(\overline{\left(\frac{\partial \varphi(\gamma_j)}
{\partial \gamma_i} \right)}\right)_{i,j}.\]
\end{definition}
We can easily check that 
the Magnus representation $r$ is a crossed homomorphism and 
the image of $r$ is included in the set $\mathrm{GL}(n,\Lambda_{\Acy_n})$ 
of invertible matrices. 
By Lemma \ref{lem:key} $(1)$, we see that 
the Magnus representation defined 
here gives a generalization of the original. 
\begin{example}\label{twoconn}
Consider the monoid $\End_2 (F_n)$ of all $2$-connected 
endomorphisms of $F_n$. We have a natural homomorphism 
$\End_2 (F_n) \to \Aut (\Acy_n)$ by Proposition \ref{prop:isom}. 
For any $f \in \End_2 (F_n)$, the Magnus matrix $r(f)$ can be 
obtained by using the original Fox derivatives. In particular, 
$r$ is injective on $\End_2 (F_n)$. Therefore, 
we see that $\End_2 (F_n)$ is a submonoid of $\Aut (\Acy_n)$. 
Every automorphism of $\Aut (N_k(F_n))$ can be lifted to 
a $2$-connected endomorphism of $F_n$. Hence the homomorphisms 
$\Aut (\Acy_n) \to \Aut (N_k (\Acy_n)) \cong \Aut (N_k(F_n))$ 
are surjective for all $k \ge 2$. 
\end{example}

Finally, by using the Dehn-Nielsen type theorem for $n=2g$, 
we obtain the (universal) Magnus representation 
\index{Magnus representation!for homology cylinders}
\[r: \Cgb \longrightarrow \mathrm{GL}(2g, \Lambda_{\AC{\pi}})\]
for homology cylinders, 
which induces $r: \Hgb \to \mathrm{GL}(2g, \Lambda_{\AC{\pi}})$. 


\section{Magnus representations for homology cylinders II}
\label{sec:extension}
In this section, we discuss another method for extending 
Magnus representations by using 
twisted homology of homology cylinders. 
This time, we follow the Kirk-Livingston-Wang's construction 
\cite{klw}. 
In connection with it, we also mention another invariant 
of homology cylinders arising 
from torsion. 

For our purpose, we first recall the setting of 
{\it higher-order Alexander invariants} 
\index{higher-order Alexander invariants} originating in 
Cochran-Orr-Teichner \cite{cot}, 
Cochran \cite{coc} and Harvey \cite{har,har2}, 
where PTFA groups play an important role. 
A group $\Gamma$ is said 
to be {\it poly-torsion-free abelian $($PTFA$)$} 
\index{poly-torsion-free abelian (PTFA) group} 
if it has a sequence 
\[\Gamma=\Gamma_1 \triangleright \Gamma_2 \triangleright 
\cdots \triangleright \Gamma_n = \{1\}\]
whose successive quotients $\Gamma_i/\Gamma_{i+1}$ $(i \ge 1)$ 
are all torsion-free abelian. An advantage of using PTFA groups is that 
the group ring $\mathbb{Z} [\Gamma]$ of $\Gamma$ 
is known to be an {\it Ore domain} so that it can be embedded into 
the field (skew field in general) 
\[\mathcal{K}_\Gamma:= \mathbb{Z}[\Gamma] 
(\mathbb{Z}[\Gamma] - \{0\})^{-1}\]
called the {\it right field of fractions}. 
We refer to the books of Cohn \cite{co} and Passman \cite{pa} 
for generalities of localizations of non-commutative rings. 
A typical example of a PTFA group is $\mathbb{Z}^n$, 
where $\mathcal K_{\mathbb{Z}^n}$ is isomorphic to 
the field of rational functions with $n$ variables. 
More generally, 
free nilpotent quotients $N_k(F_n)$ and $N_k (\pi)$ are 
known to be PTFA. 

Let $M=(M,i_+,i_-) \in \Cgb$. We fix a homomorphism 
$\rho: \pi_1 M \to \Gamma$ into a PTFA group $\Gamma$. 
The following lemma is crucial in our construction of 
Magnus matrices (cf. Lemma \ref{lem:key}). 
For the direct proof, see \cite[Proposition 2.1]{klw}. 
See also Cochran-Orr-Teichner \cite[Section 2]{cot} 
for a more general treatment. 
\begin{lemma}\label{lem:relative}
For $\pm \in \{+,-\}$, 
we have $H_\ast (M,i_\pm (\Sgb);\mathcal{K}_\Gamma)=0$. 
Equivalently, 
\[i_\pm: 
H_\ast (\Sgb,\{p\} ;i_{\pm}^\ast \mathcal{K}_{\Gamma}) 
\longrightarrow 
H_\ast (M,\{p\} ;\mathcal{K}_{\Gamma})\]
is an isomorphism of 
right $\mathcal{K}_\Gamma$-vector spaces. 
\end{lemma}
\noindent
\begin{remark}\label{rem:acyclic}
The same conclusion as in the above lemma holds 
for the homology with coefficients in any 
$\Z [\pi_1 M]$-algebra $\mathcal{R}$ satisfying: 
Every matrix with entries in $\Z [\pi_1 M]$ sent to 
an invertible one by the trivializer $\mathfrak{t}: 
\Z [\pi_1 M] \to \Z$ is 
also invertible in $\mathcal{R}$ (cf. Proposition \ref{prop:cohn}). 
By a theorem of Strebel \cite{strebel}, we see that 
$\mathcal{K}_{\Gamma}$ satisfies this property for any 
homomorphism $\pi_1 M \to \Gamma$ into a PTFA group $\Gamma$. 
\end{remark}

Since $S:=\dis\bigcup_{i=1}^{2g} \gamma_i \subset \Sgb$ 
(see Figure \ref{fig:generator}) 
is a deformation 
retract of $\Sgb$ relative to $p$, we have 
$\pi \cong \pi_1 S$ and 
\[H_1(\Sgb,\{p\};i_{\pm}^\ast \mathcal{K}_\Gamma) \cong 
H_1 (S,\{p\};i_{\pm}^\ast \mathcal{K}_\Gamma) = 
C_1 (\widetilde{S}) \otimes_{\Z [\pi]} 
i_{\pm}^\ast \mathcal{K}_\Gamma 
\cong \mathcal{K}_\Gamma^{2g}\]
with basis
$\{ \widetilde{\gamma_1} \otimes 1, \ldots , 
\widetilde{\gamma_{2g}} \otimes 1\} 
\subset C_1 (\widetilde{S}) \otimes_{\Z [\pi]} 
i_{\pm}^\ast \mathcal{K}_\Gamma$ as a right 
$\mathcal{K}_\Gamma$-vector space 
(see Section \ref{subsec:magnus_MCG_symp}). 

\begin{definition}\label{def:Mag2}
For $M=(M,i_+,i_-) \in \Cgb$ and 
a homomorphism $\pi_1 M \to \Gamma$ into a PTFA group $\Gamma$, 
the {\it Magnus matrix} \index{Magnus matrix} 
$r_\rho (M) \in \mathrm{GL}(2g,\mathcal{K}_\Gamma)$ 
{\it associated with} $\rho$ is 
defined as the representation matrix of 
the right $\mathcal{K}_\Gamma$-isomorphism
\begin{align*}
\mathcal{K}_\Gamma^{2g} \cong 
H_1 (\Sgb,\{p\};i_-^\ast \mathcal{K}_\Gamma) 
&\xrightarrow[\cong]{i_-} 
H_1 (M,\{p\};\mathcal{K}_\Gamma) \\
&\xrightarrow[\cong]{i_+^{-1}} 
H_1 (\Sgb,\{p\};i_+^\ast \mathcal{K}_\Gamma) 
\cong \mathcal{K}_\Gamma^{2g},
\end{align*}
\noindent
where the first and the last isomorphisms use 
the basis mentioned above. 
\end{definition}

A method for computing 
$r_\rho (M)$ is given in \cite[Section 4]{gs08}, 
which is based on one of 
Kirk-Livingston-Wang \cite{klw}. 
An {\it admissible presentation} of $\pi_1 M$ is defined to be 
one of the form 
\begin{align}\label{admissible}
\langle i_- (\gamma_1),\ldots,i_- (\gamma_{2g}), 
z_1 ,\ldots, z_l, 
i_+ (\gamma_1),\ldots,i_+ (\gamma_{2g}) \mid 
r_1, \ldots, r_{2g+l}
\rangle
\end{align}
for some integer $l \ge 0$. 
That is, it is a finite presentation with deficiency $2g$ 
whose generating set 
contains $i_- (\gamma_1),\ldots,i_- (\gamma_{2g}), 
i_+ (\gamma_1),\ldots,i_+ (\gamma_{2g})$ and is ordered as above. 
Such a presentation always exists. For any admissible presentation, 
we define $2g \times (2g+l)$, $l \times (2g+l)$ and 
$2g \times (2g+l)$ matrices $A,B,C$ by 
\[A=\overline{
\left(\frac{\partial r_j}{\partial i_-(\gamma_i)}
\right)}_{\begin{subarray}{c}
{}1 \le i \le 2g\\
1 \le j \le 2g+l
\end{subarray}}, \ 
B=\overline{
\left(\frac{\partial r_j}{\partial z_i}
\right)}_{\begin{subarray}{c}
{}1 \le i \le l\\
1 \le j \le 2g+l
\end{subarray}}, \ 
C=\overline{
\left(\frac{\partial r_j}{\partial i_+(\gamma_i)}
\right)}_{\begin{subarray}{c}
{}1 \le i \le 2g\\
1 \le j \le 2g+l
\end{subarray}}\]
over $\Z [\Gamma] \subset \mathcal{K}_\Gamma$. 
\begin{proposition}[{\cite[Propositions 4.5, 4.6]{gs08}}]
\label{prop:MagnusFormula}
The square matrix $\begin{pmatrix} A \\ B \end{pmatrix}$ 
is invertible over $\mathcal{K}_\Gamma$ and we have 
\begin{equation}\label{eq:mag_formula}
r_\rho(M) = 
-C \begin{pmatrix} A \\ B \end{pmatrix}^{-1} \!
\begin{pmatrix} I_{2g} \\ 0_{(l,2g)}\end{pmatrix} 
\in \mathrm{GL}(2g,\mathcal{K}_\Gamma).
\end{equation}
\end{proposition}
\begin{remark}
We shall meet the same formula (\ref{eq:mag_formula}) when 
we compute the Magnus matrix following the definition in the 
previous section. From this, we can conclude that 
the definitions in this and the previous sections are the same. 
\end{remark}

Formula (\ref{eq:mag_formula}) gives 
the following properties of Magnus matrices:
\begin{proposition}\label{prop:mag_rho}
Let $\Gamma$ be a PTFA group. 
\begin{itemize}
\item[$(1)$] For $\varphi \in \Mgb \hookrightarrow \Aut \pi$ and 
a homomorphism $\rho :\pi_1 (\Sgb \times [0,1]) =\pi \to \Gamma$, 
we have
\[r_\rho ((\Sgb \times [0,1], 
\Id \times 1, \varphi \times 0)) = 
\sideset{^\rho\!}{_{i,j}}
{\Tmatrix{
\overline{\left(
\displaystyle\frac{\partial \varphi(\gamma_j)}{\partial \gamma_i}
\right)}}}.\]

\item[$(2)$] $($Functoriality\,$)$ 
For $M,N \in \Cgb$ and a homomorphism 
$\rho: \pi_1 (M \cdot N) \to \Gamma$, we have 
\[r_{\rho} (M \cdot N) = 
r_{\rho \circ i} (M) \cdot r_{\rho \circ j} (N),\]
where $i:\pi_1M \to \pi_1 (M \cdot N)$ and 
$j:\pi_1 N \to \pi_1 (M \cdot N)$ are the induced maps from 
the inclusions $M \hookrightarrow M \cdot N$ and 
$N \hookrightarrow M \cdot N$. 

\item[$(3)$] $($Homology cobordism invariance\,$)$ Suppose 
$M, N \in \Cgb$ are homology cobordant by a homology cobordism 
$W$. For any homomorphism $\rho:\pi_1 W \to \Gamma$, we have 
\[r_{\rho \circ i} (M) = r_{\rho \circ j} (N),\]
where $i:\pi_1 M \to \pi_1 W$ and 
$j:\pi_1 N \to \pi_1 W$ are the induced maps from 
the inclusions $M \hookrightarrow W$ and $N \hookrightarrow W$. 
\end{itemize}
\end{proposition}

Hence Magnus matrices are invariants of {\it pairs} of 
a homology cylinder and a homomorphism $\rho$. 
To obtain a map from a submonoid of $\Cgb$ solely, 
we need a natural 
choice of $\rho$ for all homology cylinders involved 
that have some compatibility with respect 
to the product operation in $\Cgb$. 
For that purpose, we here use 
the nilpotent quotient $N_k(\pi)$ with 
fixed $k \ge 2$. Using Stallings' theorem, 
we can consider the composition
\[\mathfrak{q}_k: \pi_1 M \longrightarrow 
N_k( \pi_1 M) \xrightarrow[\cong]{i_+^{-1}}
N_k (\pi)\]
for every $(M,i_+,i_-) \in \Cgb$. 
Then we have a map 
\[r_{\mathfrak{q}_k}: \Cgb \longrightarrow 
\mathrm{GL}(2g,\mathcal{K}_{N_k (\pi)})\]
and Proposition \ref{prop:mag_rho} can be rewritten 
as follows: 
\begin{proposition}[{\cite{sakasai08}}] Let $k \ge 2$ 
be an integer. 
\begin{itemize}
\item[$(1)$] The map $r_{\mathfrak{q}_k}$ extends 
the corresponding 
Magnus representation for $\Mgb$. 

\item[$(2)$] The map $r_{\mathfrak{q}_k}$ is a crossed homomorphism, 
namely, the equality 
\[r_{\mathfrak{q}_k}(M_1 \cdot M_2) = 
r_{\mathfrak{q}_k}(M_1) \cdot 
{}^{\sigma_k(M_1)} (r_{\mathfrak{q}_k}(M_2))\]
holds for any $M_1, M_2 \in \Cgb$ 
by using $\sigma_k:\Cgb \to \Aut (N_k(\pi))$. 

\item[$(3)$] $r_{\mathfrak{q}_k}$ induces a crossed homomorphism 
$r_{\mathfrak{q}_k}: \Hgb \to \mathrm{GL}(2g,\mathcal{K}_{N_k (\pi)})$. 
\end{itemize}
\end{proposition}
\noindent
As in the case of $\Mgb$, the restrictions of $r_{\mathfrak{q}_k}$ to 
$\Cgb [k]$ and $\Hgb [k]$ give genuine homomorphisms.

We can naturally generalize the arguments 
in Sections \ref{subsec:magnus_MCG_symp} and 
\ref{subsec:magnus_MCG_nonfaithful}. For example, 
the (twisted) symplecticity 
\begin{equation}\label{eq:symplecticity}
\overline{r_{\mathfrak{q}_k} (M)^T} \ 
{}^{\mathfrak{q}_k} \!\widetilde{J} \ 
r_{\mathfrak{q}_k} (M) = 
{}^{\sigma_k (M)} ({}^{\mathfrak{q}_k} \!\widetilde{J}) 
\in \mathrm{GL}(2g,\mathcal{K}_{N_k (\pi)})
\end{equation}
\noindent
holds. 
Note that the proof in \cite{sa3} is also applicable to 
the universal Magnus representation $r$. 

\begin{remark}
The author does not know whether we can define 
(crossed) homomorphisms from $\Cgb$ and $\Hgb$ by using 
derived quotients of $\pi_1 M$. 
This is because there are no results for derived 
quotients of groups corresponding completely to 
Stallings' theorem except that 
Cochran-Harvey \cite{coc_har} gave a partial result, 
which was 
used to define homology cobordism invariants of 3-manifolds 
arising from $L^2$-signature invariants 
(see Harvey \cite{har3} for example). 
\end{remark}

\begin{example}[{\cite[Example 4.4]{sakasai08}}]\label{ex:string_comp} 
Let $L$ be the pure string link of Figure \ref{fig:ex4_4} with 
$2$ strings. 

\begin{figure}[htbp]
\begin{center}
\includegraphics{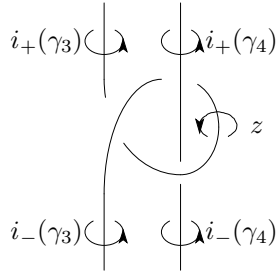}
\end{center}
\caption{String link $L$}
\label{fig:ex4_4}
\end{figure}

By Levine's construction in Example \ref{sgtocg}, 
$L$ yields a homology cylinder 
$(M_L,i_+,i_-) \in \Cgb$ with 
$\pi_1 M_L$ having an admissible presentation: 
\[{\small \left\langle
\begin{array}{c|c}
\begin{array}{c}
i_-(\gamma_1),\ldots,i_-(\gamma_4)\\
z\\
i_+(\gamma_1),\ldots,i_+(\gamma_4)\\
\end{array} &
\begin{array}{l}
i_+(\gamma_1) i_-(\gamma_3)^{-1} i_+(\gamma_4) i_-(\gamma_1)^{-1},\\
{[i_+(\gamma_1),i_+(\gamma_3)]} i_+(\gamma_2) z i_-(\gamma_2)^{-1} 
{[i_-(\gamma_3),i_-(\gamma_1)]},\\
i_+(\gamma_4) i_-(\gamma_3) i_+(\gamma_4)^{-1} z^{-1},\\
i_-(\gamma_3) i_+(\gamma_3)^{-1} i_-(\gamma_3)^{-1} z, \ 
i_-(\gamma_4) z^{-1} i_+(\gamma_4)^{-1} z

\end{array}
\end{array}\right\rangle}.\]

Let us compute the Magnus matrix $r_{\mathfrak{q}_2}(M_L)$. 
We identify $H=N_2 (\pi)$ and $N_2 (\pi_1 M_L)=H_1 (M_L)$ by 
using $i_+$. From the presentation, we have $z=i_-(\gamma_3)=\gamma_3$, 
$i_-(\gamma_4)=\gamma_4$, $i_-(\gamma_2)=\gamma_2\gamma_3$ and 
$i_-(\gamma_1)=\gamma_1\gamma_3^{-1}\gamma_4$ in $H$. Then 
\begin{align*}
\begin{pmatrix}A \\ B \end{pmatrix} &=
\begin{pmatrix}
-1&\gamma_3^{-1}-1& 0 & 0 & 0 \\
0 & -1 & 0 & 0 & 0 \\
-\gamma_1^{-1}\gamma_3 & 
1- \gamma_1^{-1} \gamma_3 \gamma_4^{-1} &
\gamma_4^{-1}& 1-\gamma_3 & 0\\
0 & 0 & 0 & 0 & 1 \\
0 & \gamma_2^{-1}& -1 & \gamma_3 & \gamma_3-\gamma_3 \gamma_4^{-1}
\end{pmatrix}, \\
C &=\begin{pmatrix}
1& 1-\gamma_3^{-1} &0&0&0\\
0&1&0&0&0\\
0&\gamma_1^{-1} -1&0&-1&0\\
\gamma_1^{-1}\gamma_3&0&1-\gamma_3^{-1}& 0&-\gamma_3 
\end{pmatrix},
\end{align*}
\noindent
over $\Z [H]$. The Magnus matrix $r_{\mathfrak{q}_2}(M_L) = -C 
\begin{pmatrix}A \\ B \end{pmatrix}^{-1} \!
\begin{pmatrix}I_4 \\ 0_{(1,4)}\end{pmatrix}$ is given by
\[
\left(\begin{array}{cccc}
1&0&0&0\\
&&&\\
0&1&0&0
\\
&&&\\
\frac{-\gamma_1^{-1}}{\gamma_3^{-1}+\gamma_4^{-1}-1}
&\frac{\gamma_2^{-1} \gamma_3^{-1} \gamma_4^{-1} 
-\gamma_4^{-1}+1}{\gamma_3^{-1}+\gamma_4^{-1}-1} &
\frac{\gamma_3^{-1}}{\gamma_3^{-1}+\gamma_4^{-1}-1}
&\frac{\gamma_4^{-1}(\gamma_4^{-1}-1)}{\gamma_3^{-1}+\gamma_4^{-1}-1}\\
&&&\\
\frac{\gamma_1^{-1}\gamma_3\gamma_4^{-1}}{\gamma_3^{-1}+\gamma_4^{-1}-1}
& \frac{(1-\gamma_3^{-1})(\gamma_2^{-1} \gamma_3^{-1} 
-\gamma_2^{-1}-1)}{\gamma_3^{-1}+\gamma_4^{-1}-1} 
&\frac{\gamma_3^{-1}-1}{\gamma_3^{-1}+\gamma_4^{-1}-1}
&\frac{-\gamma_3^{-1}\gamma_4^{-1}+\gamma_3^{-1}
+2\gamma_4^{-1}-1}{\gamma_3^{-1}+\gamma_4^{-1}-1}
\end{array}\right).\]
Note that 
\[\det (r_{\mathfrak{q}_2}(M_L))= \gamma_3^{-1} \gamma_4^{-1} 
\frac{\gamma_3 + \gamma_4-1}{\gamma_3^{-1}+\gamma_4^{-1}-1}.\]
Since $r_{\mathfrak{q}_2}(M_L)$ has an entry not belonging to $\Z [H]$, 
we see that $M_L$ is not in $\Mgb$. 
In other words, $L$ is not a braid. 
\end{example}

We close this section by introducing another invariant of 
homology cylinders called the {\it $\Gamma$-torsion}. 
We refer to Milnor \cite{milnor}, Turaev \cite{turaev} 
and Rosenberg \cite{ros} 
for generalities of torsions and basics of $K_1$-group. 
Here we only recall that for a ring $R$, the abelian group $K_1 (R)$ is 
defined as the abelianization of the group $\mathrm{GL}(R)=
\dis\varinjlim_{n} \mathrm{GL}(n,R)$ 
of invertible matrices with entries in $R$. 
By Lemma \ref{lem:relative}, the relative complex 
$C_\ast (M,i_+(\Sigma_{g,1});\mathcal{K}_\Gamma)$ 
obtained from any cell decomposition of $(M, i_+(\Sigma_{g,1}))$ 
is acyclic, so that the torsion 
$\tau(C_\ast (M,i_+(\Sigma_{g,1}); \mathcal{K}_\Gamma))$ 
can be defined. 
\begin{definition}
Let $M=(M,i_+,i_-) \in \Cgb$ with a homomorphism 
$\rho:\pi_1 M \to \Gamma$ into a PTFA group $\Gamma$. 
The $\Gamma$-{\it torsion} $\tau_{\rho}^+ (M)$ 
of $M$ is defined by 
\[\tau_{\rho}^+ (M):=
\tau(C_\ast (M,i_+(\Sgb);\mathcal{K}_\Gamma))
\in K_1 (\mathcal{K}_\Gamma)/\pm \rho (\pi_1 M).\]
\end{definition}
\noindent
Note that for any field $\mathcal{K}$, 
the {\it Dieudonn\'e determinant} 
gives an isomorphism $K_1(\mathcal{K}) \cong H_1 (\mathcal{K}^\times)$, 
where $\mathcal{K}^\times=\mathcal{K}-\{0\}$ denotes the unit group. 
The $\Gamma$-torsion is trivial when 
$(M,i_+,i_-) \in \Cgb$ is contained in $\Mgb$ 
since $M=\Sgb \times [0,1]$ collapses 
to $i_+(\Sgb)$. 

The $\Gamma$-torsion satisfies the following properties: 
\begin{proposition}
Suppose $\Gamma$ is a PTFA group and 
$M, N \in \Cgb$. 
\begin{itemize}
\item [$(1)$] 
For a homomorphism $\rho:\pi_1 M \to \Gamma$, 
the $\Gamma$-torsion $\tau_\rho^+ (M)$ 
can be computed from 
any admissible presentation of $\pi_1 M$ and is given by 
\ $\begin{pmatrix}A \\ B \end{pmatrix} \in 
K_1 (\mathcal{K}_\Gamma)/\pm \rho (\pi_1 M)$.  

\item [$(2)$] $($Functoriality\,$)$ For a homomorphism 
$\rho:\pi_1 (M \cdot N) \to \Gamma$, we have 
\[\tau_{\rho}^+ (M \cdot N) = 
\tau_{\rho \circ i}^+ (M) \cdot \tau_{\rho \circ j}^+ (N),\]
where $i:\pi_1M \to \pi_1 (M \cdot N)$ and 
$j:\pi_1 N \to \pi_1 (M \cdot N)$ are the induced maps from 
the inclusions $M \hookrightarrow M \cdot N$ and 
$N \hookrightarrow M \cdot N$. 
\end{itemize}
\end{proposition}
By an argument similar to $r_{\mathfrak{q}_k}$, 
we can obtain a crossed homomorphism 
\[\tau_{\mathfrak{q}_k}^+ : 
\Cgb \longrightarrow 
K_1 (\mathcal{K}_{N_k (\pi)})/(\pm N_k (\pi))\]
for $\Gamma=N_k (\pi)$.

\begin{example}
For the homology cylinder $M_L$ in Example \ref{ex:string_comp}, we 
have 
\[\det (\tau_{\mathfrak{q}_2}^+ (M_L))
= -1 + \gamma_3 -\gamma_3 \gamma_4^{-1}.\]
Since it is non-trivial, we see again that $M_L \notin \Mgb$.
\end{example}

\section{Applications of Magnus representations to homology cylinders}
\label{sec:application}
The final section presents a number of applications of Magnus 
representations to homology cylinders. 
The following subsections are 
independent of each other.

\subsection{Higher-order Alexander invariants and 
homologically fibered knots}\label{subsec:higher-order}

Let $G$ be a group and let 
$\rho:G \to \Gamma$ be a homomorphism 
into a PTFA group $\Gamma$. 
For a pair $(G,\rho)$, the {\it higher-order Alexander module} 
$\mathcal{A}^\rho (G)$ is defined by
\[\mathcal{A}^\rho (G):= H_1 (G;\Z [\Gamma]),\]
where $\Z [\Gamma]$ is regarded as a $\Z [G]$-module through 
$\rho$. 
{\it Higher-order Alexander invariants} generally indicate 
invariants derived from $\mathcal{A}^\rho (G)$. 
After having been defined 
and developed by Cochran-Orr-Teichner \cite{cot}, 
Cochran \cite{coc} and Harvey \cite{har,har2}, 
many applications to the theory of knots and 3-manifolds were 
obtained. 
In the theory of higher-order Alexander invariants, 
one of the important problems was to find methods 
for computing the invariants explicitly and extract 
topological information from them. 
This problem arises from 
the difficulty in non-commutative rings involved in the definition. 

Let $K$ be a knot in $S^3$. 
We fix a homomorphism 
$\rho:G(K)=\pi_1 (E(K)) \to \Gamma$ 
into a PTFA group $\Gamma$. 
It was shown in Cochran-Orr-Teichner \cite[Section 2]{cot} 
and Cochran \cite[Section 3]{coc} that 
$H_\ast (E(K);\mathcal{K}_\Gamma)=0$ if $\rho$ is non-trivial. 
Then we can define the torsion 
\[\tau_{\rho} (E(K)):=
\tau (C_\ast (E(K);\mathcal{K}_\Gamma)) 
\in K_1 (\mathcal{K}_\Gamma)/\pm \rho (G(K)). \]
Friedl \cite{fri} observed that this torsion $\tau_{\rho} (E(K))$ can be 
regarded as a higher-order Alexander invariant for $G(K)$. In the case where 
$\rho$ is the the abelianization map 
$\rho_1:G(K) \to \langle t \rangle$, 
Milnor's formula \cite{milnor2} 
$\tau_{\rho_1}(E(K))=\displaystyle\frac{\Delta_K (t)}{1-t}$ is recovered. 

We now try to understand the higher-order invariant $\tau_\rho (E(K))$ 
for a homologically fibered knot $K$ by factorizing it 
into the invariants we have seen in the previous section. 
The formula is given as follows:
\begin{theorem}[{\cite[Theorem 3.6]{gs10}}]\label{thm:factorization}
Let $K$ be a homologically fibered knot with a minimal genus 
Seifert surface $R$ of genus $g$ and let $(M_R,i_+,i_-) \in \Cgb$ be 
the corresponding homology cylinder. 
For any non-trivial homomorphism 
$\rho:G (K) \to \Gamma$ into 
a PTFA group $\Gamma$, a loop $\mu$ representing the meridian of 
$K$ satisfies $\rho(\mu) \neq 1 \in \Gamma$ and we have 
a factorization 
\begin{equation}\label{eq:factorization_eq}
\tau_\rho (E(K)) = \frac{\tau_{\rho}^+ (M_R) \cdot 
(I_{2g} -\rho (\mu) r_{\rho} (M_R))}{1-\rho(\mu)} \ \ 
\in K_1(\mathcal{K}_\Gamma)/\pm\rho(G(K))
\end{equation}
of the torsion $\tau_\rho (E(K))$. 
\end{theorem}
\noindent
When $K$ is a fibered knot and 
$\rho=\rho_1$, the abelianization map, 
we recover the formula (\ref{eq:alexander_fibered}) by using 
Milnor's formula mentioned above. 

The explicit computation of $\tau_\rho (E(K))$ is still difficult 
after the factorization (\ref{eq:factorization_eq}) in general. 
However, when 
we consider the projection $\rho_2: G(K) \to G(K)/G(K)^{(2)}$ to 
the metabelian quotient, which is known to be PTFA 
(see Strebel \cite{strebel}), 
then the situation gets interesting as follows. 

In the group extension
\[1\longrightarrow G(K)^{(1)}/G(K)^{(2)} 
 \longrightarrow G(K)/G(K)^{(2)} 
 \longrightarrow G(K)/G(K)^{(1)} \cong \mathbb{Z}
 \longrightarrow 1,\]
we have $G(K)^{(1)}/G(K)^{(2)} \cong H_{1}(R) \cong H_{1}(M_{R})$ 
since it coincides with 
the first homology of the infinite cyclic 
covering of $E(K)$, which can be seen as 
the product of infinitely many copies of $M_R$. 
In particular, we may regard $H \cong H_1 (M_{R})$ as 
a natural, independent of choices of minimal genus 
Seifert surfaces, 
subgroup of $G(K)/G(K)^{(2)}$. 
We can easily observe that 
$\tau_{\rho_2}^+ (M_R)=\tau_{\mathfrak{q}_2}^+ (M_R)$ and 
$r_{\rho_2} (M_{R}) = r_{\mathfrak{q}_2} (M_{R})$, 
namely they can be 
determined by computations on 
a commutative subfield $\mathcal{K}_{H} \cong \mathcal{K}_{H_1 (M_R)}$ in 
$\mathcal{K}_{G(K)/G(K)^{(2)}}$.

\begin{remark}
From the formula (\ref{eq:factorization_eq}) 
with the above observation, it seems reasonable 
to say that after applying the Dieudonn\'e determinant, 
$\tau_{\rho_2}^+ (M_R)=\tau_{\mathfrak{q}_2}^+ (M_R)$ 
is the ``bottom coefficient'' of 
$\tau_{\rho_2} (E(K))$ with respect to $\rho (\mu)$. 
Note that $\tau_{\mathfrak{q}_2}^+ (M_R)$ may be regarded as 
a special case of a {\it decategorification} of 
the sutured Floer homology as shown by Friedl-Juh\'asz-Rasmussen \cite{fjr}. 
\end{remark}

\begin{example}[{\cite[Example 6.7]{gs10}}]\label{ex:concordance}
Let $K$ and $K'$ be the knots obtained as 
the boundaries of the Seifert surfaces $R$ and $R'$ 
in Figure \ref{fig:concordance}. 
Here the side with the darker color in $R$ and $R'$ 
means the $+$-side. 

\begin{figure}[htbp]
\begin{center}
\includegraphics[width=.8\textwidth]{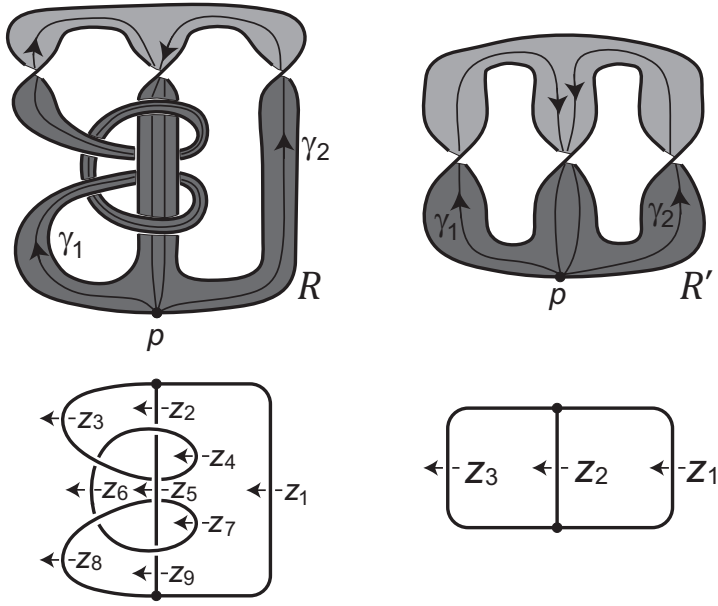}
\end{center}
\caption{Homologically fibered knots $K$ and $K'$ (Pictures are 
taken from \cite{gs10}.)}
\label{fig:concordance}
\end{figure}

$K'$ is the trefoil knot, which is a fibered knot with fiber $R'$. 
We can easily check that $K$ is a homologically fibered 
knot with a minimal genus Seifert surface $R$. 
It is known that $(M_R,i_+,i_-)$, $(M_{R'},j_+,j_-)$ 
give homology cobordant 
homology cylinders in $\mathcal{C}_{1,1}$. 
An admissible presentation of $\pi_1 M_R$ is given by 
\[{\small \left\langle
\begin{array}{c|c}
\begin{array}{c}
i_-(\gamma_1), i_-(\gamma_2)\\
z_1, \ldots, z_9\\
i_+(\gamma_1) , i_+(\gamma_2)\\
\end{array} &
\begin{array}{l}
z_{1}z_{2}z_{3},\,z_{1}z_{9}z_{8},\ 
z_{4}z_{5}z_{4}^{-1}z_2^{-1},\ 
z_{4}^{-1}z_{5}z_{3}^{-1}z_5^{-1},\\
z_{3}z_{6}z_{3}^{-1}z_4,\ 
z_{7}z_{5}z_{8}z_{5}^{-1},\ 
z_{7}^{-1}z_{9}z_{7}z_{5}^{-1},\\
i_{-}(\gamma_{1})z_{1}z_7z_{4}^{-1}z_2z_5^{-1}z_3z_8^{-1}z_5,\ 
i_{-}(\gamma_{2})z_{8}^{-1}z_7z_4^{-1}z_1^{-1},\\
i_{+}(\gamma_{1})z_7 z_{4}^{-1} z_2 z_5^{-1} z_3 z_8^{-1} z_5, \ 
i_{+}(\gamma_{2})z_{7}z_4^{-1}z_{1}^{-1}
\end{array}
\end{array}\right\rangle}.\]
\noindent
From this, we have 
\begin{align*}
\det (\tau_{\mathfrak{q}_2}^+ (M_R)) &= 3-\frac{1}{\gamma_1}-\gamma_1
-\frac{\gamma_1}{\gamma_2}+\frac{\gamma_1^2}{\gamma_2} 
+\frac{\gamma_2}{\gamma_1^2}-\frac{\gamma_2}{\gamma_1},\\
r_{\mathfrak{q}_2} (M_R) &= 
\begin{pmatrix}
1 & \gamma_2^{-1}\\
-\gamma_1^{-1} \gamma_2 & 1-\gamma_1^{-1}
\end{pmatrix},
\end{align*}
\noindent
where the value of $\det (\tau_{\mathfrak{q}_2}^+ (M_R))$ 
shows that $K$ is not fibered. 
On the other hand, an admissible presentation of 
$\pi_1 M_{R'}$ is given by 
\[{\small \left\langle
\begin{array}{c|c}
\begin{array}{c}
i_-(\gamma_1), i_-(\gamma_2)\\
z_1, z_2, z_3\\
i_+(\gamma_1) , i_+(\gamma_2)\\
\end{array} &
\begin{array}{l}
z_{1}z_{2}z_{3},\ i_-(\gamma_1)z_{3}^{-1},\ 
i_-(\gamma_2)z_3^{-1}z_1^{-1}, \\
i_+(\gamma_1) z_2,\ 
i_+(\gamma_2)z_1^{-1}
\end{array}
\end{array}\right\rangle}\]
and we have
\begin{align*}
\det (\tau_{\mathfrak{q}_2}^+ (M_R)) &= \frac{1}{\gamma_2},\\
r_{\mathfrak{q}_2} (M_R) &= 
\begin{pmatrix}
1 & \gamma_2^{-1}\\
-\gamma_1^{-1} \gamma_2 & 1-\gamma_1^{-1}
\end{pmatrix}.
\end{align*}
\end{example}
From this example, we see that the torsion $\tau_\rho^+$ is not 
preserved under homology cobordism relation in general. 
See also the formula (\ref{formula:hcob}) mentioned later. 
More examples are exhibited in \cite{gs10} with particular 
interest in non-fiberedness of homologically fibered knots. 

\subsection{Bordism invariants and signature invariants}
\label{subsec:signature}

In this subsection, we introduce two kinds of invariants 
of homology cylinders of topological nature: {\it bordism 
invariants} and {\it signature invariants}. Then we discuss 
how Magnus matrices behave in their interrelationship. 

Let us first introduce bordism invariants, 
which naturally generalize 
those for $\Mgb$ given by Heap \cite{he}. 
Since (the infinitesimal version of) those homomorphisms are 
fully discussed in the chapter of Habiro-Massuyeau 
\cite[Section 3.3]{habiromassuyeau}, 
we here recall it briefly. 

Let $(M,i_+,i_-) \in \Cgb[k]$. Then 
we have $i_+=i_-:N_k (\pi) \xrightarrow{\cong} N_k (\pi_1 M)$. 
Consider the composition 
\[f_M: M \longrightarrow K(\pi_1 M,1) \longrightarrow 
K(N_k(\pi_1 M),1) \xrightarrow{(i_+)^{-1}=
(i_-)^{-1}} K(N_k (\pi),1) \]
of continuous maps. We can assume 
$f_M \circ i_+ = f_M \circ i_-:\Sgb \to K(N_k (\pi),1)$ 
after adjusting by homotopy, if necessary. 
$f_M$ induces a continuous map 
$\widetilde{f}_M:C_M \to K(N_k,1)$ from the closure $C_M$ of $M$. 
Define a map $\theta_k : \Cgb [k] \to \Omega_3 (N_k (\pi))$ by 
\[\theta_k (M,i_+,i_-) := (C_M,\widetilde{f}_M),\]
where $\Omega_3 (N_k (\pi))$ denotes the third bordism group of 
$K(N_k (\pi),1)$. 
Then we have the following. 
\begin{theorem}[{\cite[Theorem 7.1]{sa2}}]\label{thm:invariant}
For $k \ge 2$, $\theta_k$ is a homomorphism and 
factors through $\Hgb [k]$. Moreover, 
the induced homomorphism 
$\theta_k :\Hgb [k] \to \Omega_3 (N_k (\pi))$ 
gives an exact sequence
\[1 \longrightarrow \Hgb[2k-1] \longrightarrow \Hgb[k] 
\stackrel{\theta_k}{\longrightarrow} 
\Omega_3 (N_k (\pi)) \longrightarrow 1.\]
\end{theorem}
\noindent
\begin{proof}[Sketch of Proof]
The proof is divided into the following steps:
\begin{enumerate}
\item $\theta_k$ factors through $\Hgb [k]$; 
\item $\theta_k$ is actually a homomorphism;
\item $\theta_k$ is onto; 
\item $\Ker \theta_k = \Hgb [2k-1]$.
\end{enumerate}
(1) and (2) follow from standard topological constructions. 
We use arguments in Orr \cite{orr} 
and Levine \cite{le1} to 
reduce the proof of (3) to that of Theorem \ref{thm:autsp}. 
The proof of (4) proceeds as follows. 
We have a natural 
isomorphism $\Omega_3 (N_k (\pi)) \cong H_3 (N_k (\pi))$ 
by assigning $f([X]) \in H_3 (N_k (\pi))$ to 
$(X,f) \in \Omega_3 (N_k (\pi))$, 
where $[X] \in H_3 (X)$ is the fundamental class of a closed 
oriented 3-manifold $X$. 
Igusa-Orr \cite{io} showed that 
the homomorphism $H_3 (N_{2k-1} (\pi)) \to H_3 (N_k (\pi))$ 
induced by the natural projection $N_{2k-1} (\pi) \to N_k (\pi)$ 
is trivial. From this, we see that 
$\Hgb[2k-1] \subset \Ker \theta_k$. On the other hand, 
the induced homomorphism 
$\theta_k :\Hgb[k]/\Hgb[2k-1] \to \Omega_3 (N_k (\pi))$ 
turns out to be an epimorphism between 
free abelian groups of the same rank, which shows that 
it is an isomorphism. In particular, the identity 
$\Hgb[2k-1] = \Ker \theta_k$ follows. 
\end{proof}

Next, we briefly review Atiyah-Patodi-Singer's 
$\rho\,$-invariant in \cite{aps1,aps2}. 
Let $(M,g)$ be a $(2l-1)$-dimensional compact 
oriented Riemannian manifold, 
and let $\alpha:\pi_1 M \to U(m)$ be a unitary representation. 
Consider the self-adjoint operator 
$B_\alpha:\Omega^{\mathrm{even}}(M;V_\alpha) \to 
\Omega^{\mathrm{even}}(M;V_\alpha)$ 
on the space of all differential 
forms of even degree on $M$ with values in the flat bundle $V_\alpha$ 
associated with $\alpha$ defined by 
\[B_\alpha \varphi :=(\sqrt{-1})^l (-1)^{p+1} 
(\ast d_\alpha -d_\alpha \ast) \varphi\]
for $\varphi \in \Omega^{2p} (M;V_\alpha)$. 
Here, $\ast$ is the Hodge star operator. 
Then we define the spectral function $\eta_\alpha (s)$ of $B_\alpha$ by 
\[\eta_\alpha (s):=\sum_{\lambda \neq 0} (\mathrm{sign} \lambda) 
| \lambda |^{-s}, \]
where $\lambda$ runs over all non-zero eigenvalues of $B_\alpha$ with 
multiplicities. This function converges to an analytic function 
for $s \in \mathbb{C}$ having sufficiently large real part, 
and is continued analytically 
as a meromorphic function on the complex plane so that 
it takes a finite value 
at $s=0$. 
The value $\eta_{\alpha} (0)$ is called the {\it $\eta\,$-invariant} 
of $(M,g)$ associated with $\alpha$. We simply write $\eta (0)$ for 
the $\eta\,$-invariant associated with the trivial representation 
$\pi_1 M \to U(1)$. 
\begin{theorem}[Atiyah-Patodi-Singer \cite{aps2}]
The value 
\[\rho_{\alpha} (M):=\eta_{\alpha} (0)-m \cdot \eta (0)\]
does not 
depend on a metric of $M$, so that it defines a diffeomorphism 
invariant of $M$ called the $\rho\,$-invariant 
associated with $\alpha$. 
Moreover, if there exists a compact smooth manifold $N$ such that 
$M=\partial N$ and if $\alpha$ can be extended 
to a unitary representation $\widetilde{\alpha}:\pi_1 N \to U(m)$ of 
$\pi_1 N$, then 
\[\rho_\alpha (M) = m \cdot \mathrm{sign} (N) - 
\mathrm{sign}_{\widetilde{\alpha}} (N)\]
holds, where $\mathrm{sign} (N)$ and $\mathrm{sign}_{\widetilde{\alpha}} (N)$ 
denote the signature and the twisted signature of $N$. 
\end{theorem}

Levine \cite{le4} applied the theory of $\rho\,$-invariants 
to the following situation and obtained some invariants of links. 
Let $R_m (G)$ be the space of all unitary representations 
$G \to U(m)$ of a group $G$. If $G$ is generated by $l$ elements, 
$R_m (G)$ can be realized as a real algebraic subvariety of 
the direct product $U(m)^{\times l}$ of $l$-tuples of $U(m)$. 
We endow $R_m (G)$ with the usual (Hausdorff) topology 
as a subspace of $U(m)^{\times l}$. 

For a pair $(M,\alpha)$ consisting of 
an odd-dimensional closed manifold $M$ and a group 
homomorphism $\alpha:\pi_1 M \to G$, we define a function 
\[\sigma (M,\alpha):R_m (G) \longrightarrow \R\]
by $\sigma (M,\alpha)(\theta) := \rho_{\theta \circ \alpha} (M)$. 
This function has the following properties. 
\begin{theorem}[Levine \cite{le4}]
$(1)$ \ For each pair $(M,\alpha)$, there exists a proper algebraic 
subvariety $\Sigma$ of $R_m (G)$ such that 
$\sigma (M,\alpha) \bigl|_{R_m (G) - \Sigma}$ is a continuous 
real valued function. \\
$(2)$ \ If $(M,\alpha)$ and $(M',\alpha')$ are homology $G$-bordant, 
there exists an algebraic subvariety $\Sigma'$ of $R_m (G)$ 
such that 
\[\sigma (M,\alpha)\bigl|_{R_m (G) - \Sigma'} = 
\sigma (M',\alpha')\bigl|_{R_m (G) - \Sigma'}.\]
\end{theorem}
\noindent
Here, two pairs $(M,\alpha), (M',\alpha')$ are said to be 
{\it homology $G$-bordant} 
if there exists a pair $(N,\widetilde{\alpha})$ such that 
$\partial N=M' \cup -M$, 
$H_\ast (N,M)=H_\ast (N,M')=0$, and the pullback of 
$\widetilde{\alpha}$ on $\pi_1 M$ (resp. $\pi_1 M'$) coincides 
with $\alpha$ (resp. $\alpha'$) up to conjugation in $G$. Note that 
by an argument in \cite{aps3}, $\sigma (M,\alpha) \bmod \Z$ is 
continuous on $R_m (G)$. From this, we can show that 
$\sigma (M,\alpha)$ is a bounded function on $R_m (G)$. 

Now we return to our situation. 
We now consider $R_1 (N_2 (\pi))=R_1 (H)$  
to construct an invariant of $\Hgb [2]$. 
Fix a diffeomorphism $R_1 (H) \cong T^{2g}$, 
where $T^{2g}$ denotes the $2g$-dimensional torus, by using a basis of 
$H$. We give a standard measure $d\theta$ normalized by 
$\int_{T^{2g}} \, d\theta =1$ to $T^{2g}$. Then we define 
\[\rho_{H,1} : \Hgb [2] \longrightarrow \R\]
by
\[\rho_{H,1}(M,i_+,i_-) := \int_{T^{2g}} \sigma 
(C_M,\widetilde{f}_M)(\theta) \, d\theta. \]
Note that for each element of $\Hgb [2]$, 
$(C_M,\widetilde{f}_M)$ is uniquely determined up to 
homology $H$-bordism. Since 
$\sigma (C_M,\widetilde{f}_M)$ is bounded, continuous 
and takes the same value for two homology $H$-bordant 
manifolds almost everywhere in $T^{2g}$, 
the map $\rho_{H,1}$ is well-defined. 

\begin{theorem}\label{thm:mainthm} 
The map $\rho_{H,1} : \Hgb [2] \to \R$ has the following properties: 
\begin{itemize}
\item[$(1)$] The restriction of $\rho_{H,1}$ to 
$\Ker r_{\mathfrak{q}_2}$ 
is a homomorphism; 

\item[$(2)$] $\rho_{H,1}(\Hgb [[\infty]])$ is an infinitely generated 
$($over $\Z \,)$ subgroup of $\R$. 
\end{itemize}
\end{theorem}
\begin{proof}
For $k=2$, the bordism invariant $\theta_2$ gives an exact sequence
\[1 \longrightarrow \Hgb[3] \longrightarrow 
\Hgb[2] \stackrel{\theta_2}{\longrightarrow} 
\Omega_3 (H) \longrightarrow 1.\]
From this, we see that if $(M,i_+,i_-) \in \Hgb [3]$, then 
the pair consisting of the closure $C_M$ of $M$ and 
the homomorphism $\widetilde{f}_M:\pi_1 C_M \to H$ induced from 
the continuous map $\widetilde{f}_M:C_M \to K(H,1)$ is 
the boundary of a pair $(W_M,f_{W_M})$. 
Then the function $\sigma (C_M,\widetilde{f}_M)$ 
has an interpretation as a signature defect and 
\begin{align*}
\rho_{H,1} (M,i_+,i_-) &= \int_{T^{2g}} \sigma 
(C_M,\widetilde{f}_M)(\theta) \, d\theta \\
&= \int_{T^{2g}} \bigl( \mathrm{sign}(W_M) - 
\mathrm{sign}_{\theta \circ f_{W_M}}(W_M) \bigl)\, d\theta \\
&= \mathrm{sign}(W_M) - \int_{T^{2g}} 
\mathrm{sign}_{\theta \circ f_{W_M}}(W_M) \, d\theta
\end{align*}
follows, where $\mathrm{sign}_{\theta \circ f_{W_M}}(W_M)$ is the 
signature of the intersection form induced on 
$H_2 (W_M;\mathbb{C}_{\theta \circ f_{W_M}})$ with coefficients 
in the left $\pi_1 W_M$-module $\mathbb{C}$ 
on which $\pi_1 W_M$ acts 
through $\theta \circ f_{W_M}:\pi_1 W_M \to U(1)$. 
To show (1), it suffices to show that both 
$\mathrm{sign}(W_M)$ and $\mathrm{sign}_{\theta \circ f_{W_M}}(W_M)$ 
are additive. 

Let $M_1=(M_1,i_+,i_-)$, $M_2=(M_2,j_+,j_-) \in \Ker r_{\mathfrak{q}_2}$. 
Note that $\Ker r_{\mathfrak{q}_2} \subset \Hgb [3]$. We take 
a pair $(W_{M_i},f_{W_{M_i}})$ satisfying 
$(C_{M_i},\widetilde{f}_{M_i}) = \partial (W_{M_i},f_{W_{M_i}})$. 
By performing surgeries on $W_{M_i}$ preserving the $H$-bordism class, 
if necessary, we can assume that 
$\pi_1 W_{M_i} \cong H_1 (W_{M_i}) \cong H$. 
Then the manifold 
\[W:=W_{M_1} \cup_{\Sgb \times [0,1]} W_{M_2}\]
obtained from $W_{M_1}$ and $W_{M_2}$ by 
gluing along $\Sgb \times [0,1] \subset C_{M_i}$ 
together with the homomorphism $f_W:=f_{W_{M_1}} \cup f_{W_{M_2}}$ 
satisfy 
$\partial (W,f_W)=(M_1 \cdot M_2, \widetilde{f}_{M_1 \cdot M_2})$. 
See Figure \ref{fig:wall}. 

\begin{figure}[htbp]
\begin{center}
\includegraphics{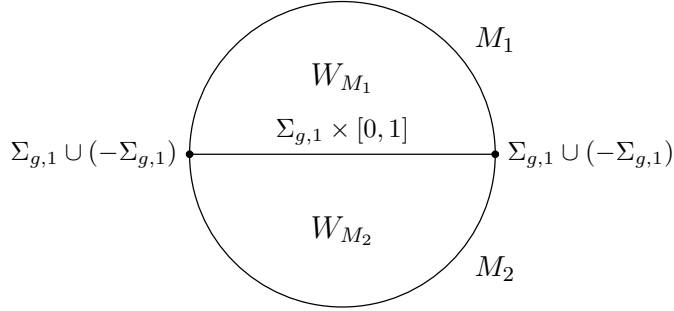}
\end{center}
\caption{The manifold $W$}
\label{fig:wall}
\end{figure}

\noindent
If we apply Wall's non-additivity theorem \cite{wa} of 
signatures to $W,W_{M_1},W_{M_2}$, we see 
that the correction term is zero when $M_1,M_2 \in \Hgb [2]$ by 
an argument associated with the Meyer cocycle \cite{me}, and 
therefore the additivity of signatures follows. 

For the additivity of $\mathrm{sign}_{\theta \circ f_{W_M}}(W_M)$, 
we need to use a local coefficient system version of Wall's theorem 
in \cite{me}. We can see that if the Magnus matrix 
$r_{\mathfrak{q}_2}$ is trivial, 
the correction term is zero. Indeed, 
under the observation that $\mathbb{C}_{\theta \circ f_{W_{M_i}}}$ becomes 
a $\mathcal{K}_H$-vector space almost everywhere in $R_1 (H)$, 
the vector spaces which appear in the calculation of the 
correction term coincide with each other 
if their Magnus matrices are trivial, and 
therefore the correction term is zero. 
Since the correction term is zero almost everywhere in $R_1 (H)$, 
their integration on $R_1 (H)$ becomes additive. 

(2) is shown by the following explicit examples. 
Note that these examples are based on those in 
Cochran-Orr-Teichner \cite{cot,cot2} and Harvey \cite{har3} 
to show the infinite generation of 
some subgroups of the knot (or string link) concordance group. 
For $(\Sgb \times [0,1], \Id \times 1, \Id \times 0) \in \Cgb$, 
we take a loop $l$ in the interior of $\Sgb \times [0,1]$ representing 
$\gamma_1 \in H \cong H_1 (\Sgb \times [0,1])$. 
We remove an open tubular neighborhood $N(l)$ of $l$ from 
$\Sgb \times [0,1]$ and then glue 
the exterior $E(K)$ of a knot $K \subset S^3$ so that 
the canonical longitude (resp. the meridian) of $E(K)$ 
corresponds to the meridian (resp. the inverse of the longitude) 
of $N(l)$. 
We can check that the resulting manifold $M_K$ becomes 
a homology cylinder. Moreover it belongs to  $\Hgb [[\infty]]$ since 
$\pi_1 (\Sgb \times [0,1]-N(l)) \to \pi_1 (\Sgb \times [0,1]) 
\cong \pi \to \AC{\pi}$ extends to $\pi_1 M_K$. 
Then we can show that 
\[\rho_{H,1}(M_K)=\int_{\theta \in S^1} \sigma_\theta (K)\, 
d\theta, \]
where $\sigma_\theta (K)$ is the Levine-Tristram signature of 
the knot $K$ associated with $\theta \in S^1$. 
It was shown in \cite[Section 5]{cot2} that 
the above values move around an infinitely generated subgroup of $\R$ 
when $K$ runs over all knots. Therefore (2) follows. 
\end{proof}
\begin{corollary}
The groups $\Ker r_{\mathfrak{q}_2}$, 
$\Hgb [[\infty]]$ and their abelianizations are 
all infinitely generated. 
\end{corollary}

We can consider results of Cochran-Harvey-Horn \cite{coc_har_horn} 
to be a further generalization of the invariant $\rho_{H,1}$. 
They constructed 
{\it von Neumann $\rho$-invariants} for homology cylinders 
by using the theory of $L^2$-signature invariants. 
Note that Magnus matrices also appear in their context as 
obstructions to the additivity of invariants. In fact, 
we can see that the correction term vanishes under the triviality 
of the corresponding Magnus matrix by rewriting 
Wall's argument \cite{wa} word-by-word in terms of $L$-groups. 

We close this subsection by posing the following problem:
\begin{problem}
Determine $H_3 (\Acy_n)$.
\end{problem}
\noindent
This problem is an analogue of a similar problem 
for the algebraic closure of a free group. 
It was shown in \cite{sa2} that 
the bordism invariant similar to $\theta_k$ gives an epimorphism 
$\theta: \Hgb [[\infty]] \twoheadrightarrow H_3 (\AC{\pi})$. 
At present, however, 
we cannot extract any information on $\Hgb [[\infty]]$ from $\theta$ 
since it is not known even whether $H_3 (\Acy_n)$ is trivial or not.

\subsection{Abelian quotients of groups of homology cylinders}
\label{subsec:abelian}
Abelian quotients of a monoid or group are helpful 
not only to know how big the monoid or group is, but 
to extract information on its structure. In this subsection, we focus on 
abelian quotients of monoids and homology cobordism groups 
of homology cylinders and we compare them to 
the corresponding results for mapping class groups. 
We assume that $g \ge 1$. 

As we have seen in Sections \ref{subsec:JohnsonMorita} and 
\ref{subsec:JMhomHC}, 
the Johnson homomorphisms give finite rank 
abelian quotients of $\Mgb[k]$, $\Cgb[k]$ and $\Hgb[k]$ for 
each $k \ge 2$. 
Indeed the image of $\Cgb[k]$ and $\Hgb[k]$ is generally bigger than 
that of $\Mgb[k]$. 

Before discussing further, as commented in \cite{gs09}, 
we point out that $\Cgb$ has the monoid $\theta_\Z^3$ 
of homology 3-spheres as an abelian quotient. 
In fact, we have 
a {\it forgetful} homomorphism $F:\mathcal{C}_{g,1} 
\longrightarrow \theta_\Z^3$ defined by 
$F(M,i_+,i_-) = S^3 \sharp X_1 \sharp X_2 \sharp 
\cdots \sharp X_n$ for the prime decomposition 
$M=M_0 \sharp X_1 \sharp X_2 \sharp 
\cdots \sharp X_n$ of $M$ with $M_0 \in \Cgb^\mathrm{irr}$ and 
$X_i \in \theta_\Z^3$ $(1 \le i \le n)$ 
(recall Section \ref{subsec:defHC}). 
The map $F$ owes its well-definedness to 
the uniqueness of the prime decomposition of 3-manifolds. 
The map $F$ gives a splitting of the construction of 
Example \ref{ex:sphere} and is surjective. 

The underlying 3-manifolds of homology cylinders obtained 
from $\mathcal{M}_{g,1}$ are all 
$\Sgb \times [0,1]$ and, in particular, irreducible. 
Therefore it seems more reasonable to compare $\Mgb$ with 
$\Cgb^\mathrm{irr}$. Until now, 
many infinitely generated abelian quotients for 
monoids and homology cobordism groups of irreducible 
homology cylinders have been given, which are completely 
different from the corresponding cases 
for mapping class groups. We present them in order. 

\begin{theorem}[{\cite[Corollary 6.16]{sakasai08}}]
The submonoids $\Cgb^\mathrm{irr} \cap \Cgb[k]$ for $k \ge 2$ and 
$\Ker (\Cgb^\mathrm{irr} \to \Hgb)$ have abelian quotients 
isomorphic to $(\Z_{\ge 0})^\infty$.
\end{theorem}
\noindent
The proof uses homomorphisms constructed from the 
torsions $\tau_{\mathfrak{q}_k}^+$. Precisely speaking, 
irreducibility was not discussed in \cite{sakasai08}. However, 
we can modify the argument. 

\begin{theorem}[{Morita \cite[Corollary 5.2]{morita_GD}}]
\label{thm:moritaGD}
$\Hgb [2]$ has an abelian quotient isomorphic to $\Z^\infty$. 
\end{theorem}
\noindent
For the construction, he used the trace maps mentioned 
in Remark \ref{rem:trace} 
with a deep observation of the Johnson filtration of $\Hgb$. 

It was shown by Harer \cite{harer} that 
$\Mgb$ is a perfect group for $g \ge 3$ 
(see also Farb-Margalit \cite{Farb_Margalit}). 
By taking into account the similarity between 
$\Mgb$, $\Cgb^\mathrm{irr}$ and $\Hgb$ as we have seen, 
it had been conjectured that 
$\Cgb^\mathrm{irr}$ and $\Hgb$ do not have 
non-trivial abelian quotients. However, Goda and 
the author showed the following:
\begin{theorem}[{\cite[Theorem 2.6]{gs09}}]
The monoid $\Cgb^\mathrm{irr}$ has an abelian quotient 
isomorphic to $(\Z_{\ge 0})^\infty$.
\end{theorem}
\begin{proof}[Sketch of Proof]
The proof uses 
some results of {\it sutured Floer homology} 
(a variant of Heegaard Floer homology) 
developed by Ni \cite{ni0,ni1} 
and Juh\'asz \cite{juhasz, juhasz2}. 

For each homology cylinder $(M,i_+,i_-) \in \Cgb$, 
we have a natural decomposition 
$\partial M = i_+(\Sgb) \cup_{i_+(\partial \Sgb)=i_-(\partial \Sgb)} 
i_-(\Sgb)$ of $\partial M$. Such a decomposition 
defines a {\it sutured manifold} $(M,\zeta)$ with the suture 
$\zeta =i_+(\partial \Sgb)=i_-(\partial \Sgb)$. 
Since the sutured manifold obtained from 
a homology cylinder is {\it balanced} in the sense of 
Juh\'asz \cite[Definition 2.2]{juhasz}, 
the sutured Floer homology $SFH(M,\zeta)$ is defined. 
By taking the rank of $SFH$, we obtain a map 
$R: \Cgb^\mathrm{irr} \longrightarrow \Z_{\ge 0}$ 
defined by $R(M,i_+,i_-) =  \mathrm{rank}_\Z (SFH(M,\zeta))$. 
Deep results of Ghiggini \cite{ghi}, Ni \cite{ni0,ni1} 
and Juh\'asz \cite{juhasz, juhasz2} show that the map 
$R$ is a monoid homomorphism from $\Cgb^\mathrm{irr}$ to 
the monoid $\Z_{>0}^\times$ of positive integers 
whose product is given by multiplication. 
By the uniqueness of the prime decomposition of an integer, 
we can decompose $R$ into prime factors
\[R=\bigoplus_{\scriptsize \mbox{$p$\,:\,prime}} 
R_p : \Cgb^\mathrm{irr} \longrightarrow 
\Z_{>0}^\times = 
\bigoplus_{\scriptsize \mbox{$p$\,:\,prime}} \Z_{\ge 0}^{(p)},\]
where $\Z_{\ge 0}^{(p)}$ is a copy of $\Z_{\ge 0}$, 
the monoid of non-negative integers 
whose product is given by sum. We can check that 
$\{R_p:\Cgb^\mathrm{irr} \to \Z_{\ge 0}\mid \mbox{$p$\,:\,prime}\}$ 
contains infinitely many non-trivial homomorphisms 
that are linearly independent as homomorphisms 
from $\Cgb^\mathrm{irr}$ to $\Z_{\ge 0}$. 
\end{proof}
\noindent
It was also observed in \cite{gs09} that the above homomorphisms 
$R_p$ are not homology cobordism invariants. 

As seen in Example \ref{ex:concordance}, 
the $\Gamma$-torsion generally changes 
under homology cobordism. 
However, Cha-Friedl-Kim \cite{cfk} found 
a way to extract homology cobordant invariants 
of homology cylinders from 
the torsion 
\[\tau_{\mathfrak{q}_2}^+: \Cgb \longrightarrow 
\mathcal{K}_H^\times/(\pm H),\]
which is a crossed homomorphism, as follows. 

First they consider the subgroup $A \subset \mathcal{K}_H^\times$ 
defined by
\[A:=\{f^{-1} \cdot \varphi (f) \mid f \in \mathcal{K}_H^\times, \ 
\varphi \in \mathrm{Sp} (2g,\Z)\},\]
by which we can obtain a {\it homomorphism} 
\begin{equation}\label{eq:torsionhom}
\tau_{\mathfrak{q}_2}^+ : 
\Cgb \longrightarrow \mathcal{K}_H^\times/(\pm H \cdot A).
\end{equation}
\noindent
Note that $f=\overline{f}$ holds in 
$\mathcal{K}_H^\times/(\pm H \cdot A)$ since $-I_{2g} \in 
\mathrm{Sp}(2g,\Z)$. 
Second, they observe that if $(M,i_+,i_-), (N,j_+,j_-) \in \Cgb$ are 
homology cobordant, then there exists 
$q \in \mathcal{K}_H^\times$ such that 
\begin{equation}\label{formula:hcob}
\tau_{\mathfrak{q}_2}^+(M)= 
\tau_{\mathfrak{q}_2}^+(N) \cdot q \cdot \overline{q} 
\in \mathcal{K}_H^\times/(\pm H)
\end{equation}
\noindent
by using torsion duality. Note that a similar formula treating 
general situations had been obtained 
by Turaev \cite[Theorem 1.11.2]{turaev_knot}.
From this, we see that if we put  
\[N:=\{f \cdot \overline{f} \mid f \in \mathcal{K}_H^\times \},\]
then we can finally obtain a homomorphism 
\[\tau_{\mathfrak{q}_2}^+: 
\Hgb \longrightarrow \mathcal{K}_H^\times/(\pm H \cdot A \cdot N).\]
Note that 
$f^2=1$ holds for any $f \in \mathcal{K}_H^\times/(\pm H \cdot A \cdot N)$. 

We can see the structure 
of $\mathcal{K}_H^\times/(\pm H \cdot A \cdot N)$ as follows. 
Recall that $\mathcal{K}_H = \Z [H](\Z [H] -\{0\})^{-1}$. 
The ring $\Z [H]$ is a Laurent polynomial ring of $2g$ variables and 
it is a unique factorization domain. 
Thus every Laurent polynomial $f$ is factorized 
into irreducible polynomials
uniquely up to multiplication by a unit in $\Z [H]$. 
In particular, for every irreducible polynomial $\lambda$, 
we can count the exponent of $\lambda$ in the factorization 
of $f$. 
This counting naturally extends to that for elements in 
$\mathcal{K}_H^\times$. Under the identification by 
$\pm H \cdot A \cdot N$, an element in 
$\mathcal{K}_H^\times/(\pm H \cdot A \cdot N)$ is determined by 
the exponents of all $\mathrm{Sp}(2g,\Z)$-orbits 
of irreducible polynomials 
(up to multiplication by a unit in $\Z [H]$) modulo $2$. 
Hence $\mathcal{K}_H^\times/(\pm H \cdot A \cdot N)$ is 
isomorphic to $(\Z/2\Z)^\infty$. 
Finally by using 
$\Z_2$-torsion of the knot concordance group, they show the following:
\begin{theorem}[Cha-Friedl-Kim \cite{cfk}]\label{thm:cfk}
The homomorphism 
\[\tau_{\mathfrak{q}_2}^+: 
\Hgb \longrightarrow \mathcal{K}_H^\times/(\pm H \cdot A \cdot N)\]
is not surjective but its image is isomorphic to $(\Z/2\Z)^\infty$. 
\end{theorem}
\begin{remark}
Cha-Friedl-Kim showed that the same statement as above holds for 
$\Hg{g}{0}$. Moreover, they considered abelian quotients of the other 
$\Hg{g}{n}$ and showed that a similar construction gives an epimorphism 
$\Hg{g}{n} \twoheadrightarrow 
(\Z/2\Z)^\infty \oplus \Z^\infty$ if $n \ge 2$ (when $g \ge 1$) or 
$n \ge 3$ (when $g=0$). 
\end{remark}

Now we return to our discussion on applications of 
Magnus representations. We use the above Cha-Friedl-Kim's idea. 
Since Magnus representations are homology cobordism invariant, 
we have two maps
\begin{align*}
\widehat{r}_{\mathfrak{q}_2} &: \Hgb 
\xrightarrow{r_{\mathfrak{q}_2}}
\mathrm{GL}(2g,\mathcal{K}_H)
\xrightarrow{\det} \mathcal{K}_H^\times 
\longrightarrow \mathcal{K}_H^\times/(\pm H),\\
\widetilde{r}_{\mathfrak{q}_2} &: \Hgb 
\xrightarrow{\widehat{r}_{\mathfrak{q}_2}}
\mathcal{K}_H^\times/(\pm H)
\longrightarrow \mathcal{K}_H^\times/(\pm H \cdot A).
\end{align*}
\noindent
While $\widehat{r}_{\mathfrak{q}_2}$ is a crossed homomorphism, 
its restriction to $\Hgb [2]$ and 
$\widetilde{r}_{\mathfrak{q}_2}$ are homomorphisms. 
Note that both $\mathcal{K}_H^\times/(\pm H)$ and 
$\mathcal{K}_H^\times/(\pm H \cdot A)$ are isomorphic to 
$\Z^\infty$. 

\begin{theorem}[{\cite{sakasai10}}]
\begin{itemize}
\item[$(1)$] For $g \ge 1$ and $(M,i_+,i_-) \in \Cgb$, 
the equality
\[\widehat{r}_{\mathfrak{q}_2}(M) = 
\overline{\tau_{\mathfrak{q}_2}^+(M)} \cdot 
(\tau_{\mathfrak{q}_2}^+(M))^{-1} \ \ 
\in \mathcal{K}_H^\times/(\pm H)\]
holds. 

\item[$(2)$] For $g \ge 1$, the homomorphism 
$\widetilde{r}_{\mathfrak{q}_2}: \Hgb \to 
\mathcal{K}_H^\times/(\pm H \cdot A)$ is trivial. 

\item[$(3)$] For $g \ge 2$, the homomorphism 
$\widehat{r}_{\mathfrak{q}_2}: \Hgb [2] \to 
\mathcal{K}_H^\times/(\pm H)$ is not surjective but 
its image is isomorphic to $\Z^\infty$. 
\end{itemize}
\end{theorem}
\begin{proof}[Sketch of Proof]
$(1)$ can be shown by using the formula (\ref{eq:mag_formula}) 
and torsion duality. 
As mentioned above, the action of $\mathrm{Sp}(2g,\Z)$ implies that 
$f=\overline{f}$ for any $f \in \mathcal{K}_H^\times/(\pm H \cdot A)$. 
Then our claim $(2)$ immediately follows from $(1)$. 
To show $(3)$, we use 
the homology cylinder $M_L \in \Cg{2}{1}$ 
in Example \ref{ex:string_comp}. 
While $M_L \notin \Cg{2}{1}[2]$, we 
can adjust it by some $g_1 \in \Mg{2}{1}$ so that 
$M_L \cdot g_1 \in \Cg{2}{1}[2]$. Since 
$\widehat{r}_{\mathfrak{q}_2}$ is trivial on $\Mg{2}{1}$, 
we have 
\[\widehat{r}_{\mathfrak{q}_2} (M_L \cdot g_1) = 
\widehat{r}_{\mathfrak{q}_2} (M_L) =
\frac{\gamma_3 + \gamma_4-1}{\gamma_3^{-1}+\gamma_4^{-1}-1}
\in \mathcal{K}_H^\times/(\pm H).\] 
Take $f \in \Mg{2}{1}$ such that $\sigma_2 (f) \in \mathrm{Sp}(4,\Z)$ maps 
\[\gamma_1 \longmapsto \gamma_1 + \gamma_4, 
\quad \gamma_2 \longmapsto \gamma_2, 
\quad \gamma_3 \longmapsto \gamma_2 + \gamma_3, \quad 
\gamma_4 \longmapsto \gamma_4.\] 
Consider $f^m \cdot M_L \in \Cg{2}{1}$ 
and adjust it by 
some $g_m \in \Mg{2}{1}$ so that $f^m \cdot M_L \cdot g_m \in \Cg{2}{1}[2]$. 
Then we have 
\[\widehat{r}_{\mathfrak{q}_2} (f^m \cdot M_L \cdot g_m) = 
{}^{\sigma_2 (f^m)} (\widehat{r}_{\mathfrak{q}_2} (M_L))=
\frac{\gamma_2^m \gamma_3 + \gamma_4-1}
{\gamma_2^{-m}\gamma_3^{-1}+\gamma_4^{-1}-1} 
\in \mathcal{K}_H^\times/(\pm H).\] 
We can check that the values 
$\left\{\dis\frac{\gamma_2^m \gamma_3 + \gamma_4-1}
{\gamma_2^{-m}\gamma_3^{-1}+\gamma_4^{-1}-1}\right\}_{m=0}^\infty$ 
generate an infinitely generated subgroup of 
$\mathcal{K}_H^\times/(\pm H)$. This completes the proof when $g=2$. 
We can use the above computation 
for $g \ge 3$.
\end{proof}
\noindent
Consequently, we obtain a result similar to Theorem \ref{thm:moritaGD}.

\subsection{Generalization to higher-dimensional cases}\label{subsec:higher}

We can consider homology cylinders over $X$ for 
any compact oriented connected $k$-dimensional manifold $X$ 
with $k \ge 3$ by rewriting Definition \ref{def:HC} word-by-word. 
Let $\mathcal{M} (X)$, $\mathcal{C} (X)$ and $\mathcal{H} (X)$ denote 
the corresponding diffeotopy group, monoid of homology cylinders and 
homology cobordism group of homology cylinders. We have natural 
homomorphisms 
\[\SelectTips{cm}{}
\xymatrix{
\mathcal{M} (X) \ar[r] & \mathcal{C} (X) 
\ar@{>>}[r] & \mathcal{H} (X)}\]
and we can apply the argument in Section \ref{sec:universal} to 
$\mathcal{C} (X)$ and $\mathcal{H} (X)$. 

For $k \ge 2$ and $n \ge 1$, we put 
\[X_n^k := \mathop{\#}_{n} (S^1 \times S^{k-1}).\]
Since $X_n^2 = \Sigma_{n,0}$, the manifold 
$X_n^k$ is a natural generalization of a closed surface. 

Suppose $k \ge 3$. Then $\pi_1 (X_n^k - \mathrm{Int}\,D^k) \cong 
\pi_1 X_n^k \cong F_n$, 
where $\mathrm{Int}\,D^k$ is an open $k$-ball.  
We have homomorphisms 
\[\sigma^\mathrm{acy}: \mathcal{C} (X_n^k-\mathrm{Int}\,D^k) 
\longrightarrow \Aut (\Acy_n), \qquad 
\sigma^\mathrm{acy}: \mathcal{C} (X_n^k) 
\longrightarrow \Out (\Acy_n)\]
and similarly for $\mathcal{H} (X_n^k-\mathrm{Int}\,D^k)$ and 
$\mathcal{H} (X_n^k)$. 
Consider the composition 
\[\widetilde{r}_\mathfrak{q_2}:
\Aut (\Acy_n) 
\xrightarrow{r_\mathfrak{q_2}}
\mathrm{GL}(n,\mathcal{K}_{H_1}) 
\xrightarrow{\det} 
\mathcal{K}_{H_1}^\times 
\longrightarrow 
\mathcal{K}_{H_1}^\times/(\pm H_1 \cdot A') \cong \Z^\infty,\]
where $A':=\{f^{-1} \cdot \varphi (f) \mid 
f \in \mathcal{K}_{H_1}^\times, \ 
\varphi \in \mathrm{GL} (n,\Z)\}$. The map 
$\widetilde{r}_\mathfrak{q_2}$ 
is a homomorphism for the same reason 
mentioned in the previous subsection. 
\begin{theorem}[{\cite{sakasai10}}]\label{thm:acyMag}
For any $k \ge 3$ and $n \ge 2$, we have:
\begin{itemize}
\item[$(1)$] $\sigma^{\mathrm{acy}}:
\mathcal{H} (X_n^k -\mathrm{Int}\, D^k) \to \Aut (\Acy_n)$ and 
$\sigma^{\mathrm{acy}}: \mathcal{H} (X_n^k) \to \Out (\Acy_n)$ 
are surjective. 

\item[$(2)$] The image of $\widetilde{r}_\mathfrak{q_2}$ 
is an infinitely generated subgroup of $\Z^\infty$. In particular, 
$H_1 (\Aut (\Acy_n))$ and 
$H_1 (\mathcal{H} (X_n^k-\mathrm{Int}\,D^k))$ have 
infinite rank. 

\item[$(3)$] $\widetilde{r}_\mathfrak{q_2}$ factors 
through $\Out (\Acy_n)$, so that 
$H_1 (\Out (\Acy_n))$ and 
$H_1 (\mathcal{H} (X_n^k))$ have infinite rank. 
\end{itemize}
\end{theorem}
\begin{proof}[Sketch of Proof]
$(1)$ follows from a construction similar to the 
one used in the proof of Theorem \ref{thm:autsp}. To show $(2)$, 
consider $2$-connected homomorphisms $f_m: F_n \to F_n$ defined by 
\[f_m(\gamma_1) = 
(\gamma_1 \gamma_2^{-1} \gamma_1^{-1} \gamma_2^{-1})^m 
\gamma_1 \gamma_2^{2m}, \qquad
f_m(\gamma_i) = \gamma_i \ (2 \le i \le n), \] 
\noindent
which in turn give automorphisms of $\Acy_n$. We can easily check that 
\[\widetilde{r}_\mathfrak{q_2} (f_m) = 
1 -\gamma_2 +\gamma_2^2-\gamma_2^3+ \cdots 
+ \gamma_2^{2m}.\]
Then $(2)$ follows from the irreducibility of these polynomials 
when $2m+1$ is prime by a well-known fact on the cyclotomic 
polynomials. $(3)$ can be easily checked. 
\end{proof}
\begin{remark}
The statements in Theorem \ref{thm:acyMag} 
do not hold for $k=2$ by the symplecticity of 
the Magnus representation $r_\mathfrak{q_2}$ 
as seen in the previous subsection. 
When $k=3$, the theorem can be seen as a partial generalization 
of a theorem of Laudenbach \cite[Theorem 4.3]{la} stating that 
there exists an exact sequence 
\[1 \longrightarrow (\Z/2\Z)^n \longrightarrow 
\mathcal{M} (X_n^3) \longrightarrow \Aut (F_n) 
\longrightarrow 1,\] 
where the $i$-th summand of $(\Z/2\Z)^n$ corresponds to 
the rotation of $S^2$ in 
the $i$-th factor of $X_n^3$ by using $\pi_1 (SO(3)) \cong \Z/2 \Z$.
\end{remark}

In contrast with the case of surfaces, 
the homomorphism $\mathcal{M} (X) \to \mathcal{C} (X)$ 
is {\it not} necessarily injective for a general manifold $X$. 
In fact, if $[\varphi] \in \Ker(\mathcal{M} (X) \to \mathcal{C} (X))$, 
the definition of the homomorphism only says that 
$\varphi$ is a {\it pseudo isotopy} over $X$, for which 
we refer to Cerf \cite{cerf} and Hatcher-Wagoner \cite{hw}. 
Note also that we can argue about homology cylinders 
in other categories such as piecewise linear and continuous, 
which would bring us to further different and interesting phenomena.



\frenchspacing
\begin{thebibliography}{1}

\bibitem{abd} M.~Abdulrahim, 
Complex specializations of the reduced Gassner 
representation of the pure braid group, 
\textit{Proc.\ Amer.\ Math.\ Soc.} 125 (1997), 1617--1624. 

\bibitem{abmp} J.~E.~Andersen, A.~Bene, J.~B.~Meilhan, R.~Penner, 
Finite type invariants and fatgraphs, 
\textit{Adv.\ Math.} 225 (2010), 2117--2161. 

\bibitem{abp} J.~E.~Andersen, A.~Bene, R.~Penner, 
Groupoid extensions of mapping class representations 
for bordered surfaces, 
\textit{Topology Appl.} 156 (2009), 2713--2725. 

\bibitem{andrea} S.~Andreadakis, 
On the automorphisms of free groups and free nilpotent groups, 
\textit{Proc.\ London Math.\ Soc.} 15 (1965), 239--268. 

\bibitem{artin} E.~Artin, 
Theorie der Z\"opfe, 
\textit{Abhandlungen Hamburg} 4 (1925), 47--72. 

\bibitem{artin2} E.~Artin, 
Theory of braids, 
\textit{Ann.\ of Math.} 48 (1947), 101--126.

\bibitem{aps1} M.~F.~Atiyah, V.~K.~Patodi, I.~M.~Singer, 
Spectral asymmetry and Riemannian geometry.~I, 
\textit{Math.\ Proc.\ Camb.\ Phil.\ Soc.} 77 (1975), 43--69.

\bibitem{aps2} M.~F.~Atiyah, V.~K.~Patodi, I.~M.~Singer, 
Spectral asymmetry and Riemannian geometry.~II, 
\textit{Math.\ Proc.\ Camb.\ Phil.\ Soc.} 78 (1975), 405--432.

\bibitem{aps3} M.~F.~Atiyah, V.~K.~Patodi, I.~M.~Singer, 
Spectral asymmetry and Riemannian geometry.~III, 
\textit{Math.\ Proc.\ Camb.\ Phil.\ Soc.} 79 (1976), 71--99.

\bibitem{bachmuth} S.~Bachmuth, 
Automorphisms of free metabelian groups, 
\textit{Trans.\ Amer.\ Math.\ Soc.} 118 (1965), 93--104. 

\bibitem{bm1} S.~Bachmuth, H.~Y.~Mochizuki, 
The non-finite generation of $\Aut (G)$, 
$G$ free metabelian of rank $3$, 
\textit{Trans.\ Amer.\ Math.\ Soc.} 270 (1982), 693-700.

\bibitem{bm2} S.~Bachmuth, H.~Y.~Mochizuki, 
$\Aut (F) \to \Aut (F/F'')$ is surjective 
for free group for rank $\ge 4$, 
\textit{Trans.\ Amer.\ Math.\ Soc.} 292 (1985), 81-101.

\bibitem{birman_inverse} J.~Birman, 
An inverse function theorem for free groups, 
\textit{Proc.\ Amer.\ Math.\ Soc.} 41 (1973), 634--638.

\bibitem{bi} J.~Birman, 
\textit{Braids, Links and Mapping Class Groups}, 
Ann.\ of Math.\ Stud. 82, Princeton University Press (1974). 

\bibitem{bo} A.~Bousfield, 
\textit{Homological localization towers for 
groups and $\pi$-modules}, 
Mem.\ Amer.\ Math.\ Soc. 186 (1977). 

\bibitem{br} K.~Brown, 
\textit{Cohomology of Groups}, 
Graduate Texts in Mathematics 87, Springer-Verlag, (1982).

\bibitem{bg} R.~M.~Bryant, C.~K.~Gupta, 
Automorphism groups of free nilpotent groups, 
\textit{Arch.\ Math.} 52 (1989), 313-320. 

\bibitem{burau} W.~Burau, 
\"Uber Zopfgruppen und gleichsinnig 
verdrillte Verkettungen,  
\textit{Abh.\ Math.\ Sem.\ Univ.\ Hamburg} 11 (1935), 179--186. 

\bibitem{cerf} J.~Cerf, 
La stratification naturelle des espaces de fonctions 
deff\'erentiables r\'eelles et le th\'eor\`eme 
de la pseudo-isotopie, 
\textit{Publ.\ Math.\ Inst.\ Hautes \'Etudes Sci.} 39 (1970), 5--173.

\bibitem{cha} J.~C.~Cha,
Injectivity theorems and algebraic closures 
of groups with coefficients, 
\textit{Proc.\ London Math.\ Soc.} 96 (2008), 227-250.

\bibitem{cfk} J.~C.~Cha, S.~Friedl, T.~Kim, 
The cobordism group of homology cylinders, 
\textit{preprint} (2009), arXiv:0909.5580.

\bibitem{chein} O.~Chein, 
$IA$ automorphisms of free and free metabelian groups, 
\textit{Comm.\ Pure Appl.\ Math.} 21 (1968), 605--629. 

\bibitem{cf} T.~Church, B.~Farb, 
Infinite generation of the kernels of the Magnus 
and Burau representations, 
\textit{Algebr.\ Geom.\ Topol.} 10 (2010), 837--851. 

\bibitem{coc} T.~Cochran, 
Noncommutative knot theory, 
\textit{Algebr.\ Geom.\ Topol.} 4 (2004), 347--398. 

\bibitem{coc_har} T.~Cochran, S.~Harvey, 
Homology and derived series of groups, 
\textit{Geom.\ Topol.} 9 (2005), 2159--2191. 

\bibitem{coc_har_horn} T.~Cochran, S.~Harvey, P.~Horn, 
Higher-order signature cocycles for subgroups of 
mapping class groups and homology cylinders, 
\textit{preprint} (2010), arXiv:1003.4977. 

\bibitem{cot} T.~Cochran, K.~Orr, P.~Teichner, 
Knot concordance, Whitney towers and $L^2$-signatures, 
\textit{Ann.\ of Math.} 157 (2003), 433--519.

\bibitem{cot2} T.~Cochran, K.~Orr, P.~Teichner, 
Structure in the classical knot concordance group, 
\textit{Comment.\ Math.\ Helv.} 79 (2004), 105--123.

\bibitem{co} P.~M.~Cohn, 
\textit{Free Rings and their Relations}, 
Academic Press, New York - London (1985).

\bibitem{ct} R.~Crowell, H.~Trotter, 
A class of pretzel knots, 
\textit{Duke Math.\ J.} 30 (1963), 373--377.

\bibitem{Farb_Margalit} B.~Farb, D.~Margalit, 
\textit{A primer on mapping class groups}, 
To be published by Princeton University Press.

\bibitem{fox1} R.~H.~Fox, 
Free differential calculus, 
\textit{Ann.\ of Math.} 57 (1953), 547--560.

\bibitem{fri} S.~Friedl, 
Reidemeister torsion, the Thurston norm 
and Harvey's invariants, 
\textit{Pacific J.\ Math.} 230 (2007), 271--296.

\bibitem{fjr}
S.~Friedl, A.~Juh\'{a}sz, J.~Rasmussen, 
The decategorification of sutured Floer homology,
\textit{preprint} (2009), arXiv:0903.5287. 

\bibitem{gl}
S.~Garoufalidis, J.~Levine, 
Tree-level invariants of three-manifolds, Massey products and 
the Johnson homomorphism, 
In \textit{Graphs and patterns in mathematics and theorical physics}, 
Proc.\ Sympos.\ Pure Math. 73 (2005), 173--205.

\bibitem{gassner} B.~J.~Gassner, 
On braid groups, 
\textit{Abh.\ Math.\ Sem.\ Univ.\ Hamburg} 25 (1961), 10--22. 

\bibitem{gersten} S.~Gersten, 
A presentation for the special automorphism group of 
a free group, 
\textit{J.\ Pure Appl.\ Algebra} 33 (1984), 269--279.

\bibitem{GH} S.~Gervais, N.~Habegger, 
The topological IHX relation, pure braids, and the 
Torelli group, 
\textit{Duke Math.\ J.} 112 (2002), 265--280.

\bibitem{ghi} P.~Ghiggini, 
Knot Floer homology detects genus-one fibred knots, 
\textit{Amer.\ J.\ Math.} 130 (2008), 1151--1169. 

\bibitem{gs08}
H.~Goda, T.~Sakasai, 
Homology cylinders in knot theory, 
\textit{preprint} (2008), arXiv:0807.4034. 

\bibitem{gs09}
H.~Goda, T.~Sakasai, 
Abelian quotients of monoids of homology cylinders, 
\textit{preprint} (2009), arXiv:0905.4775, 
To appear in Geom.\ Dedicata. 

\bibitem{gs10}
H.~Goda, T.~Sakasai, 
Factorization formulas and computations of 
higher-order Alexander invariants for homologically fibered knots, 
\textit{preprint} (2010), arXiv:1004.3326, 
To appear in J.\ Knot Theory Ramifications.

\bibitem{gou} M.~Goussarov,
Finite type invariants and $n$-equivalence of $3$-manifolds,
\textit{C.\ R.\ Math.\ Acad.\ Sci.\ Paris} 329 (1999), 517--522.

\bibitem{gupta1} C.~K.~Gupta, 
Around automorphisms of relatively free groups, 
\textit{Algebra, Trends Math., Birkh\"auser} (1999), 63--74.

\bibitem{gupta_shpilrain} C.~K.~Gupta, V.~Shpilrain, 
Lifting automorphisms: a survey, 
\textit{London Math.\ Soc.\ Lecture Note Ser.} 211 (1995), 249--263. 

\bibitem{habe} N.~Habegger, 
Milnor, Johnson, and tree level perturbative invariants, 
\textit{preprint} (2000). 

\bibitem{habe_lin} N.~Habegger, X.~S.~Lin, 
The classification of links up to link-homotopy, 
\textit{J.\ Amer.\ Math.\ Soc.} 3 (1990), 389--419. 

\bibitem{habiro} K.~Habiro, 
Claspers and finite type invariants of links, 
\textit{Geom.\ Topol.} 4 (2000), 1--83.


\bibitem{habiromassuyeau} K.~Habiro, G.~Massuyeau, 
From mapping class groups to monoids of 
homology cobordisms: a survey, 
Chapter ? of \textit{Handbook of Teichm\"uller Theory Volume III} 
edited by A.~Papadopoulos (2010), ???--???. 

\bibitem{harer} J.~Harer, 
The second homology group of the mapping class group 
of an orientable surface, 
\textit{Invent.\ Math.} 72 (1983), 221--239.

\bibitem{har} S.~Harvey, 
Higher-order polynomial invariants of $3$-manifolds 
giving lower bounds for the Thurston norm, 
\textit{Topology} 44 (2005), 895--945.


\bibitem{har2} S.~Harvey, 
Monotonicity of degrees of generalized 
Alexander polynomials of groups and $3$-manifolds, 
\textit{Math.\ Proc.\ Cambridge Philos.\ Soc.} 
140 (2006), 431--450.

\bibitem{har3} S.~Harvey, 
Homology cobordism invariants and the Cochran-Orr-Teichner 
filtration of the link concordance group, 
\textit{Geom.\ Topol.} 12 (2008), 387--430. 

\bibitem{hw} A.~Hatcher, J.~Wagoner, 
\textit{Pseudo-isotopies of compact manifolds}, 
Ast\'erisque, 6, Soc.\ Math.\ France, Paris, (1973).

\bibitem{he} A.~Heap, 
Bordism invariants of the mapping class group, 
\textit{Topology} 45 (2006), 851--886.

\bibitem{hempel} J.~Hempel, 
\textit{$3$-Manifolds}. 
Ann.\ of Math.\ Studies, 86, 
Princeton University Press (1976).

\bibitem{hi} J.~Hillman, 
\textit{Algebraic Invariants of Links}, 
Series on Knots and Everything -- Vol.\ 32, 
World Scientific Press (2002).

\bibitem{io} K.~Igusa, K.~Orr, 
Links, pictures and the homology of nilpotent groups, 
\textit{Topology} 40 (2001), 1125--1166.

\bibitem{ivanov} N.~V.~Ivanov, 
\textit{Mapping class groups}, 
Handbook of geometric topology, 
North-Holland, Amsterdam (2002), 523--633.

\bibitem{jo4} D.~Johnson, 
A survey of the Torelli group, 
\textit{Contemp.\ Math.} 20 (1983), 165--179.

\bibitem{juhasz} A.~Juh\'asz, 
Holomorphic discs and sutured manifolds, 
\textit{Algebr.\ Geom.\ Topol.} 6 (2006), 1429--1457. 

\bibitem{juhasz2} A.~Juh\'asz, 
Floer homology and surface decompositions, 
\textit{Geom.\ Topol.} 12 (2008), 299--350.

\bibitem{kawazumi_Magnus} N.~Kawazumi, 
Cohomological aspects of Magnus expansions, 
\textit{preprint} (2005), arXiv:0505497.

\bibitem{kawazumi} N.~Kawazumi, 
Canonical $2$-forms on the moduli space 
of Riemann surfaces, 
Chapter 6 of \textit{Handbook of Teichm\"uller Theory Volume II} 
edited by A.~Papadopoulos (2009), 217--237. 

\bibitem{kmt} T.~Kitano, T.~Morifuji, M.~Takasawa, 
$L^2$-torsion invariants of a surface bundle over $S^1$, 
\textit{J.\ Math.\ Soc.\ Japan} 56 (2004), 503--518.

\bibitem{klw}
P.~Kirk, C.~Livingston, Z.~Wang, 
The Gassner representation for string links, 
\textit{Commun.\ Contemp.\ Math.} 3 (2001), 87--136. 

\bibitem{ks} M.~Korkmaz, A.~Stipsicz, 
The second homology groups of mapping class groups 
of oriented surfaces, 
\textit{Math.\ Proc.\ Cambridge Philos.\ Soc.} 134 (2003), 479--489. 

\bibitem{la} F.~Laudenbach, 
\textit{Topologie de la dimension trois: 
homotopie et isotopie}, Ast\'erisque 12, 
Soci\'et\'e Math\'ematique de France, Paris (1974).

\bibitem{ld0} J.~Y.~Le\ Dimet, 
\textit{Cobordisme d'enlacements de disques}, 
M\'em.\ Soc.\ Math.\ France 32 (1988). 

\bibitem{ld} J.~Y.~Le\ Dimet, 
Enlacements d'intervalles et repr\'esentation de Gassner, 
\textit{Comment.\ Math.\ Helv.} 67 (1992), 306--315.

\bibitem{le1} J.~Levine, 
Link concordance and algebraic closure, II, 
\textit{Invent.\ Math.} 96 (1989), 571--592.

\bibitem{le2} J.~Levine, 
Algebraic closure of groups, 
\textit{Contemp.\ Math.} 109 (1990), 99--105.

\bibitem{le4} J.~Levine, 
Link invariants via the eta invariant, 
\textit{Comment.\ Math.\ Helv.} 69 (1994), 82--119.

\bibitem{le3} J.~Levine, 
Pure braids, a new subgroup of the mapping class group 
and finite-type invariants, 
In \textit{Tel Aviv Topology Conference: Rothenberg Festschrift, 
Contemporary Mathematics} 231 (1999), 137--157. 

\bibitem{levine} J.~Levine, 
Homology cylinders: 
an enlargement of the mapping class group, 
\textit{Algebr.\ Geom.\ Topol.} 1 (2001), 243--270.

\bibitem{magnus2} W.~Magnus, 
\"Uber $n$-dimensionale Gittertransformationen, 
\textit{Acta Math.} 64 (1935), 353-367.

\bibitem{magnus} W.~Magnus, 
On a theorem of Marshall Hall, 
\textit{Ann.\ of Math.} 40 (1939), 764--768.

\bibitem{magnus-peluso} W.~Magnus, A.~Peluso, 
On a theorem of V.~I.~Arnol'd, 
\textit{Comm.\ Pure Appl.\ Math.} 22 (1969), 683--692. 

\bibitem{mm} G.~Massuyeau, J.-B.~Meilhan, 
Characterization of $Y_2$-equivalence for homology cylinders, 
\textit{J.\ Knot Theory Ramifications} 12 (2003), 493--522.

\bibitem{me} W.~Meyer, 
\textit{Die Signatur von lokalen Koeffizientensystemen 
und Faserb\"undeln}, 
Bonner Math.\ Schriften. 53 (1972).

\bibitem{milnor2} J.~Milnor, 
A duality theorem for Reidemeister torsion, 
\textit{Ann.\ of Math.} 76 (1962), 137--147. 

\bibitem{milnor} J.~Milnor, 
Whitehead torsion, 
\textit{Bull.\ Amer.\ Math.\ Soc.} 72 (1966), 358--426.

\bibitem{mo1} S.~Morita, 
Characteristic classes of surface bundles, 
\textit{Invent.\ Math.} 90 (1987), 551--577.

\bibitem{morita_jac1} S.~Morita, 
Families of Jacobian manifolds and characteristic classes 
of surface bundles I, 
\textit{Ann.\ Inst.\ Fourier} 39 (1989), 777--810.

\bibitem{morita_jac2} S.~Morita, 
Families of Jacobian manifolds and characteristic classes 
of surface bundles. II, 
\textit{Math.\ Proc.\ Camb.\ Phil.\ Soc.} 105 (1989), 79--101.

\bibitem{mo9} S.~Morita, 
Casson's invariant for homology $3$-spheres and characteristic 
classes of surface bundles I, 
\textit{Topology}, 28 (1989), 305--323.

\bibitem{mo} S.~Morita, 
Abelian quotients of subgroups of the mapping class 
group of surfaces, 
\textit{Duke Math.\ J.} 70 (1993), 699--726.

\bibitem{morita_beyond} S.~Morita, 
Cohomological structure of the mapping class group 
and beyond, In \textit{Problems on mapping class groups 
and related topics}, edited by B.~Farb, 
Proc.\ Sympos.\ Pure Math. 74 (2006), 329--354.

\bibitem{morita_survey} S.~Morita, 
Introduction to mapping class groups of surfaces and 
related groups, Chapter 7 of 
{\it Handbook of Teichm\"uller Theory Volume I} edited by 
A.~Papadopoulos (2007), 353--386.

\bibitem{morita_GD} S.~Morita, 
Symplectic automorphism groups of nilpotent 
quotients of fundamental groups of surfaces, 
In \textit{Groups of diffeomorphisms}, 
Adv.\ Stud.\ Pure Math. 52 (2008), 443--468.

\bibitem{myers} R.~Myers, 
Homology cobordisms, link concordances, 
and hyperbolic 3-manifolds, 
\textit{Trans.\ Amer.\ Math.\ Soc.} 278 (1983), 271--288.

\bibitem{ni0} Y.~Ni, 
Sutured Heegaard diagrams for knots, 
\textit{Algebr.\ Geom.\ Topol.} 6 (2006), 513--537.

\bibitem{ni1} Y.~Ni, 
Knot Floer homology detects fibred knots, 
\textit{Invent. Math.} 170 (2007), 577--608.

\bibitem{ni} J.~Nielsen, 
Die Isomorphismengruppe der freien Gruppen, 
\textit{Math.\ Ann.} 91 (1924), 169--209.

\bibitem{oda} T.~Oda, 
A lower bound for the graded modules associated with the 
relative weight filtration on the Teichmu\"ller group, 
\textit{preprint}. 

\bibitem{orr} K.~Orr, 
Homotopy invariants of links, 
\textit{Invent.\ Math.} 95 (1989) 379--394.

\bibitem{papa} C.~D.~Papakyriakopoulos, 
Planar regular coverings of orientable closed surfaces, 
\textit{Ann.\ of Math.\ Stud.} 84, Princeton Univ.\ Press (1975), 261--292

\bibitem{paris} L.~Paris, 
Braid groups and Artin groups, 
Chapter 11 of 
{\it Handbook of Teichm\"uller Theory Volume II} edited by 
A.~Papadopoulos (2009), 389--451. 

\bibitem{pa} D.~Passman, 
\textit{The Algebraic Structure of Group Rings}, 
John Wiley and Sons (1977).

\bibitem{penner} R.~Penner, 
The decorated Teichm\"uller space of punctured surfaces, 
\textit{Comm.\ Math.\ Phys.} 113 (1987), 299--339.

\bibitem{perron} B.~Perron, 
A homotopic intersection theory on surfaces: 
applications to mapping class group and braids, 
\textit{Enseign.\ Math.} 52 (2006), 159--186. 

\bibitem{ros} J.~Rosenberg, 
\textit{Algebraic K-theory and its applications}, 
Graduate Texts in Mathematics 147, Springer-Verlag (1994).

\bibitem{sa2} T.~Sakasai, 
Homology cylinders and the acyclic closure of a free group, 
\textit{Algebr.\ Geom.\ Topol.} 6 (2006), 603--631.

\bibitem{sa3} T.~Sakasai, 
The symplecticity of the Magnus representation for 
homology cobordisms of surfaces, 
\textit{Bull.\ Austral.\ Math.\ Soc.} 76 (2007), 421--431. 

\bibitem{sakasai08} T.~Sakasai, 
The Magnus representation and higher-order 
Alexander invariants for homology cobordisms of surfaces, 
\textit{Algebr.\ Geom.\ Topol.} 8 (2008), 803--848.

\bibitem{sakasai10} T.~Sakasai, 
The Magnus representation and homology cobordism 
groups of homology cylinders, 
\textit{in preparation}. 

\bibitem{satoh_twisted_H} T.~Satoh, 
Twisted first homology groups of the automorphism 
group of a free group, 
\textit{J.\ Pure Appl.\ Algebra} 204 (2006), 334--348.

\bibitem{satoh_magnus} T.~Satoh, 
The kernel of the Magnus representation of 
the automorphism group of a free group is not finitely generated, 
\textit{preprint} (2009), arXiv:0910.0386.

\bibitem{shpilrain} V.~Shpilrain, 
Automorphisms of $F/R'$ groups, 
\textit{Internat.\ J.\ Algebra Comput.} 1 (1991), 177--184.

\bibitem{st} J.~Stallings, 
Homology and central series of groups, 
\textit{J.\ Algebra} 2 (1965), 170--181.

\bibitem{strebel} R.~Strebel, 
Homological methods applied to the derived series of groups, 
\textit{Comment.\ Math.\ Helv.} 49 (1974), 302--332.

\bibitem{su} M.~Suzuki, 
The Magnus representation of the Torelli 
group $\mathcal{I}_{g,1}$ is not faithful for $g \ge 2$, 
\textit{Proc.\ Amer.\ Math.\ Soc.} 130 (2002), 909--914. 

\bibitem{suzuki_irred} M.~Suzuki, 
Irreducible decomposition of the Magnus representation 
of the Torelli group, 
\textit{Bull.\ Austral.\ Math.\ Soc.} 67 (2003), 1--14. 

\bibitem{suz} M.~Suzuki, 
Geometric interpretation of the Magnus representation of 
the mapping class group, 
\textit{Kobe J.\ Math.} 22 (2005), 39--47.

\bibitem{suz2} M.~Suzuki, 
On the kernel of the Magnus representation of the Torelli group, 
\textit{Proc.\ Amer.\ Math.\ Soc.} 133 (2005), 1865--1872. 

\bibitem{turaev} V.~G.~Turaev, 
Intersections of loops in two-dimensional manifolds, 
\textit{Mat.\ Sb.} 106(148) (1978), 566--588.

\bibitem{turaev_knot} V.~G.~Turaev, 
Reidemeister torsion in knot theory, 
\textit{Uspekhi Mat.\ Nauk} 41 (1986), 97--147. 
English translation: \textit{Russian Math. Surveys} 41 (1986), 119--182.

\bibitem{tu2} V.~Turaev, 
\textit{Introduction to combinatorial torsions}, Lectures 
Math.\ ETH Z\"urich, Birkh\"auser (2001).

\bibitem{wa} C.~T.~C.~Wall, 
Non-additivity of the signature, 
\textit{Invent.\ Math.} 7 (1969), 269--274.
\end{thebibliography}
\end{document}